\documentclass[10pt]{article}
\usepackage[T1]{fontenc}
\usepackage[margin=1in]{geometry}
\usepackage[utf8]{inputenc}

\usepackage{microtype}
\usepackage{graphicx}
\usepackage{subfig}
\usepackage{caption}
\usepackage{booktabs} %

\usepackage{amsmath}
\usepackage{amsmath,amsfonts,amssymb}
\usepackage[colorlinks=true,linkcolor=blue,citecolor=red,hyperfootnotes=false]{hyperref}
\usepackage[capitalise]{cleveref}

\title{On the Sample Complexity of Decentralized Linear Quadratic Regulator with Partially Nested Information Structure}

\author{%
Lintao Ye%
\thanks{Department of Electrical Engineering at the University of Notre Dame, Notre Dame, IN, USA; \texttt{lye2@nd.edu,vgupta2@nd.edu}.}
\and
Zhu Hao%
\thanks{Department of Electrical and Computer Engineering at the University of Texas at Austin, USA; \texttt{haozhu@utexas.edu}.}
\and
Vijay Gupta%
\footnotemark[1]}

\RequirePackage{mathtools}
\RequirePackage{dsfont}
\RequirePackage{delimset}

\usepackage{amsthm}
\usepackage{amssymb}
\usepackage{thmtools}

\DeclareMathOperator*{\argmin}{arg\,min}

\newtheorem{theorem}{Theorem}
\newtheorem{corollary}{Corollary}
\newtheorem{observation}{Observation}
\newtheorem{lemma}{Lemma}
\newtheorem{remark}{Remark}

\newtheorem{proposition}{Proposition}
\newtheorem{example}{Example}
\newtheorem{assumption}{Assumption}

\newcommand{\BS}{\mathbb{S}}
\newcommand{\M}{\mathcal{M}}
\newcommand{\T}{\mathcal{T}}
\newcommand{\CL}{\mathcal{L}}
\newcommand{\A}{\mathcal{A}}
\newcommand{\CN}{\mathcal{N}}

\newcommand{\CO}{\mathcal{O}}
\newcommand{\R}{\mathbb{R}}
\newcommand{\CR}{\mathcal{R}}
\newcommand{\CS}{\mathcal{S}}
\newcommand{\G}{\mathcal{G}}
\newcommand{\Z}{\mathbb{Z}}
\newcommand{\E}{\mathbb{E}}
\newcommand{\F}{\mathcal{F}}
\newcommand{\CE}{\mathcal{E}}
\newcommand{\CH}{\mathcal{H}}
\newcommand{\Prob}{\mathbb{P}}
\newcommand{\CP}{\mathcal{P}}
\newcommand{\U}{\mathcal{U}}
\newcommand{\V}{\mathcal{V}}

\newcommand{\I}{\mathcal{I}}

\newcommand{\tr}{\text{Tr}}

\usepackage{enumitem}
\usepackage[numbers]{natbib}
\usepackage{crossreftools}
\usepackage{algorithm}
\usepackage[noend]{algpseudocode}
\algnewcommand{\LineComment}[1]{\State \(\triangleright\) #1}

\algnewcommand{\IfThenElse}[3]{%
  \State \algorithmicif\ #1\ \algorithmicthen\ #2\ \algorithmicelse\ #3}

\begin{document}

\maketitle

\begin{abstract}
We study the problem of control policy design for decentralized state-feedback linear quadratic control with a partially nested information structure, when the system model is unknown. We propose a model-based learning solution, which consists of two steps. First, we estimate the unknown system model from a single system trajectory of finite length, using least squares estimation. Next, based on the estimated system model, we design a decentralized control policy that satisfies the desired information structure. We show that the suboptimality gap between our control policy and the optimal decentralized control policy (designed using accurate knowledge of the system model) scales linearly with the estimation error of the system model. Using this result, we provide an end-to-end sample complexity result for learning decentralized controllers for a linear quadratic control problem with a partially nested information structure.
\end{abstract}

\section{Introduction}\label{sec:intro}
In large-scale control systems, the control policy is often required to be decentralized, where different controllers may only use partial state information, when designing their local control policies. For example, a given controller may only receive a subset of the global state measurements (e.g., \cite{shah2013cal}), and there may be a delay in receiving the measurements (e.g., \cite{lamperski2012dynamic}). In general, finding a globally optimal control policy under information constraints is NP-hard, even if the system model is known at the controllers \cite{witsenhausen1968counterexample,papadimitriou1986intractable,blondel2000survey}. This has led to a large literature on identifying tractable subclasses of the problem. For instance, if the information structure describing the decentralized control problem is partially nested \cite{ho1972team}, the optimal solution to the state-feedback linear quadratic control problem can be solved efficiently using dynamic programming \cite{lamperski2015optimal}. Other conditions, such as quadratic invariance \cite{rotkowitz2005characterization,rotkowitz2011nearest}, have also been identified as tractable subclasses of the problem. 

However, the classical work in this field assumes the knowledge of the system model at the controllers. In this work, we are interested in the situation when the system model is not known a priori \cite{hou2013model}. In such a case, the existing algorithms do not apply. Moreover, it is not clear whether subclasses such as problems with partially nested information patterns or where quadratic invariance is satisfied are any more tractable than the general decentralized control problem in this case. 

In this paper, we consider a decentralized infinite-horizon state-feedback Linear Quadratic Regulator (LQR) control problem with a partially nested information structure \cite{shah2013cal,lamperski2015optimal} and assume that the controllers do not have access to the system model. We use a model-based learning approach, where we first identify the system model, and then use it to design a decentralized control policy that satisfies the prescribed information constraints. 

\subsection*{Related Work}
Solving optimal control problems without prior system model knowledge has receive much attention recently. One of the most studied problems is the centralized LQR problem. For this problem, two broad classes of methods have been studied, i.e., model based learning \cite{abbasi2011regret,mania2019certainty,dean2020sample}, and model-free learning \cite{fazel2018global,zhang2020policy,malik2020derivative,gravell2020learning}. In the model-based learning approach, a system model is first estimated from observed system trajectories using some system identification method. A control policy can then be obtained based on the estimated system model. In the model-free learning approach, the objective function in the LQR problem is first viewed as a function of the control policies. Based on zeroth-order optimization methods (e.g., \cite{ghadimi2013stochastic,nesterov2017random}), the optimal solution can then be obtained using gradient descent, where the gradient of the objective function is estimated from the data samples from system trajectories. Moreover, the model-based learning approach has also been studied for the centralized linear quadratic Gaussian control problem \cite{zheng2021sample}. In general, compared to model-free learning, model-based learning tends to require less data samples in order to achieve a policy of equivalent performance \cite{tu2019gap}.  

Most of the previous works on model-based learning for centralized LQR build on recent advances in non-asymptotic analyses for system identification of linear dynamical systems with full state observations (e.g., \cite{faradonbeh2018finite,simchowitz2018learning,sarkar2019near}). Such non-asymptotic analyses (i.e., sample complexity results) relate the estimation error of the system matrices to the number of samples used for system identification. In particular, it was shown in \cite{simchowitz2018learning} that when using a single system trajectory, the least squares approach for system identification achieves the optimal sample complexity up to logarithmic factors. In this paper, we utilize a similar least squares approach for estimating the system matrices from a single system trajectory. Although the system matrices in our problem are structured, as dictated by the interconnections among the subsystems, we leverage the results in \cite{abbasi2011improved,cohen2019learning} to provide a non-asymptotic analysis of the resulting estimation error.

There are few results on solving decentralized linear quadratic control problems with information constraints, when the system model is unknown. In \cite{furieri2020learning}, the authors studied a decentralized output-feedback linear quadratic control problem, under the assumption that the quadratic invariance condition is satisfied. The authors proposed a model-free approach and provided a sample complexity analysis. They focused on a finite-horizon setting, since gradient-based optimization methods may not converge to the optimal controller for infinite-horizon decentralized linear quadratic control problems with information constraints, even when the system model is known \cite{feng2019exponential,bu2019lqr}. In \cite{li2019distributed}, the authors proposed a consensus-based model-free learning algorithm for multi-agent decentralized LQR over an infinite horizon, where each agent (i.e., controller) has access to a subset of the global state without delay. They showed that their algorithm converges to a control policy that is a stationary point of the objective function in the LQR problem. In \cite{fattahi2020efficient}, the authors studied model-based learning for LQR with subspace constraints on the closed-loop responses. However, those constraints may not lead to controllers that satisfy the information constraints that we consider in this paper (e.g., \cite{zheng2020equivalence}).

There is also a line of research on online adaptive control for centralized LQR with unknown system models, using either model-based learning \cite{abbasi2011regret,dean2018regret,cohen2019learning}, or model-free learning \cite{abbasi2019politex,cassel2021online}. The goal there is to adaptively design a control policy in an online manner when new data samples from the system trajectory become available, and bound the corresponding regret.

\subsection*{Contributions}
We propose a two-step model-based approach to solving the problem of learning decentralized LQR with a partially nested information structure. Here, we summarize our contributions and technical challenges in the paper. 
\begin{itemize}
\item In Section~\ref{sec:system ID}, we provide a \textit{sample complexity} result for estimating the system model from a single system trajectory using a least squares approach. Despite the existence of a sparsity pattern in the system model considered in our problem, we adapt the analyses in \cite{cohen2019learning,cassel2020logarithmic} for least squares estimation of general linear system models (without any sparsity pattern) to our setting, and show that such a system identification method for general system models suffices for our ensuing analyses.

\item In Section~\ref{sec:control design}, based on the estimated system model, we design a novel \textit{decentralized control policy} that satisfies the given information structure. Our control policy is inspired by \cite{lamperski2015optimal}, which developed the optimal controller for the decentralized LQR problem with a partially nested information structure and known system model. The optimal controller therein depends on some internal states, each of which evolves according to an auxiliary linear system (characterized by the actual model of the original system with a disturbance term from the original system) and correlates with other internal states. Accordingly, this complicated form of the internal states makes it challenging to extend the design in \cite{lamperski2015optimal} to the case when the system model is unknown. To tackle this, we capitalize on the observation that  the optimal controller proposed in \cite{lamperski2015optimal} can be viewed as a disturbance-feedback control policy that maps the history of past disturbances (affecting the original system) to the current control input. Thanks to this viewpoint, we put forth a control policy that uses the aforementioned estimated system model and maps the {\it estimates} of past disturbances to the current control input via some {\it estimated} internal states. Particularly, the estimates of disturbances are obtained using the estimated system model and the state information of original system, and each of the estimated internal states evolves according to a linear system characterized by the estimated system model and the estimated disturbances. More importantly, we show that the proposed control policy can be implemented in a decentralized manner that satisfies the prescribed information structure, which requires a careful investigation of the structure of our problem.

\item In Section~\ref{sec:perturb bound on cost}, we characterize the {\it performance guarantee (i.e., suboptimality)} of the control policy proposed in Section~\ref{sec:control design}. As we discussed above, our control policy requires obtaining estimates of the past disturbances and maintaining the estimated internal states. When we compare the performance of our control policy to that of the optimal decentralized control policy in \cite{lamperski2015optimal}, both the estimates of the past disturbances and the estimated internal states contribute to the suboptimality of our control policy, which creates the major technical challenge in our analyses. We overcome this challenge by carefully investigating the structure of the proposed control policy, and we show that the suboptimality gap between our control policy and the optimal decentralized control policy (designed based on accurate knowledge of the system model) provided in \cite{lamperski2015optimal} can be decomposed into two terms, both of which scale {\it linearly} with the estimation error of the system model.

\item In Section~\ref{sec:sample complexity}, we combine the above results together and provide an {\it end-to-end sample complexity} result for learning decentralized LQR with a partially nested information structure. Surprisingly, despite the existence of the information constraints and the fact that the optimal controller is a linear dynamic controller, our sample complexity result matches with that of learning centralized LQR without any information constraints \cite{dean2020sample}.
\end{itemize}

\section{Preliminaries and Problem Formulation}\label{sec:preliminaries and problem formulation}
\subsection{Notation and Terminology}
The sets of integers and real numbers are denoted as $\mathbb{Z}$ and $\mathbb{R}$, respectively. The set of integers (resp., real numbers) that are greater than or equal to $a\in\mathbb{R}$ is denoted as $\mathbb{Z}_{\ge a}$ (resp., $\R_{\ge a}$). For a real number $a$, let $\lceil a \rceil$ be the smallest integer that is greater than or equal to $a$. The space of $m$-dimensional real vectors is denoted by $\mathbb{R}^{m}$, and the space of $m\times n$ real matrices is denoted by $\mathbb{R}^{m\times n}$. For a matrix $P\in\R^{n\times n}$, let $P^{\top}$, $\tr(P)$, and $\{\sigma_i(P):i\in\{1,\dots,n\}\}$ be its transpose, trace, and set of singular values, respectively. Without loss of generality, let the singular values of $P$ be ordered as $\sigma_1(P)\ge\cdots\ge\sigma_n(P)$. Let $\norm{\cdot}$ denote the $\ell_2$ norm, i.e., $\norm{P}=\sigma_1(P)$ for a matrix $P\in\R^{n\times n}$, and $\norm{x}=\sqrt{x^{\top}x}$ for a vector $x\in\R^n$. Let $\norm{P}_F=\sqrt{\tr(PP^{\top})}$ denote the Frobenius norm of $P\in\R^{n\times m}$. A positive semidefinite matrix $P$ is denoted by $P\succeq0$, and $P\succeq Q$ if and only if $P-Q\succeq0$. Let $\BS_+^n$ (resp., $\BS_{++}^n$) denote the set of $n\times n$ positive semidefinite (resp., positive definite) matrices. Let $I$ denote an identity matrix whose dimension can be inferred from the context. Given any integer $n\ge1$, we define $[n]=\{1,\dots,n\}$. The cardinality of a finite set $\mathcal{A}$ is denoted by $|\mathcal{A}|$. Let $\CN(\mu,\Sigma)$ denote a Gaussian distribution with mean $\mu\in\R^m$ and covariance $\Sigma\in\BS^m_+$.

\subsection{Solution to Decentralized LQR with Sparsity and Delay Constraints}\label{sec:dist LQR known matrices}
In this section, we sketch the method developed in \cite{lamperski2015optimal,shah2013cal}, which presents the optimal solution to a decentralized LQR problem with a {\it partially nested} information structure \cite{ho1972team}, when the system model is known a priori. First, let us consider a networked system that consists of $p\in\Z_{\ge1}$ interconnected linear-time-invariant (LTI) subsystems. Letting the state, input and disturbance of the subsystem corresponding to node $i\in[p]$ be $x_i(t)\in\R^{n_i}$, $u_i(t)\in\R^{m_i}$, and $w_i(t)$, respectively, the  subsystem corresponding to node $i$ is given by
\begin{equation}
\label{eqn:system for node i}
 x_i(t+1) = \Big(\!\sum_{j\in\CN_i}\!A_{ij}x_j(t)+B_{ij}u_j(t)\Big)+w_i(t)\ \forall i\in\V,
\end{equation}
where $\CN_i\subseteq[p]$ is the set of subsystems whose states and inputs directly affect the state of subsystem $j$, $A_{ij}\in\R^{n_i\times n_i}$, $B_{ij}\in\R^{n_i\times m_i}$, and $w_i(t)\in\R^{n_i}$ is a white Gaussian noise process with $w_i(t)\sim\mathcal{N}(0,\sigma_w^2I)$ for all $t\in\Z_{\ge0}$, where $\sigma_w\in\R_{>0}$.\footnote{The analysis can be extended to the case when $w_i(t)$ is assumed to be a zero-mean white Gaussian noise process with covariance $W\in\BS_{++}^{n_i}$. In that case, our analysis will depend on $\max_{i\in\V}\sigma_1(W_i)$ and $\min_{i\in\V}\sigma_n(W_i)$.} For simplicity, we assume throughout this paper that $n_i\ge m_i$ for all $i\in\V$. We can also write Eq.~\eqref{eqn:system for node i} as
\begin{equation}
\label{eqn:dynamics for x_i(t)}
x_i(t+1) = A_ix_{\CN_i}(t) + B_iu_{\CN_i}(t) + w_i(t)\quad\forall i\in\V,
\end{equation}
where $A_i\triangleq\begin{bmatrix}A_{ij_1} & \cdots A_{ij_{|\CN_i|}}\end{bmatrix}$, $B_i\triangleq\begin{bmatrix}B_{ij_1} & \cdots B_{ij_{|\CN_i|}}\end{bmatrix}$, $x_{\CN_i}(t)\triangleq\begin{bmatrix}x_{j_1}(t) & \cdots x_{j_{|\CN_i|}}(t)\end{bmatrix}^{\top}$, and $u_{\CN_i}(t)\triangleq\begin{bmatrix}u_{j_1}(t) & \cdots u_{j_{|\CN_i|}}(t)\end{bmatrix}^{\top}$, with $\CN_i=\{j_1,\dots,j_{|\CN_i|}\}$. Further letting $n=\sum_{i\in\V}n_i$ and $m=\sum_{i\in\V}m_i$, and defining $x(t)=\begin{bmatrix}x_1(t)^{\top} & \cdots & x_p(t)^{\top}\end{bmatrix}^{\top}$, $u(t)=\begin{bmatrix}u_1(t)^{\top} & \cdots & u_p(t)^{\top}\end{bmatrix}^{\top}$ and $w(t)=\begin{bmatrix}w_1(t)^{\top} & \cdots & w_p(t)^{\top}\end{bmatrix}^{\top}$, we can compactly write Eq.~\eqref{eqn:system for node i} into the following matrix form:
\begin{equation}\label{eqn:overall system}
    x(t+1)=Ax(t) + Bu(t) + w(t),
\end{equation}
where the $(i,j)$th block of $A\in\R^{n\times n}$ (resp., $B\in\R^{n\times m}$), i.e., $A_{ij}$ (resp., $B_{ij}$) satisfies $A_{ij}=0$ (resp., $B_{ij}=0$) if $j\notin\CN_i$. We assume that $w_i(t_1)$ and $w_j(t_2)$ are independent for all $i,j\in\V$ with $i\neq j$ and for all $t_1,t_2\in\Z_{\ge0}$. In other words, $w(t)$ is a white Gaussian noise process with $w(t)\sim\CN(0,\sigma_w^2I)$ for all $t\in\Z_{\ge0}$. For simplicity, we assume that $x(0)=0$ throughout this paper.\footnote{The analysis can be extended to the case when $x(0)$ is given by a zero-mean Gaussian distribution, as one may view $x(0)$ as $w(-1)$.} 

Next, we use a directed graph $\G(\V,\A)$ with $\V=[p]$ to characterize the information flow among the subsystems in $[p]$ due to communication constraints on the subsystems. Each node in $\G(\V,\A)$ represents a subsystem in $[p]$, and we assume that $\G(\V,\A)$ does not have self loops. We associate any edge $(i,j)\in\A$ with a delay of either $0$ or $1$, further denoted as $i\xrightarrow[]{0}j$ or $i\xrightarrow[]{1}j$, respectively.\footnote{The framework described in this paper can also be used to handle $\G(\V,\A)$ with larger delays; see \cite{lamperski2015optimal} for a detailed discussion.} Then, we define the delay matrix corresponding to $\G(\V,\A)$ as $D\in\R^{p\times p}$ such that: (i) If $i\neq j$ and there is a directed path from $j$ to $i$ in $\G(\V,\A)$, then $D_{ij}$ is equal to the sum of delays along the directed path from node $j$ to node $i$ with the smallest accumulative delay; (ii) If $i\neq j$ and there is no directed path from $j$ to $i$ in $\G(\V,\A)$, then $D_{ij}=+\infty$; (iii) $D_{ii}=0$ for all $i\in\V$. Here, we consider the scenario where the information (e.g., state information) corresponding to subsystem $j\in\V$ can propagate to subsystem $i\in\V$ with a delay of $D_{ij}$ (in time), if and only if there exists a directed path from $j$ to $i$ with an accumulative delay of $D_{ij}$. Note that as argued in \cite{lamperski2015optimal}, we assume that there is no directed cycle with zero accumulative delay; otherwise, one can first collapse all the nodes in such a directed cycle into a single node, and equivalently consider the resulting directed graph in the framework described above. 

To proceed, we consider designing the control input $u(t)$ for the LTI system in Eq.~\eqref{eqn:overall system}. We focus on {\it state-feedback} control, i.e., we can view $u(t)$ as a policy that maps the states of the LTI system to a control input. Moreover, we require that $u(t)$ satisfy the information structure according to the directed graph $\G(\V,\A)$ and the delay matrix $D\in\R^{p\times p}$, described above. Specifically, considering any $i\in\V$ and any $t\in\Z_{\ge0}$, and noting that the controller corresponding to subsystem $i\in\V$ provides the control input $u_i(t)\in\R^{m_i}$, the state information that is available to the controller corresponding to $i\in\V$ is given by 
\begin{equation}\label{eqn:info set}
    \I_i(t) = \{x_j(k):j\in\V_i,0\le k\le t-D_{ij}\},
\end{equation}
where $\V_i\triangleq\{j\in\V:D_{ij}\neq+\infty\}$. In other words, the control policy $u_i(t)$ maps the states contained in $\I_i(t)$ to a control input. In the sequel, we also call $\I_i(t)$ the {\it information set} of controller $i\in\V$ at time $t\in\Z_{\ge0}$. Note that $\I_i(t)$ contains the states corresponding to the subsystems in $\V$ that have enough time to reach subsystem $i\in\V$ at time $t\in\Z_{\ge0}$, due to the sparsity and delay constraints described above. Now, based on the information set $\I_i(t)$, we further define $\CS(\I_i(t))$ to be the set that consists of all the policies that map the states in $\I_i(t)$ to a control input at node $i$. The goal is then to solve the following constrained optimization problem:
\begin{equation}
\label{eqn:dis LQR obj}
\begin{split}
\min_{u(0),u(1),\dots}&\lim_{T\to\infty}\E\Big[\frac{1}{T}\sum_{t=0}^{T-1}(x(t)^{\top}Qx(t)+u(t)^{\top}Ru(t))\Big]\\
s.t.\ &x(t+1)=Ax(t)+Bu(t)+w(t),\\
&u_i(t)\in\CS(\I_i(t))\quad \forall i\in\V,\forall t\in\Z_{\ge0},
\end{split}
\end{equation}
where $Q\in\BS_+^n$ and $R\in\BS_{++}^m$ are the cost matrices, and the expectation is taken with respect to $w(t)$ for all $t\in\Z_{\ge0}$. 
Throughout the paper, we always assume that the following assumption on the information propagation pattern among the subsystems in $\V$ holds (e.g., \cite{lamperski2015optimal,yu2022online}). 
\begin{assumption}
\label{ass:info structure}
For all $j\in\CN_i$, it holds that $D_{ij}\le1$, where $\CN_i$ is given in Eq.~\eqref{eqn:system for node i}.
\end{assumption}
Assumption~1 says that the state of subsystem $i\in\V$ is affected by the state and input of subsystem $j\in\V$, if and only if there is a communication link with a delay of at most $1$ from subsystem $j$ to $i$ in $\G(\V,\A)$. As shown in \cite{lamperski2015optimal}, Assumption~\ref{ass:info structure} ensures that the information structure associated with the system given in Eq.~\eqref{eqn:system for node i} is partially nested \cite{ho1972team}.  Assumption~\ref{ass:info structure} is frequently used in decentralized control problems (e.g., \cite{lamperski2015optimal,shah2013cal} and the references therein), and one can see that the assumption is satisfied in networked systems where information propagates at least as fast as dynamics. To illustrate our arguments above, we introduce Example~\ref{exp:running example}.
\vspace{0cm}
\begin{example}
\label{exp:running example}
Consider a directed graph $\G(\V,\A)$ given in Fig.~\ref{fig:directed graph}, where $\V=\{1,2,3\}$ and each directed edge is associated with a delay of $0$ or $1$. The corresponding LTI system is then given by 
\begin{equation}
\label{eqn:LTI in exp}
\begin{bmatrix}
x_1(t+1)\\x_2(t+1)\\x_3(t+1)
\end{bmatrix}=
\begin{bmatrix}
A_{11} & A_{12} & A_{13}\\
0 & A_{22} & A_{23}\\
0 & A_{32} & A_{33}
\end{bmatrix}
\begin{bmatrix}
x_1(t)\\x_2(t)\\x_3(t)
\end{bmatrix}
+\begin{bmatrix}
B_{11} & B_{12} & B_{13}\\
0 & B_{22} & B_{23}\\
0 & B_{32} & B_{33}
\end{bmatrix}
\begin{bmatrix}
u_1(t)\\u_2(t)\\u_3(t)
\end{bmatrix}
+\begin{bmatrix}
w_1(t)\\w_2(t)\\w_3(t)
\end{bmatrix}.
\end{equation}
\end{example}
\begin{figure}[htbp]
    \centering
    \includegraphics[width=0.3\linewidth]{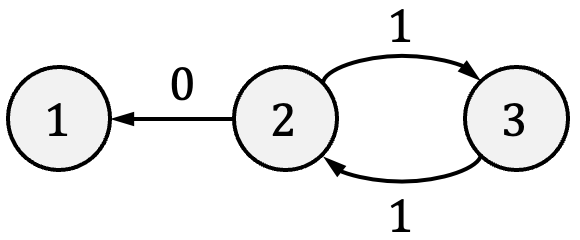} 
    \caption{{\ The directed graph of Example~\ref{exp:running example}. Node $i\in\V$ represents a subsystem with state $x_i(t)$ and edge $(i,j)\in\A$ is labelled with the information propagation delay from  $i$ to $j$.}}
\label{fig:directed graph}
\end{figure}

Now, in order to present the solution to \eqref{eqn:dis LQR obj} given in, e.g.,  \cite{lamperski2015optimal}, we need to construct an information graph $\CP(\U,\CH)$ (see \cite{lamperski2015optimal} for more details). Considering any directed graph $\G(\V,\A)$ with $\V=[p]$, and the delay matrix $D\in\R^{p\times p}$ as we described above, let us first define $s_j(k)$ to be the set of nodes in $\G(\V,\A)$ that are reachable from node $j$ within $k$ time steps, i.e., $s_j(k)=\{i\in\V:D_{ij}\le k\}$. The information graph $\CP(\U,\CH)$ is then constructed as
\begin{equation}
\label{eqn:def of info graph}
\begin{split}
\U&=\{s_j(k):k\ge0,j\in\V\},\\
\CH&=\{(s_j(k),s_j(k+1)):k\ge0, j\in\V\}.
\end{split}
\end{equation}
Thus, we see from \eqref{eqn:def of info graph} that each node $s\in\U$ corresponds to a set of nodes from $\V=[p]$ in the original directed graph $\G(\V,\A)$. Using a similar notation to that for the graph $\G(\V,\A)$, if there is an edge from $s$ to $r$ in $\CP(\U,\CH)$, we denote the edge as $s\to r$. Additionally, considering any $s_i(0)\in\U$, we write $w_i\xrightarrow[]{}s_i(0)$ to indicate the fact that the noise $w_i(t)$ is injected to node $i\in\V$ at time $t\in\Z_{\ge0}$.\footnote{Note that we have assumed that there is no directed cycle with zero accumulative delay in $\CP(\U,\CH)$. Hence, one can show that for any $s_i(0)\in\U$, $w_i$ is the only noise term such that $w_i\rightarrow s_i(0)$.} From the above construction of the information graph $\CP(\U,\CH)$, one can show that the following properties hold.
\begin{lemma}
\label{lemma:properties of info graph}
\cite[Proposition~1]{lamperski2015optimal} Given a directed graph $\G(\V,\A)$ with $\V=[p]$, the information graph $\CP(\U,\CH)$ constructed in \eqref{eqn:def of info graph} satisfies the following: (i) For every $r\in\U$, there is a unique $s\in\U$ such that $(r,s)\in\CH$, i.e., $r\xrightarrow[]{}s$; (ii) every path in $\CP(\U,\CH)$ ends at a node with a self loop; and (iii) $n\le|\U|\le p^2-p+1$.
\end{lemma}
\begin{remark}
\label{remark:tree info graph}
One can see from the construction of $\CP(\U,\CH)$ and Lemma~\ref{lemma:properties of info graph} that $\CP(\U,\CH)$ is a forest, i.e., a set of disconnected directed trees, where each directed tree in the forest is oriented to a node with a self loop in $\CP(\U,\CH)$. Specifically, $s_i(0)$ for all $i\in\V$ are the leaf nodes in $\CP(\U,\CH)$, and the nodes with self loop are root nodes in $\CP(\U,\CH)$.
\end{remark}
To illustrate the construction steps and the properties of the information graph discussed above, we again use Example~\ref{exp:running example}; the resulting information graph $\CP(\U,\CH)$ is then given in Fig.~\ref{fig:info graph}. Note that the information graph $\CP(\U,\CH)$ in Fig.~\ref{fig:info graph} contains two disconnected directed trees, one of which is an isolated node $\{1\}\in\U$ with a self loop. Also notice that $s_1(0)=\{1\}$, $s_2(0)=\{1,2\}$ and $s_3(0)=\{3\}$. In fact, we can check that the results in Lemma~\ref{lemma:properties of info graph} hold for $\CP(\U,\CH)$ in Fig.~\ref{fig:info graph}.
\begin{figure}[htbp]
    \centering
    \includegraphics[width=0.30\linewidth]{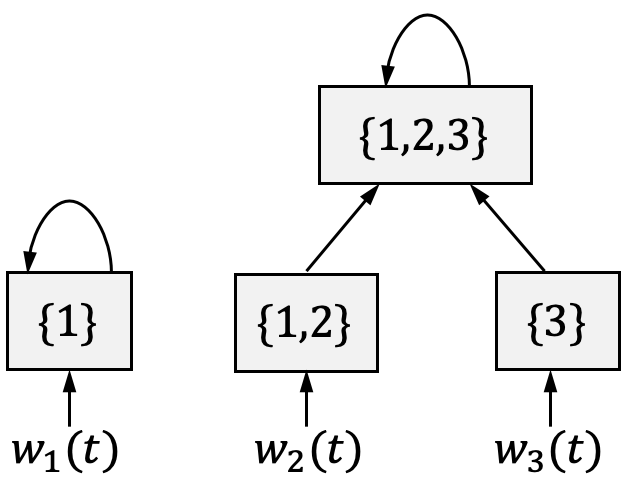} 
    \caption{The information graph of Example~\ref{exp:running example}. Each node in the information graph is a subset of the nodes in the directed graph given in Fig.~\ref{fig:directed graph}.}
\label{fig:info graph}
\end{figure}

Throughout this paper, we assume that the elements in $\V=[p]$ are ordered in an increasing manner, and that the elements in $s$ are also ordered in an increasing manner for all $s\in\U$. Now, for any $s,r\in\U$, we use $A_{sr}$ (or $A_{s,r}$) to denote the submatrix of $A$ that corresponds to the nodes of the directed graph $\G(\V,\A)$ contained in $s$ and $r$. For example, $A_{\{1\},\{1,2\}}=\begin{bmatrix}A_{11} & A_{12}\end{bmatrix}$. In the sequel, we will also use similar notations to denote submatrices of $B$, $Q$, $R$ and the identity matrix $I$. We will make the following standard assumptions (see, e.g., \cite{lamperski2015optimal}).
\begin{assumption}
\label{ass:pairs}
For any $s\in\U$ that has a self loop, the pair $(A_{ss},B_{ss})$ is stabilizable and the pair $(A_{ss},C_{ss})$ is detectable, where $Q_{ss}=C_{ss}^{\top}C_{ss}$.
\end{assumption}

Leveraging the partial nestedness of \eqref{eqn:dis LQR obj}, the authors in \cite{lamperski2015optimal} obtained the optimal solution to \eqref{eqn:dis LQR obj}, which we summarize in the following lemma.
\begin{lemma}\label{lemma:opt solution}
\cite[Corollary~4]{lamperski2015optimal} Consider the problem given in \eqref{eqn:dis LQR obj}, and let $\CP(\U,\CH)$ be the associated information graph. Suppose Assumption \ref{ass:pairs} holds. For all $r\in\U$, define matrices $P_r$ and $K_r$ recursively as
\begin{align}
K_r &= -(R_{rr}+B_{sr}^{\top} P_s B_{sr}^{\top})^{-1} B_{sr}^{\top} P_s A_{sr},\label{eqn:set of DARES K}\\
P_r &= Q_{rr}+K_r^{\top}R_{rr}K_r+(A_{sr}+B_{sr}K_r)^{\top}P_s(A_{sr}+B_{sr}K_r),\label{eqn:set of DARES P}
\end{align}
where for each $r\in\U$, $s\in\U$ is the unique node such that $r\rightarrow s$. In particular, for any $s\in\U$ that has a self loop, the matrix $P_s$ is the unique positive semidefinite solution to the Riccati equation given by Eq.~\eqref{eqn:set of DARES P} , and the matrix $A_{ss}+B_{ss}K_s$ is stable. The optimal solution to \eqref{eqn:dis LQR obj} is then given by
\begin{equation}
\label{eqn:dynamics of zeta}
\zeta_s(t+1) = \sum_{r\rightarrow s}(A_{sr}+B_{sr}K_r)\zeta_r(t) + \sum_{w_i\rightarrow s}I_{s,\{i\}}w_i(t),
\end{equation}
and
\begin{equation}
\label{eqn:exp for u^star}
u_i^{\star}(t) = \sum_{r\ni i}I_{\{i\},r}K_r\zeta_r(t),
\end{equation}
for all $t\in\Z_{\ge0}$, where $\zeta_s(t)$ is an internal state initialized with $\zeta_s(0)=\sum_{w_i\rightarrow s}I_{s,\{i\}}x_i(0)=0$ for all $s\in\U$. The corresponding optimal cost of \eqref{eqn:dis LQR obj}, denoted as $J_{\star}$, is given by
\begin{equation}
\label{eqn:opt J}
J_{\star}=\sigma_w^2\sum_{\substack{i\in\V\\ w_i\rightarrow s}}\tr\big(I_{\{i\},s}P_sI_{s,\{i\}}\big).
\end{equation}
\end{lemma}

Let us use Example~\ref{exp:running example} to illustrate the results in Lemma~\ref{lemma:opt solution}. First, considering node $\{1\}\in\U$ in the information graph $\CP(\U,\CH)$ given in Fig.~\ref{fig:info graph}, we have from Eq.~\eqref{eqn:dynamics of zeta} that 
\begin{align*}
\zeta_{1}(t+1) &= (A_{11}+B_{11}K_1)\zeta_1(t) + \sum_{w_i\to\{1\}}I_{\{1\},\{i\}}w_i(t)\\
&=(A_{11}+B_{11}K_1)\zeta_1(t) + w_1(t).
\end{align*}
Next, considering node $2\in\V$ in the directed graph $\G(\V,\A)$ given in Fig.~\ref{fig:directed graph}, we see from Eq.~\eqref{eqn:exp for u^star} and Fig.~\ref{fig:info graph} that 
\begin{equation*}
u_2^{\star}(t)=\sum_{r\ni 2}I_{\{2\},r}K_r\zeta_r(t)=I_{\{2\},\{1,2\}}K_{\{1,2\}}\zeta_{\{1,2\}}(t)+I_{\{2\},\{1,2,3\}}K_{\{1,2,3\}}\zeta_{\{1,2,3\}}(t),
\end{equation*}
where $K_r$ is given by Eq.~\eqref{eqn:set of DARES K}.

\begin{remark}
\label{remark:decentralized setting}
Obtaining the optimal policy $u^{\star}_i(t)$, for any $i\in\V$, given by Lemma~\ref{lemma:opt solution} requires global knowledge of the system matrices $A$ and $B$, the cost matrices $Q$ and $R$, and the directed graph $\G(\V,\A)$ with the associated delay matrix $D$. 
Moreover, $u^{\star}(t)$ given in Lemma~\ref{lemma:opt solution} is not a static state-feedback controller, but a linear dynamic controller based on the internal states $\zeta_r(\cdot)$ for all $r\in\U$. For any controller $i\in\V$ and for any $t\in\Z_{\ge0}$, the authors in \cite{lamperski2015optimal} proposed an algorithm to determine $\zeta_r(t)$ for all $r\in\U$ such that $i\in r$, and thus $u_i^{\star}(t)$, using only the memory maintained by the algorithm, the state information contained in the information set $\I_i(t)$ defined in Eq.~\eqref{eqn:info set}, and the global information described above.  
\end{remark}

\subsection{Problem Formulation and Summary of Results}\label{sec:problem formulation}
We now formally introduce the problem that we will study in this paper. We consider the scenario where the system matrices $A$ and $B$ are unknown. However, we assume that the directed graph $\G(\V,\A)$ and the associated delay matrix $D$ are known. Similarly to, e.g., \cite{dean2020sample,zheng2021sample}, we consider the scenario where we can first conduct experiments in order to estimate the unknown system matrices $A$ and $B$. Specifically, starting from the initial state $x(0)=0$, we evolve the system given in Eq.~\eqref{eqn:overall system} for $N\in\Z_{\ge1}$ time steps using a given control input sequence $\{u(0),u(1),\dots,u(N-1)\}$, and collect the resulting state sequence $\{x(1),x(2),\dots,x(N)\}$. Based on $\{u(0),\dots,u(N-1)\}$ and $\{x(0),\dots,x(N)\}$, we use a least squares approach to obtain estimates of the system matrices $A$ and $B$, denoted as $\hat{A}$ and $\hat{B}$, respectively. Using the obtained $\hat{A}$ and $\hat{B}$, the goal is still to solve \eqref{eqn:dis LQR obj}. Since the true system matrices $A$ and $B$ are unknown, it may no longer be possible to solve \eqref{eqn:dis LQR obj} optimally, using the methods introduced in Section~\ref{sec:dist LQR known matrices}. Thus, we aim to provide a solution to \eqref{eqn:dis LQR obj} using $\hat{A}$ and $\hat{B}$, and characterize its performance (i.e., suboptimality) guarantees. 

In the rest of this paper, we first analyze the estimation error of $\hat{A}$ and $\hat{B}$ obtained from the procedure described above. In particular, we show in Section~\ref{sec:system ID} that the estimation errors $\norm{\hat{A}-A}$ and $\norm{\hat{B}-B}$ scale as $\tilde{\CO}(1/\sqrt{N})$ with high probability.\footnote{Throughout this paper, we let $\tilde{\CO}(\cdot)$ hide logarithmic factors in $N$.} Next, in Section~\ref{sec:control design}, we design a control policy $\hat{u}(\cdot)$, based on $\hat{A}$ and $\hat{B}$, which satisfies the information constraints given in \eqref{eqn:dis LQR obj}. Supposing $\norm{\hat{A}-A}\le\varepsilon$ and $\norm{\hat{B}-B}\le\varepsilon$, where $\varepsilon\in\R_{>0}$, and denoting the cost of \eqref{eqn:dis LQR obj} corresponding to $\hat{u}(\cdot)$ as $\hat{J}$, we show in Section~\ref{sec:perturb bound on cost} that 
\begin{equation*}
\hat{J}-J_{\star}\le C\varepsilon,
\end{equation*}
as long as $\varepsilon\le C_0$, where $J_{\star}$ is the optimal cost of \eqref{eqn:dis LQR obj} given by \eqref{eqn:opt J}, and $C$ and $C_0$ are constants that explicitly depend on the problem parameters of \eqref{eqn:dis LQR obj}. Finally, combining the above results together, we show in Section~\ref{sec:sample complexity} that with high probability and for large enough $N$, the following end-to-end sample complexity of learning decentralized LQR with the partially nested information structure holds:
\begin{equation*}
\hat{J}-J_{\star}=\tilde{\CO}(\frac{1}{\sqrt{N}}).
\end{equation*}

\section{System Identification Using Least Squares}\label{sec:system ID}
As we described in Section~\ref{sec:problem formulation}, we use a least squares approach to estimate the system matrices $A\in\R^{n\times n}$ and $B\in\R^{n\times m}$, based on a single system trajectory consisting of the control input sequence $\{u(0),\dots,u(N-1)\}$ and the system state sequence $\{x(0),\dots,x(N)\}$, where $x(0)=0$ and $N\in\Z_{\ge1}$. Here, we draw the inputs $u(0),\dots,u(N-1)$ independently from a Gaussian distribution $\CN(0,\sigma^2_uI)$, where $\sigma_u\in\R_{>0}$. In other words, we let $u(t)\overset{\text{i.i.d.}}{\sim}\CN(0,\sigma_u^2I)$ for all $t\in\{0,\dots,N-1\}$. Moreover, we assume that the input $u(t)$ and the disturbance $w(t)$ are independent for all $t\in\{0,\dots,N-1\}$. Note that we consider the scenario where the estimation of $A$ and $B$ is performed in a centralized manner using a least squares approach (detailed in Algorithm~\ref{algorithm:least squares}). However, we remark that Algorithm~\ref{algorithm:least squares} can be carried out without violating the information constraints given by Eq.~\eqref{eqn:info set}, since $u(t)\overset{\text{i.i.d.}}{\sim}\CN(0,\sigma_u^2I)$ is not a function of the states in the information set defined in Eq.~\eqref{eqn:info set} for any $t\in\{0,\dots,N-1\}$. In the following, we present the least squares approach to estimate $A$ and $B$, and characterize the corresponding estimation error.

\subsection{Least Squares Estimation of System Matrices}\label{sec:estimate A and B}
Let us denote
\begin{equation}
\label{eqn:Theta_i and z_Ni}
\Theta=\begin{bmatrix}A & B\end{bmatrix}\ \text{and}\ z(t)=\begin{bmatrix}x(t)^{\top} & u(t)^{\top}\end{bmatrix}^{\top},
\end{equation}
where $\Theta\in\R^{n\times(n+m)}$ and $z(t)\in\R^{n+m}$. Given the sequences $\{z(0),\dots,z(N-1)\}$ and $\{x(1),\dots,x(N)\}$, we use the regularized least squares to obtain an estimate of $\Theta$, denoted as $\hat{\Theta}(N)$, i.e.,
\begin{equation}
\label{eqn:least squares approach}
\hat{\Theta}(N)=\argmin_{Y\in\R^{n\times (n+m)}}\Big\{\lambda\norm{Y}_F^2+\sum_{t=0}^{N-1}\norm{x(t+1)- Y z(t)}^2\Big\},
\end{equation}
where $\lambda\in\R_{>0}$ is the regularization parameter. 
We summarize the above least squares approach in Algorithm~\ref{algorithm:least squares}.
\begin{algorithm}
\textbf{Input:} parameter $\lambda>0$ and time horizon length $N$
\caption{Least Squares Estimation of $A$ and $B$}
\label{algorithm:least squares}
\begin{algorithmic}[1]
\State Initialize $x(0)=0$
\For{$t=0,\dots,N-1$}
    \State Play $u(t)\overset{\text{i.i.d.}}{\sim}\CN(0,\sigma_u^2I)$
\EndFor
\State Obtain $\hat{\Theta}(N)$ using~\eqref{eqn:least squares approach}
\State Extract $\hat{A}$ and $\hat{B}$ from $\hat{\Theta}(N)$
\end{algorithmic} 
\end{algorithm}

\subsection{Least Squares Estimation Error}\label{sec:bounds on ls error}
In order to characterize the estimation error of $\hat{\Theta}(N)$ given by \eqref{eqn:least squares approach}, we will use the following result from \cite{cohen2019learning}, which is a consequence of \cite[Theorem~1]{abbasi2011improved}.
\begin{lemma}\cite[Lemma~6]{cohen2019learning}
\label{lemma:est error of Theta_hat}
For any $t\in\Z_{\ge1}$, let $V(t)=\lambda I +\sum_{k=0}^{t-1}z(k)z(k)^{\top}$ and $\Delta(t)=\Theta-\hat{\Theta}(t)$, where $\Theta$ and $z(k)$ are given in \eqref{eqn:Theta_i and z_Ni}, $\hat{\Theta}(t)$ is given by \eqref{eqn:least squares approach}, and $\lambda\in\R_{>0}$. Then, for any $\delta_{\Theta}>0$, the following hold with probability at least $1-\delta_{\Theta}$:
\begin{equation*}
\label{eqn:est error of Theta_hat upper bound}
\tr\big(\Delta(t)^{\top}V(t)\Delta(t)\big)\le 4\sigma_w^2n\log\bigg(\frac{n}{\delta_{\Theta}}\frac{\det(V(t))}{\det(\lambda I)}\bigg)+2\lambda\norm{\Theta}_F^2,\quad\forall t\in\Z_{\ge0}.
\end{equation*}
\end{lemma}
For any $\delta>0$, we now introduce the following probabilistic events that will be useful in our analysis later:
\begin{equation}
\label{eqn:events}
\begin{split}
\CE_w&=\bigg\{\max_{0\le t\le N-1}\norm{w(t)}\le\sigma_w\sqrt{5n\log\frac{4N}{\delta}}\bigg\},\\
\CE_u&=\bigg\{\max_{0\le t\le N-1}\norm{u(t)}\le\sigma_u\sqrt{5m\log\frac{4N}{\delta}}\bigg\},\\
\CE_{\Theta}&=\bigg\{\tr\big(\Delta(N)^{\top}V(N)\Delta(N)\big)\le 4\sigma_w^2n\log\bigg(\frac{4n}{\delta}\frac{\det(V(N))}{\det(\lambda I)}\bigg)+2\lambda \norm{\Theta}_F^2\bigg\},\\
\CE_z&=\bigg\{\sum_{t=0}^{N-1}z(t)z(t)^{\top}\succeq\frac{(N-1)\underline{\sigma}^2}{40}I\bigg\},
\end{split}
\end{equation}
where $\underline{\sigma}\triangleq\min\{\sigma_w,\sigma_u\}$. Denoting
\begin{equation}
\label{eqn:good event}
\CE=\CE_{w}\cap\CE_{v}\cap\CE_{\Theta}\cap\CE_z,
\end{equation}
we have the following result; the proof is included in Appendix~\ref{sec:proofs ls}.
\begin{lemma}
\label{lemma:prob of CE}
For any $\delta>0$ and for any $N\ge200(n+m)\log\frac{48}{\delta}$, it holds that $\Prob(\CE)\ge1-\delta$.
\end{lemma}

For the analysis in the sequel, we will make the following assumption, which is also made in related literature (see e.g., \cite{lale2020logarithmic,simchowitz2020improper,zheng2021sample}).
\begin{assumption}
\label{ass:stable A}
The system matrix $A\in\R^{n\times n}$ is stable, and $\norm{A^k}\le\kappa_0\gamma_0^k$ for all $k\in\Z_{\ge0}$, where $\kappa_0\ge1$ and $\rho(A)<\gamma_0<1$.
\end{assumption}
Note that for any stable matrix $A$, we have from the Gelfand formula (e.g., \cite{horn2012matrix}) that there always exist $\kappa_0\in\R_{\ge1}$ and $\gamma_0\in\R$ with $\rho(A)<\gamma_0<1$ such that $\norm{A^k}\le\kappa_0\gamma_0^k$ for all $k\in\Z_{\ge0}$. We then have the following results; the proofs are included in Appendix~\ref{sec:proofs ls}.
\begin{lemma}
\label{lemma:upper bound on norm of z(t)}
Suppose Assumption~\ref{ass:stable A} holds. On the event $\CE$ defined in Eq.~\eqref{eqn:good event},
\begin{equation}
\norm{z(t)}\le\frac{5\kappa_0}{1-\gamma_0}\overline{\sigma}\sqrt{(\norm{B}^2m+m+n)\log\frac{4N}{\delta}},
\end{equation}
for all $t\in\{0,\dots,N-1\}$, where $N\ge1$, $\overline{\sigma}=\max\{\sigma_w,\sigma_u\}$, $\gamma_0$ and $\kappa_0$ are given in Assumption~\ref{ass:stable A}, and $z(t)=\begin{bmatrix}x(t)^{\top} & u(t)^{\top}\end{bmatrix}^{\top}$ with $x(t)\in\R^n$ and $w(t)\in\R^m$ to be the state and input of the system in Eq.~\eqref{eqn:overall system}, respectively, corresponding to Algorithm~\ref{algorithm:least squares}.
\end{lemma}

\begin{proposition}
\label{prop:upper bound on est error}
Suppose Assumption~\ref{ass:stable A} holds, and $\norm{A}\le\vartheta$ and $\norm{B}\le\vartheta$, where $\vartheta\in\R_{>0}$. Consider any $\delta>0$. Let the input parameters to Algorithm~\ref{algorithm:least squares} satisfy $N\ge200(n+m)\log\frac{48}{\delta}$ and $\lambda\ge\underline{\sigma}^2/40$, where $\underline{\sigma}=\min\{\sigma_w,\sigma_u\}$. Define
\begin{equation*}
z_b=\frac{5\kappa_0}{1-\gamma_0}\overline{\sigma}\sqrt{(\norm{B}^2m+m+n)\log\frac{4N}{\delta}},
\end{equation*}
where $\kappa_0$ and $\gamma_0$ are given in Assumption~\ref{ass:stable A}, and $\overline{\sigma}=\max\{\sigma_w,\sigma_u\}$. Then, with probability at least $1-\delta$, it holds that $\norm{\hat{A}-A}\le\varepsilon_0$ and $\norm{\hat{B}-B}\le\varepsilon_0$, where $\hat{A}$ and $\hat{B}$ are returned by Algorithm~\ref{algorithm:least squares}, and
\begin{equation}
\label{eqn:epsilon_0}
\varepsilon_0=4\sqrt{\frac{160}{N\underline{\sigma}^2}\bigg(2n\sigma_w^2(n+m)\log\frac{N+z^2_b/\lambda}{\delta}+\lambda n\vartheta^2\bigg)}.
\end{equation}
\end{proposition}

Several remarks pertaining to Algorithm~\ref{algorithm:least squares} and the result in Proposition~\ref{prop:upper bound on est error} are now in order. First, note that while considering the problem of learning centralized LQR without any information constraints, the authors in \cite{dean2020sample} proposed to obtain $\hat{A}$ and $\hat{B}$ from multiple system trajectories using least squares, where each trajectory starts from $x(0)=0$. They showed that $\norm{\hat{A}-A}=\CO(1/\sqrt{N_r})$ and $\norm{\hat{B}-B}=\CO(1/\sqrt{N_r})$, where $N_r\in\Z_{\ge1}$ is the number of system trajectories. In contrast, we estimate $A$ and $B$ from a single system trajectory, and achieve $\norm{\hat{A}-A}=\tilde{\CO}(1/\sqrt{N})$ and $\norm{\hat{B}-B}=\tilde{\CO}(1/\sqrt{N})$. 

Second, note that we use the regularized least squares in Algorithm~\ref{algorithm:least squares} to obtain estimates $\hat{A}$ and $\hat{B}$. Although least squares without regularization can also be used to obtain estimates $\hat{A}$ and $\hat{B}$ from a single system trajectory with the same $\tilde{\CO}(1/\sqrt{N})$ finite sample guarantee (e.g., \cite{simchowitz2018learning}), we choose to use regularized least squares considered in, e.g., \cite{abbasi2011improved,cohen2019learning,cassel2020logarithmic}. The reason is that introducing the regularization into least squares makes the finite sample analysis more tractable (e.g., \cite{cohen2019learning,cassel2020logarithmic}), which facilitates the adaption of the analysis in \cite{cohen2019learning,cassel2020logarithmic} to our setting described in this section. Moreover, note that the lower bound on $\lambda$ required in Proposition~\ref{prop:upper bound on est error} is merely used to guarantee that the denominator of the right-hand side of Eq.~\eqref{eqn:epsilon_0} contains the factor $1/\sqrt{N}$; choosing an arbitrary $\lambda\in\R_{>0}$ leads to a factor $1/\sqrt{N-1}$. In general, one can show that choosing any $\lambda\in\R_{>0}$ leads to the same $\tilde{\CO}(1/\sqrt{N})$ finite sample guarantee.

Third, note that we do not leverage the block structure (i.e., sparsity pattern) of $A$ and $B$ described in Section~\ref{sec:dist LQR known matrices}, when we obtain $\hat{A}$ and $\hat{B}$ using Algorithm~\ref{algorithm:least squares}. Therefore, the sparsity pattern of $\hat{A}$ and $\hat{B}$ may potentially be inconsistent with that of $A$ and $B$. Nonetheless, such a potential inconsistency does not play any role in our analysis later. The reason is that the control policy to be proposed in Section~\ref{sec:control design} does not depend on the sparsity pattern of $\hat{A}$ and $\hat{B}$. Moreover, when analyzing the suboptimality of the proposed control policy later in Section~\ref{sec:sub guarantees}, we only leverage the fact that the estimation error corresponding to submatrices in $\hat{A}$ (resp., $\hat{B}$) will be upper bounded by $\norm{\hat{A}-A}$ (resp., $\norm{\hat{B}-B}$). Specifically, considering any nodes $s,r$ in the information graph $\CP(\U,\CH)$ given by \eqref{eqn:def of info graph}, one can show that $\norm{\hat{A}_{sr}-A_{sr}}\le\norm{\hat{A}-A}$, where recall that $\hat{A}_{sr}$ (resp., $A_{sr}$) is a submatrix of $\hat{A}$ (resp., $A$) that corresponds to the nodes of the directed graph $\G(\V,\A)$ contained in $s$ and $r$.

Finally, we remark that one may also use system identification schemes and the associated sample complexity analysis dedicated to sparse system matrices (e.g., \cite{fattahi2019learning}). Under some extra assumptions on $A$ and $B$ (e.g., \cite{fattahi2019learning}), one may then obtain $\hat{A}$ and $\hat{B}$ that have the same sparsity pattern as $A$ and $B$, and remove the logarithmic factor in $N$ in $\varepsilon_0$ defined in Proposition~\ref{prop:upper bound on est error}. However, the assumptions on $A$ and $B$ made in e.g., \cite{fattahi2019learning} can be restrictive and hard to check in practice.

\section{Control Policy Design}\label{sec:control design}
While the estimation of $A$ and $B$ is performed in a centralized manner as we discussed in Section~\ref{sec:estimate A and B}, we assume that each controller $i\in\V$ receives the estimates $\hat{A}$ and $\hat{B}$ after we conduct the system identification step described in Algorithm~\ref{algorithm:least squares}. Given the matrices $\hat{A}$, $\hat{B}$, $Q$, and $R$, and the directed graph $\G(\V,\A)$ ($\V=[p]$) with the delay matrix $D$, in this section we design a control policy that can be implemented in a decentralized manner, while satisfying the information constraints described in Section~\ref{sec:dist LQR known matrices}. To this end, we leverage the structure of the optimal policy $u^{\star}(\cdot)$ given in Lemma~\ref{lemma:opt solution} (when $A$ and $B$ are known). Note that the optimal policy $u^{\star}(\cdot)$ cannot be applied to our scenario, since only $\hat{A}$ and $\hat{B}$ are available.

First, given the directed graph $\G(\V,\A)$ with $\V=[p]$ and the delay matrix $D$, we construct the information graph $\CP(\U,\CH)$ given by \eqref{eqn:def of info graph}. Recall from Remark~\ref{remark:tree info graph} that $\CP(\U,\CH)$ is a forest that contains a set of disconnected directed trees. We then let $\CL$ denote the set of all the leaf nodes in $\CP(\U,\CH)$, i.e.,
\begin{equation}
\label{eqn:set of leaf nodes}
\CL=\{s_i(0)\in\U:i\in\V\}.
\end{equation}
Moreover, for any $s\in\U$, we denote
\begin{equation}
\label{eqn:set of leaf nodes of s}
\CL_s=\{v\in\CL:v\rightsquigarrow s\},
\end{equation}
where we write $v\rightsquigarrow s$ if and only if there is a unique directed path from node $v$ to node $s$ in $\CP(\U,\CH)$. In other words, $\CL_s$ is the set of leaf nodes in $\CP(\U,\CH)$ that can reach $s$. Moreover, for any $v,s\in\U$ such that $v\rightsquigarrow s$, we let $l_{vs}$ denote the length of the unique directed path from $v$ to $s$ in $\CP(\U,\CH)$; we let $l_{vs}=0$ if $v=s$. For example, in the information graph (associated with Example~\ref{exp:running example}) given in Fig.~\ref{fig:info graph}, we have $\CL=\{\{1\},\{1,2\},\{3\}\}$, $\CL_{\{1,2,3\}}=\{\{1\},\{1,2\}\}$, and $l_{\{1\}\{1,2,3\}}=1$.

Next, in order to leverage the structure of the optimal policy $u^{\star}(\cdot)$ given in Eqs.~\eqref{eqn:set of DARES K}-\eqref{eqn:exp for u^star}, we substitute (submatrices of) $\hat{A}$ and $\hat{B}$ into the right-hand sides of Eqs.~\eqref{eqn:set of DARES K}-\eqref{eqn:set of DARES P}, and obtain $\hat{K}_r$ and $\hat{P}_r$ for all $r\in\U$. Specifically, for all $r\in\U$, we obtain $\hat{K}_r$, and $\hat{P}_r$ recursively as
\begin{align}
\hat{K}_r &= -(R_{rr}+\hat{B}_{sr}^{\top} \hat{P}_s \hat{B}_{sr}^{\top})^{-1} \hat{B}_{sr}^{\top} \hat{P}_s \hat{A}_{sr},\label{eqn:set of DARES K hat}\\
\hat{P}_r &= Q_{rr}+\hat{K}_r^{\top}R_{rr}\hat{K}_r+(\hat{A}_{sr}+\hat{B}_{sr}\hat{K}_r)^{\top}\hat{P}_s(\hat{A}_{sr}+\hat{B}_{sr}\hat{K}_r),\label{eqn:set of DARES P hat}
\end{align}
where for each $r\in\U$, we let $s\in\U$ be the unique node such that $r\rightarrow s$, and $\hat{A}_{sr}$ (resp.,  $\hat{B}_{sr}$) is a submatrix of $\hat{A}$ (resp., $\hat{B}$) obtained in the same manner as $A_{sr}$ (resp., $B_{sr}$) described before. Similarly to Eq.~\eqref{eqn:dynamics of zeta}, we then use $\hat{K}_r$ for all $r\in\U$ together with $\hat{A}$ and $\hat{B}$ to maintain an (estimated) internal state $\hat{\zeta}_r(t)$ (to be defined later) for all $r\in\U$ and for all $t\in\{0,\dots,T-1\}$, which, via a similar form to Eq.~\eqref{eqn:exp for u^star}, will lead to our control policy, denoted as $\hat{u}_i(t)$, for all $i\in\V$ and for all $t\in\{0,\dots,T-1\}$. Specifically, for all $i\in\V$ in parallel, we propose Algorithm~\ref{algorithm:control design} to compute the control policy 
\begin{equation}
\label{eqn:control policy}
\hat{u}_i(t)=\sum_{r\ni i}I_{\{i\},r}\hat{K}_r\hat{\zeta}_r(t)\quad\forall t\in\{0,\dots,T-1\}.
\end{equation}

\begin{algorithm}
\textbf{Input:} estimates $\hat{A}$ and $\hat{B}$, cost matrices $Q$ and $R$, directed graph $\G(\V,\A)$ with $\V=[p]$ and delay matrix $D$, time horizon length $T$
\caption{Control policy design for node $i\in\V$}\label{algorithm:control design}
\begin{algorithmic}[1]
\State Construct the information graph $\CP(\U,\CH)$ from \eqref{eqn:def of info graph}
\State Obtain $\hat{K}_s$ for all $s\in\U$ from Eq.~\eqref{eqn:set of DARES K}
\State Initialize $\M_i\gets\bar{\M}_i$
\For{$t=0,\dots T-1$}
    \For{$s\in\CL(\T_i)$}
        \State Find $s_j(0)\in\U$ s.t. $j\in\V$ and $s_j(0)=s$
        \State Obtain $\hat{w}_j(t-D_{ij}-1)$ from Eq.~\eqref{eqn:est w_i(t)}
        \State Obtain $\hat{\zeta}_s(t-D_{ij})$ from Eq.~\eqref{eqn:dynamics of zeta hat}
        \State $\M_i\gets \M_i\cup\{\hat{\zeta}_s(t-D_{ij})\}$
    \EndFor
    \For{$s\in\CR(\T_i)$}
        \State Obtain $\hat{\zeta}_s(t-D_{\max})$ from Eq.~\eqref{eqn:dynamics of zeta hat}
        \State $\M_i\gets \M_i\cup\{\hat{\zeta}_s(t-D_{\max})\}$
    \EndFor
    \State Play $\hat{u}_i(t)=\sum_{r\ni i}I_{\{i\},r}\hat{K}_r\hat{\zeta}_r(t)$
    \State $\M_i\gets\M_i\setminus\big(\{\hat{\zeta}_s(t-2D_{\max}-1):s\in\CL(\T_i)\}\cup\{\hat{\zeta}_s(t-D_{\max}-1):s\in\CR(\T_i)\}\big)$
\EndFor
\end{algorithmic} 
\end{algorithm}

We now describe the notations used in Algorithm~\ref{algorithm:control design} and hereafter. Let us consider any $i\in\V$. In Algorithm~\ref{algorithm:control design}, we let $\T_i$ denote the set of disconnected directed trees in $\CP(\U,\CH)$ such that the root node of any tree in $\T_i$ contains $i$. Slightly abusing the notation, 
we also let $\T_i$ denote the set of nodes of all the trees in $\T_i$. Moreover, we denote 
\begin{equation}
\label{eqn:leaf nodes in T_i}
\CL(\T_i)=\T_i\cap\CL,
\end{equation}
where $\CL$ is defined in Eq.~\eqref{eqn:set of leaf nodes}, i.e., $\CL(\T_i)$ is the set of leaf nodes of all the trees in $\T_i$. Letting $\CR\subseteq\U$ be the set of root nodes in $\CP(\U,\CH)$, we denote 
\begin{equation}
\label{eqn:root nodes in T_i}
\CR(\T_i) = \T_i\cap\CR,
\end{equation}
where we recall from Lemma~\ref{lemma:properties of info graph} that any root node in $\CP(\U,\CH)$ has a self loop.  We then see from the information graph $\CP(\U,\CH)$ given in Fig.~\ref{fig:info graph} that 
\begin{align*}
&\CL(\T_1)=\{\{1\},\{1,2\},\{3\}\},\ \CL(\T_2)=\CL(\T_3)=\{\{1,2\},\{3\}\},\\
&\CR(\T_1)=\{\{1\},\{1,2,3\}\},\ \CR_2(\T_2)=\CR(\T_3)=\{1,2,3\}.
\end{align*}
Note that if any node $s\in\T_i$ is a leaf node with a self loop (i.e., $s$ is an isolated node in $\CP(\U,\CH$)) such as the node $\{3\}$ in Fig.~\ref{fig:info graph}, we only include $s$ in $\CL(\T_i)$ (i.e., $s\in\CL(\T_i)$ but $s\not\in\CR(\T_i)$). 

Furthermore, we denote
\begin{equation}
\label{eqn:depth of T_i}
D_{\max}=\max_{\substack{i,j\in\V\\ j\rightsquigarrow i}} D_{ij},
\end{equation}
where we write $j\rightsquigarrow i$ if and only if there is a directed path from node $j$ to node $i$ in $\G(\V,\A)$, and recall that $D_{ij}$ is the sum of delays along the directed path from $j$ to $i$ with the smallest accumulative delay. Finally, the memory $\M_i$ of Algorithm~\ref{algorithm:control design} is initialized as $\M_i=\bar{\M}_i$ with
\begin{align}
\bar{\M}_i=&\big\{\hat{\zeta}_s(k):k\in\{-2D_{\max}-1,\dots,-D_{ij}-1\}, s\in\CL(\T_i),j\in\V,s_j(0)=s\big\}\cup\{\hat{\zeta}_s(-D_{\max}-1):s\in\CR(\T_i)\},\label{eqn:initial M_i}
\end{align}
where we initialize $\hat{\zeta}_s(k)=0$ for all $\hat{\zeta}_s(k)\in\bar{\M}_i$.
\begin{remark}
\label{remark:order of elements in CL}
For any $s,r\in\CL(\T_i)$, let $j_1,j_2\in\V$ be such that $s_{j_1}(0)=s$ and $s_{j_2}(0)=r$. In Algorithm~\ref{algorithm:control design}, we assume that the elements in $\CL(\T_i)$ are already ordered such that if $D_{ij_1}>D_{ij_2}$, then $s$ comes before $r$ in $\CL(\T_i)$. We then let the for loop in lines~5-9 in Algorithm~\ref{algorithm:control design} iterate over the elements in $\CL(\T_i)$ according to the above order. Considering the node $2$ in the directed graph $\G(\V,\A)$ in Example~\ref{exp:running example}, we see from Fig.~\ref{fig:directed graph} and Fig.~\ref{fig:info graph} that $\CL(\T_2)=\{\{1,2\},\{3\}\}$, where $s_2(0)=\{1,2\}$ and $s_3(0)=\{3\}$. Since $D_{22}=0$ and $D_{23}=1$, we assume that the elements in $\CL(\T_2)$ are ordered such that $\CL(\T_2)=\{\{3\},\{1,2\}\}$. 
\end{remark}

\begin{remark}
\label{remark:def of D_max}
Recall that each edge in $\G(\V,\A)$ is associated with a delay of either $0$ or $1$. Considering the scenario with only sparsity constraints (e.g., \cite{shah2013cal}), i.e., all the edges in $\G(\V,\A)$ have a zero delay, we see that $D_{ij}=0$ for all $i,j\in\V$ such that $j\rightsquigarrow i$, which implies via Eq.~\eqref{eqn:depth of T_i} that $D_{\max}=0$. 
\end{remark}

For any $r\ni i$, the dynamics of the internal state $\hat{\zeta}_r(t)$ is given by
\begin{equation}
\label{eqn:dynamics of zeta hat}
\hat{\zeta}_r(t+1)= \sum_{v\rightarrow r}(\hat{A}_{rv}+\hat{B}_{rv}\hat{K}_v)\hat{\zeta}_v(t) + \sum_{w_j\rightarrow r}I_{r,\{j\}}\hat{w}_j(t),
\end{equation}
where $\hat{w}_j(t)$ is an estimate of the disturbance $w_j(t)$ in Eq.~\eqref{eqn:dynamics for x_i(t)} obtained as
\begin{equation}
\label{eqn:est w_i(t)}
\hat{w}_j(t)=
\begin{cases}
0\ \text{if}\ t<-1,\\
x_j(0)\ \text{if}\ t=-1,\\
x_j(t+1) - \hat{A}_jx_{\CN_j}(t) - \hat{B}_j\hat{u}_{\CN_j}(t)\ \text{if}\ t\ge0,
\end{cases}
\end{equation}
where we replace $A_j$ and $B_j$ with the estimates $\hat{A}_j$ and $\hat{B}_j$ in Eq.~\eqref{eqn:dynamics for x_i(t)}, respectively, and $\hat{u}_{\CN_j}(t)$ is the vector that collects $\hat{u}_{j_1}(t)$ for all $j_1\in\CN_j$, with $\CN_j$ given in Assumption~\ref{ass:info structure}. We note from Eqs.~\eqref{eqn:dynamics of zeta hat}-\eqref{eqn:est w_i(t)} that $\hat{\zeta}_r(0)=\sum_{w_j\rightarrow r}I_{r,\{j\}}x_j(0)$, where $x(0)=0$ as we assumed previously. We emphasize that Eqs.~\eqref{eqn:dynamics of zeta hat}-\eqref{eqn:est w_i(t)} are the keys to our control policy design, and they also enable our analyses in Section~\ref{sec:sub guarantees}, where we provide a suboptimality guarantee of our control policy. As we mentioned in Section~\ref{sec:intro}, the motivation of the control policy $\hat{u}(\cdot)$ given by Eqs.~\eqref{eqn:control policy}, \eqref{eqn:dynamics of zeta hat}-\eqref{eqn:est w_i(t)} is that the optimal control policy given in Lemma~\ref{lemma:opt solution} can be viewed as a disturbance-feedback controller. Since the system matrices $A$ and $B$ are unknown, the control policy $\hat{u}(\cdot)$ constructed in Eqs.~\eqref{eqn:control policy}, \eqref{eqn:dynamics of zeta hat}-\eqref{eqn:est w_i(t)} maps the estimates of the past disturbances given by Eq.~\eqref{eqn:est w_i(t)} to the current control input via the estimated internal states given by Eq.~\eqref{eqn:dynamics of zeta hat}.

\begin{observation}
\label{ob:simplify zeta_hat}
From the structure of the information graph $\CP(\U,\CH)$ defined in~\eqref{eqn:def of info graph}, the following hold:\\
(a) If $r$ is not a leaf node in $\CP(\U,\CH)$, Eq.~\eqref{eqn:dynamics of zeta hat} reduces to $\hat{\zeta}_r(t+1)= \sum_{v\rightarrow r}(\hat{A}_{rv}+\hat{B}_{rv}\hat{K}_v)\hat{\zeta}_v(t)$.\\
(b) If $r$ is a leaf node in $\CP(\U,\CH)$ that is not isolated, Eq.~\eqref{eqn:dynamics of zeta hat} reduces to $\hat{\zeta}_r(t+1)=\sum_{w_j\rightarrow r}I_{r,\{j\}}\hat{w}_j(t)$.\\
(c) If $r$ is an isolated node in $\CP(\U,\CH)$, Eq.~\eqref{eqn:dynamics of zeta hat} reduces to $\hat{\zeta}_r(t+1)= (\hat{A}_{rr}+\hat{B}_{rr}\hat{K}_r)\hat{\zeta}_r(t)+\sum_{w_j\rightarrow r}I_{r,\{j\}}\hat{w}_j(t)$.
\end{observation}

We will show that in each iteration $t\in\{0,\dots,T-1\}$ of the for loop in lines~4-14 of Algorithm~\ref{algorithm:control design}, the internal states $\hat{\zeta}_r(t)$ for all $r\in\U$ such that $i\in r$ (i.e., for all $r\ni i$) can be determined, via Eq.~\eqref{eqn:dynamics of zeta hat}, based on the current memory $\M_i$ of the algorithm and the state information contained in (a subset of) the information set $\I_i(t)$ defined in Eq.~\eqref{eqn:info set}. As we will see, Algorithm~\ref{algorithm:control design} maintains, in its current memory $\M_i$, the internal states (with potential time delays) for a certain subset of nodes in $\U$, via the recursion in Eq.~\eqref{eqn:dynamics of zeta hat}. Given those internal states, $\hat{\zeta}_r(t)$ for all $r\ni i$ can be determined using Eq.~\eqref{eqn:dynamics of zeta hat}. Moreover, the memory $\M_i$ of Algorithm~\ref{algorithm:control design} is recursively updated in the for loop in lines~4-14 of the algorithm. Formally, we have the following result for Algorithm~\ref{algorithm:control design}; the proof can be found in Appendix~\ref{sec:proof of control policy alg}.
\begin{proposition}
\label{prop:alg1 feasible}
Suppose that any controller $i\in\V$ at any time step $t\in\Z_{\ge0}$ has access to the states in $\tilde{\I}_i(t)$ defined as
\begin{equation}
\label{eqn:info set used}
\tilde{\I}_i(t)=\big\{x_j(k):j\in\V_i,k\in\{t-D_{\max}-1,\dots,t-D_{ij}\}\big\}\subseteq\I_i(t),
\end{equation}
where $\V_i=\{j\in\V:D_{ij}\neq+\infty\}$, and $\I_i(t)$ is defined in Eq.~\eqref{eqn:info set}. Then, the following properties hold for Algorithm~\ref{algorithm:control design}:\\
(a) The memory $\M_i$ of Algorithm~\ref{algorithm:control design} can be recursively updated such that at the beginning of any iteration $t\in\{0,\dots,T-1\}$ of the for loop in lines~4-14 of the algorithm,
\begin{multline}
\M_i=\big\{\hat{\zeta}_s(k):k\in\{t-2D_{\max}-1,\dots,t-D_{ij}-1\}, s\in\CL(\T_i),j\in\V,s_j(0)=s\big\}\\\cup\{\hat{\zeta}_s(t-D_{\max}-1):s\in\CR(\T_i)\}.\label{eqn:memory of the alg t-1}
\end{multline}
(b) The control input $\hat{u}_i(t)$ in line~13 can be determined using Eq.~\eqref{eqn:dynamics of zeta hat} and the states in the memory $\M_i$ after line~12 (and before line~14) in any iteration $t\in\{0,\dots,T-1\}$ of the for loop in lines~4-14 of Algorithm~\ref{algorithm:control design}.
\end{proposition}

Since the proof of Proposition~\ref{prop:alg1 feasible} is rather involved and requires careful considerations of the structures of the directed graph $\G(\V,\A)$ and the information graph $\CP(\U,\CH)$ described in Section~\ref{sec:dist LQR known matrices}, we again use Example~\ref{exp:running example} to illustrate the steps of Algorithm~\ref{algorithm:control design} and the results and proof ideas of Proposition~\ref{prop:alg1 feasible}. 

First, we note from Fig.~\ref{fig:directed graph} and Eq.~\eqref{eqn:depth of T_i} that $D_{\max}=1$. Now, let us consider Algorithm~\ref{algorithm:control design} with respective to node $2$ in the directed graph $\G(\V,\A)$ given in Fig.~\ref{fig:directed graph}. We see that $\V_2=\{j\in\V:D_{ij}\neq\infty\}=\{2,3\}$, which implies via Eq.~\eqref{eqn:info set used} that $\tilde{\I}_2(t)=\{x_2(t-2),x_2(t-1),x_2(t),x_3(t-2),x_3(t-1)\}$ for all $t\in\{0,\dots,T-1\}$. One can check that the initial memory $\bar{\M}_2$ of Algorithm~\ref{algorithm:control design} given by Eq.~\eqref{eqn:initial M_i} satisfies Eq.~\eqref{eqn:memory of the alg t-1} for $t=0$, which implies that the memory $\M_2$ satisfies Eq.~\eqref{eqn:memory of the alg t-1} at the beginning of iteration $t=0$ of the for loop in lines~4-14 of the algorithm. 

To proceed, let us consider iteration $t=0$ of the for loop in lines~4-14 of the algorithm. Noting that $\CL(\T_2)=\{\{3\},\{1,2\}\}$ from Remark~\ref{remark:order of elements in CL}, Algorithm~\ref{algorithm:control design} first considers $s=\{3\}$ in the for loop in lines~5-9, which implies that $j=3$ in line~7. We then see from Eq.~\eqref{eqn:est w_i(t)} that in order to obtain $\hat{w}_3(t-2)$, we need to know $x_3(t-1)$, $x_3(t-2)$, $x_2(t-2)$, $\hat{u}_3(t-2)$ and $\hat{u}_2(t-2)$, where $x_3(t-1),x_3(t-2),x_2(t-2)\in\tilde{\I}_2(t)$, and $\hat{u}_2(t-2),\hat{u}_3(t-2)$ are given by Eq.~\eqref{eqn:control policy}. One can then check that the internal states $\hat{\zeta}_r(t^{\prime})$ that are needed to determine $\hat{u}_2(t-2)$ and $\hat{u}_3(t-2)$ are available in the current memory $\M_2$ of Algorithm~\ref{algorithm:control design} or become available via further applications of Eq.~\eqref{eqn:dynamics of zeta hat}. After $\hat{w}_3(t-2)$ is obtained, we see from Eq.~\eqref{eqn:dynamics of zeta hat} that $\hat{\zeta}_{\{3\}}(t-1)$ can also be obtained. Algorithm~\ref{algorithm:control design} then updates its current memory $\M_2$ in line~9 and finishes the iteration with respect to $s=\{3\}$ in the for loop in lines~5-9. Next, Algorithm~\ref{algorithm:control design} considers $s=\{1,2\}$ in the for loop in lines~5-9, which implies that $j=2$ in line~7. Following similar arguments to those above for $s=\{3\}$ and noting that the current memory $\M_2$ of Algorithm~\ref{algorithm:control design} has been updated, one can show that $\hat{\zeta}_{\{1,2\}}(t)$ can be obtained from Eq.~\eqref{eqn:dynamics of zeta hat}, based on the current memory of the algorithm. Algorithm~\ref{algorithm:control design} again updates its current memory $\M_2$ in line~9 and finishes the iteration with respect to $s=\{1,2\}$ in the for loop in lines~5-9.

Now, recalling that $\CR(\T_2)=\{1,2,3\}$ from Fig.~\ref{fig:info graph}, we see that Algorithm~\ref{algorithm:control design} considers $s=\{1,2,3\}$ in line~10. One can also check that $\hat{\zeta}_{\{1,2,3\}}(t)$ can be obtained from Eq.~\eqref{eqn:dynamics of zeta hat}, based on the current memory of the algorithm. Finally, based on the current memory $\M_2$ of Algorithm~\ref{algorithm:control design} after line~12, one can check that the control input $\hat{u}_2(t)$ can be determined from Eq.~\eqref{eqn:control policy}. Note that Algorithm~\ref{algorithm:control design} also removes certain internal states from its current memory in line~14 that will no longer be used. One can check that after the removal, the current memory $\M_2$ of Algorithm ~\ref{algorithm:control design} will satisfy Eq.~\eqref{eqn:memory of the alg t-1} at the beginning of iteration $t+1$ of the for loop in lines~4-14 of the algorithm, where $t=0$. One can then repeat the above arguments for iteration $t=1$ of the for loop in lines~4-14 of the algorithm and so on.

Several remarks pertaining to Algorithm~\ref{algorithm:control design} are now in order. First, since $|\CL(\T_i)|\le p$ and $|\CR(\T_i)|\le p$, one can show via the definition of Algorithm~\ref{algorithm:control design} that the number of the states in the memory $\M_i$ of Algorithm~\ref{algorithm:control design} is always upper bounded by $(2D_{\max}+2)p+2p$, where we note that $D_{\max}$ defined in Eq.~\eqref{eqn:depth of T_i} satisfies $D_{\max}\le p$, and $p$ is the number of nodes in the directed graph $\G(\V,\A)$. Moreover, one can check that Algorithm~\ref{algorithm:control design} can be implemented in polynomial time.

Second, it is worth noting that the control policy $\hat{u}_i(\cdot)$ for all $i\in\U$ that we proposed in Eq.~\eqref{eqn:control policy} is related to the certainty equivalent approach (e.g., \cite{aastrom2008adaptive}) that has been used for learning centralized LQR without any information constraints on the controllers (e.g., \cite{dean2020sample,mania2019certainty,cassel2020logarithmic}). It is known that the optimal solution to classic centralized LQR (i.e., problem~\eqref{eqn:dis LQR obj} without the information constraints) is given by a static state-feedback controller $u^{\star}(t)=Kx(t)$, where $K$ can be obtained from the solution to the Ricatti equation corresponding to $A$, $B$, $Q$ and $R$ (e.g., \cite{bertsekas2017dynamic}). The corresponding certainty equivalent controller simply takes the form $\hat{u}(t)=\hat{K}x(t)$, where $\hat{K}$ is obtained from the solution to the Ricatti equation corresponding to $\hat{A}$, $\hat{B}$, $Q$ and $R$, with $\hat{A}$ and $\hat{B}$ to be the estimates of $A$ and $B$, respectively. While we also leverage the structure of the optimal control policy $u^{\star}(\cdot)$ given in Eq.~\eqref{eqn:exp for u^star}, we cannot simply replace $K_r$ with $\hat{K}_r$ for all $r\in\U$ in Eq.~\eqref{eqn:exp for u^star}, where $\hat{K}_r$ is given by the Ricatti equations in Eqs.~\eqref{eqn:set of DARES K hat}-\eqref{eqn:set of DARES P hat}. As we argued in Remark~\ref{remark:decentralized setting}, this is because $u^{\star}(\cdot)$ is not a static state-feedback controller, but a linear dynamic controller based on the internal states $\zeta_r(\cdot)$ for all $r\in\U$, where the dynamics of $\zeta_r(\cdot)$ given by Eq.~\eqref{eqn:dynamics of zeta} also depends on $A$ and $B$. Thus, the control policy $\hat{u}_i(\cdot)$ that we proposed in Eq.~\eqref{eqn:control policy} is a linear dynamic controller based on $\hat{K}_r$ and the estimated internal states $\hat{\zeta}_r(\cdot)$ for all $r\in\U$, where the dynamics of $\hat{\zeta}_r(\cdot)$ given by Eq.~\eqref{eqn:dynamics of zeta hat} depends on $\hat{A}$ and $\hat{B}$. Such a more complicated form of $\hat{u}_i(\cdot)$ also creates several challenges when we analyze the corresponding suboptimality guarantees in the next section.

Third, for any $i\in\V$ and any $t\in\{0,\dots,T-1\}$, Proposition~\ref{prop:alg1 feasible} only requires controller $i$ to have access to a subset of the state information contained in the information set $\I_i(t)$.

Finally, we remark that Algorithm~\ref{algorithm:control design} is not the unique way to implement the control policy $\hat{u}_i(\cdot)$ given in Eq.~\eqref{eqn:control policy}, under the information constraints on each controller $i\in\V$.

\section{Suboptimality Guarantees}\label{sec:sub guarantees}
In this section, we characterize the suboptimality guarantees of the control policy $\hat{u}(\cdot)$ proposed in Section~\ref{sec:control design}. To begin with, in order to explicitly distinguish the states of the system in Eq.~\eqref{eqn:overall system} corresponding to the control policies $u^{\star}(\cdot)$ and $\hat{u}(\cdot)$ given by Eqs.~\eqref{eqn:exp for u^star} and \eqref{eqn:control policy}, respectively, we let $\hat{x}(t)$ denote the state of the system in Eq.~\eqref{eqn:overall system} corresponding to the control policy $\hat{u}(\cdot)$ given by Eq.~\eqref{eqn:control policy}, for $t\in\Z_{\ge0}$, i.e.,
\begin{equation}
\label{eqn:state x hat}
\hat{x}(t+1) = A\hat{x}(t)+B\hat{u}(t)+w(t),
\end{equation}
where we note from Eq.~\eqref{eqn:control policy} that $\hat{u}(t)=\sum_{s\in\U}I_{\V,s}\hat{K}_s\hat{\zeta}_s(t)$ with $\hat{K}_s$ and $\hat{\zeta}_s(t)$ given by Eqs.~\eqref{eqn:set of DARES K hat} and \eqref{eqn:dynamics of zeta hat}, respectively, for all $s\in\U$. We let $x(t)$ denote the state of the system in Eq.~\eqref{eqn:overall system} corresponding to the optimal control policy $u^{\star}(t)$ given by Eq.~\eqref{eqn:exp for u^star}, for $t\in\Z_{\ge0}$, i.e., 
\begin{equation}
\label{eqn:state x}
x(t+1) = Ax(t)+Bu^{\star}(t)+w(t),
\end{equation}
where $u^{\star}(t)=\sum_{s\in\U}I_{\V,s}K_s\zeta_s(t)$ with $K_s$ and $\zeta_s (t)$ given by Eqs.~\eqref{eqn:set of DARES K} and \eqref{eqn:dynamics of zeta}, respectively, for all $s\in\U$. In Eqs.~\eqref{eqn:state x hat}-\eqref{eqn:state x}, we set $\hat{x}(0)=x(0)=0$.

Moreover, for our analysis in the sequel, we introduce another control policy $\tilde{u}(t)$ given by 
\begin{equation}
\label{eqn:u tilde}
\tilde{u}_i(t)=\sum_{s\ni i}I_{\{i\},s}\hat{K}_s\tilde{\zeta}_s(t)\quad\forall i\in\V,
\end{equation}
for $t\in\Z_{\ge0}$, where for any $s\in\U$, $\hat{K}_s$ is given by Eq.~\eqref{eqn:set of DARES K hat}, and $\tilde{\zeta}_s(t)$ is given by
\begin{equation}
\label{eqn:dynamics of zeta_tilde}
\tilde{\zeta}_s(t+1) = \sum_{r\rightarrow s}(A_{sr}+B_{sr}\hat{K}_r)\tilde{\zeta}_r(t) + \sum_{w_i\rightarrow s}I_{s,\{i\}}w_i(t),
\end{equation}
with $\tilde{\zeta}_s(0)=\sum_{w_i\rightarrow s}I_{s,\{i\}}x_i(0)=0$. We then let $\tilde{x}(t)$ denote the state of the system in Eq.~\eqref{eqn:overall system} corresponding to $\tilde{u}_i(\cdot)$, for $t\in\Z_{\ge0}$, i.e.,
\begin{equation}
\label{eqn:state x tilde}
\tilde{x}(t+1) = A\tilde{x}(t)+B\tilde{u}(t)+w(t),
\end{equation}
where $\tilde{u}(t)=\sum_{s\in\U}I_{\V,s}\hat{K}_s\tilde{\zeta}_s(t)$ from Eq.~\eqref{eqn:u tilde}, and we set $\tilde{x}(0)=x(0)=0$. Roughly speaking, the auxiliary control policy $\tilde{u}_i(\cdot)$ and the corresponding internal state $\tilde{\zeta}_s(\cdot)$ introduced above allow us to decompose the suboptimality gap $\hat{J}-J_{\star}$ of the control policy $\hat{u}(\cdot)$ into two terms that are due to $\hat{K}_s$ and $\hat{\zeta}_s(\cdot)$, respectively, for all $s\in\V$. We then have the following result; the proof follows directly from \cite[Lemma~14]{lamperski2015optimal} and is thus omitted. Note that Lemma~\ref{lemma:zeta_s1 and zeta_s2 indept} is a consequence of the partially nested information structure and the structure of the information graph described in Section~\ref{sec:dist LQR known matrices}.
\begin{lemma}
\label{lemma:zeta_s1 and zeta_s2 indept}
For any $t\in\Z_{\ge0}$, the following hold: (a) $\E[\tilde{\zeta}_s(t)]=0$, for all $s\in\U$; (b) $\tilde{x}(t)=\sum_{s\in\U}I_{\V,s}\tilde{\zeta}_s(t)$;
(c) $\tilde{\zeta}_{s_1}(t)$ and $\tilde{\zeta}_{s_2}(t)$ are independent for all $s_1,s_2\in\U$ with $s_1\neq s_2$.
\end{lemma}

Using the above notations, the cost of the optimization problem in \eqref{eqn:dis LQR obj} corresponding to the control policy $\hat{u}(\cdot)$ (i.e., $\hat{J}$) can be written as
\begin{equation}
\label{eqn:J hat}
\hat{J} = \limsup_{T\to\infty}\E\Big[\frac{1}{T}\sum_{t=0}^{T-1}\big(\hat{x}(t)^{\top}Q\hat{x}(t)+\hat{u}(t)^{\top}R\hat{u}(t)\big)\Big],
\end{equation}
where we use $\limsup$ instead of $\lim$ since the limit may not exist. Furthermore, we let $\tilde{J}$ denote the cost of the optimization problem in \eqref{eqn:dis LQR obj} corresponding to the control policy $\tilde{u}(\cdot)$ given in Eq.~\eqref{eqn:u tilde},\footnote{We will show in Proposition~\ref{prop:J_tilde minus J} that the limit in Eq.~\eqref{eqn:J tilde} exists.}
\begin{equation}
\label{eqn:J tilde}
\tilde{J} = \lim_{T\to\infty}\E\Big[\frac{1}{T}\sum_{t=0}^{T-1}\big(\tilde{x}(t)^{\top}Q\tilde{x}(t)+\tilde{u}(t)^{\top}R\tilde{u}(t)\big)\Big].
\end{equation}

Supposing that the estimates $\hat{A}$ and $\hat{B}$ satisfy $\norm{\hat{A}-A}\le\varepsilon$ and $\norm{\hat{B}-B}\le\varepsilon$ with $\varepsilon\in\R_{>0}$, our ultimate goal in this section is to provide an upper bound on the suboptimality gap $\hat{J}-J_{\star}$, where $J_{\star}$ is the optimal cost given by Eq.~\eqref{eqn:opt J}. To this end, we first decompose $\hat{J}-J_{\star}$ into $\tilde{J}-J_{\star}$ and $\hat{J}-\tilde{J}$, and then upper bound $\tilde{J}-J_{\star}$ and $\hat{J}-\tilde{J}$ separately. Such a decomposition of $\hat{J}-J_{\star}$ is enabled by the structure of the control policy $\hat{u}(\cdot)$ described in Section~\ref{sec:control design}. Specifically, one may view $\tilde{J}-J_{\star}$ as the suboptimality due to $\hat{K}_r$ given by Eq.~\eqref{eqn:set of DARES K hat} for all $r\in\U$ , and view $\hat{J}-\tilde{J}$ as the suboptimality due to $\hat{\zeta}_r(t)$ given by Eq.~\eqref{eqn:dynamics of zeta hat} for all $r\in\U$ and for all $t\in\Z_{\ge0}$. Moreover, the suboptimality introduced by $\hat{\zeta}_r(t)$ is due to the fact that the dynamics of $\hat{\zeta}_r(t)$ given in Eq.~\eqref{eqn:dynamics of zeta hat} for all $r\in\U$ are characterized by $\hat{A}$, $\hat{B}$ and $\hat{w}(t)$, where $\hat{w}(t)$ given by Eq.~\eqref{eqn:est w_i(t)} is an estimate of the disturbance $w(t)$ in Eq.~\eqref{eqn:overall system}.

To proceed, we recall from Lemma~\ref{lemma:opt solution} that for any $s\in\U$ that has a self loop, the matrix $A_{ss}+B_{ss}K_s$ is stable, where $K_s$ is given by Eq.~\eqref{eqn:set of DARES K}. We then have from the Gelfand formula that for any $s\in\U$ that has a self loop, there exist $\kappa_s\in\R_{\ge1}$ and $\gamma_s\in\R$ with $\rho(A_{ss}+B_{ss}K_s)<\gamma_s<1$ such that $\norm{(A_{ss}+B_{ss}K_s)^k}\le\kappa_s\gamma_s^k$ for all $k\in\Z_{\ge0}$. For notational simplicity, let us denote
\begin{equation}
\label{eqn:kappa and gamma}
\gamma=\max\big\{\max_{s\in\CR}\gamma_s,\gamma_0\big\},\ \kappa=\max\big\{\max_{s\in\CR}\kappa_s,\kappa_0\},
\end{equation}
where $\CR\subseteq\U$ denotes the set of root nodes in $\U$, and $\kappa_0\in\R_{\ge1}$ and $\gamma_0\in\R$ with $\rho(A)<\gamma_0<1$ are given in Assumption~\ref{ass:stable A}. Thus, we see from Assumption~\ref{ass:stable A} and our above arguments that $\norm{(A_{ss}+B_{ss}K_s)^k}\le\kappa\gamma^k$ for all $s\in\CR$ and for all $k\in\Z_{\ge0}$, and $\norm{A^k}\le\kappa\gamma^k$ for all $k\in\Z_{\ge0}$, where $\kappa\in\Z_{\ge1}$ and $0<\gamma<1$. Moreover, we denote
\begin{equation}
\label{eqn:Gamma}
\begin{split}
\Gamma &= \max\big\{\norm{A},\norm{B},\max_{s\in\U}\norm{P_s},\max_{s\in\U}\norm{K_s}\big\},\\
\tilde{\Gamma} &= \Gamma+1.
\end{split}
\end{equation}
For our analysis in this section, we will make the following assumption; similar assumptions can be found in, e.g., \cite{cohen2019learning,mania2019certainty,dean2020sample}.
\begin{assumption}
\label{ass:cost matrices}
The cost matrices $R$ and $Q$ in \eqref{eqn:dis LQR obj} satisfy that $\sigma_n(R)\ge1$ and $\sigma_m(Q)\ge1$.
\end{assumption}
Note that the above assumption is not more restrictive than assuming that $R$ and $Q$ are positive definite. Specifically, supposing $R\succ0$ and $Q\succ0$, one can assume without loss of generality that $\sigma_n(R)\ge1$ and $\sigma_m(Q)\ge1$. This is because one can check that scaling the objective function in \eqref{eqn:dis LQR obj} by a positive constant does not change $K_r$ in the optimal solution to \eqref{eqn:dis LQR obj} provided in Lemma~\ref{lemma:opt solution}, for any $r\in\U$.

\subsection{Perturbation Bounds on Solutions to Ricatti Equations}\label{sec:perturbation bounds}
Supposing $\norm{\hat{A}-A}\le\varepsilon$ and $\norm{\hat{B}-B}\le\varepsilon$ with $\varepsilon\in\R_{>0}$, in this subsection we aim to provide upper bounds on the perturbations $\norm{\hat{P}_r-P_r}$ and $\norm{\hat{K}_r-K_r}$ for all $r\in\U$, where $P_r$ (resp., $\hat{P}_r$) is given by Eq.~\eqref{eqn:set of DARES P} (resp., Eq.~\eqref{eqn:set of DARES P hat}), and $K_r$ (resp., $\hat{K}_r$) is given by Eq.~\eqref{eqn:set of DARES K} (resp., Eq.~\eqref{eqn:set of DARES K hat}). We note from Lemma~\ref{lemma:opt solution} that for any $r\in\U$ that has a self loop, Eq.~\eqref{eqn:set of DARES P} (resp., Eq.~\eqref{eqn:set of DARES P hat}) reduces to a discrete Ricatti equation in $P_r$ (resp., $\hat{P}_r$). The following results characterize the bounds on $\norm{\hat{P}_r-P_r}$ and $\norm{\hat{K}_r-K_r}$, for all $r\in\U$; the proofs can be found in Appendix~\ref{sec:proofs of DARE bounds}.
\begin{lemma}
\label{lemma:upper bounds on P_r hat and K_r hat self loop}
Suppose Assumptions~\ref{ass:pairs} and \ref{ass:cost matrices} hold, and $\norm{\hat{A}-A}\le\varepsilon$ and $\norm{\hat{B}-B}\le\varepsilon$, where $\varepsilon\in\R_{>0}$. Then, for any $r\in\U$ that has a self loop, the following hold:
\begin{equation}
\label{eqn:P_r hat and P_r self loop}
\norm{\hat{P}_r-P_r}\le6\frac{\kappa^2}{1-\gamma^2}\tilde{\Gamma}^5(1+\sigma_1(R^{-1}))\varepsilon\le\frac{1}{6},
\end{equation}
\begin{equation}
\label{eqn:K_r hat and K_r self loop}
\norm{\hat{K}_r-K_r}\le18\frac{\kappa^2}{1-\gamma^2}\tilde{\Gamma}^8(1+\sigma_1(R^{-1}))\varepsilon\le1,
\end{equation}
and
\begin{equation}
\label{eqn:K_r hat stabilizable}
\norm{(A_{rr}+B_{rr}\hat{K}_r)^k}\le\kappa(\frac{\gamma+1}{2})^k,\ \forall k\ge0,
\end{equation}
under the assumption that 
\begin{equation}
\label{eqn:upper bound on epsilon 1}
\varepsilon\le\frac{1}{768}\frac{(1-\gamma^2)^2}{\kappa^4}\tilde{\Gamma}^{-11}(1+\sigma_1(R^{-1}))^{-2},
\end{equation}
where $P_r$ (resp., $\hat{P}_r$) is given by Eq.~\eqref{eqn:set of DARES P} (resp., Eq.~\eqref{eqn:set of DARES P hat}), $K_r$ (resp., $\hat{K}_r$) is given by Eq.~\eqref{eqn:set of DARES K} (resp., \eqref{eqn:set of DARES K hat}), $\gamma$ and $\kappa$ are defined in \eqref{eqn:kappa and gamma}, and $\tilde{\Gamma}$ is defined in~\eqref{eqn:Gamma}.
\end{lemma}

\begin{lemma}
\label{lemma:upper bounds on P_r and K_r general}
Suppose Assumptions~\ref{ass:pairs} and \ref{ass:cost matrices} hold, and $\norm{\hat{A}-A}\le\varepsilon$ and $\norm{\hat{B}-B}\le\varepsilon$, where $\varepsilon\in\R_{>0}$. Then, for any $r\in\U$ that does not have a self loop, the following hold:
\begin{equation}
\label{eqn:K_r hat and K_r general}
\norm{\hat{K}_r-K_r}\le18\frac{\kappa^2}{1-\gamma^2}\tilde{\Gamma}^8(1+\sigma_1(R^{-1}))(20\tilde{\Gamma}^9\sigma_1(R))^{l_{rs}-1}\varepsilon\le1,
\end{equation}
and
\begin{equation}
\label{eqn:P_r hat and P_r general}
\norm{\hat{P}_r-P_r}\le6\frac{\kappa^2}{1-\gamma^2}\tilde{\Gamma}^5(1+\sigma_1(R^{-1}))(20\tilde{\Gamma}^9\sigma_1(R))^{l_{rs}}\varepsilon\le\frac{1}{6},
\end{equation}
under the assumption that 
\begin{equation}
\label{eqn:upper bound on epsilon 2}
\varepsilon\le\frac{1}{768}\frac{(1-\gamma^2)^2}{\kappa^4}\tilde{\Gamma}^{-11}(1+\sigma_1(R^{-1}))^{-2}(20\tilde{\Gamma}^9\sigma_1(R))^{-D_{\max}},
\end{equation}
where $K_r$ (resp., $\hat{K}_r$) is given by Eq.\eqref{eqn:set of DARES K} (resp., Eq.~\eqref{eqn:set of DARES K hat}), $P_r$ (resp., $\hat{P}_r$) is given by Eq.~\eqref{eqn:set of DARES P} (resp., Eq.~\eqref{eqn:set of DARES P hat}), $\tilde{\Gamma}$ is defined in~\eqref{eqn:Gamma}, $\kappa$ and $\gamma$ are defined in ~\eqref{eqn:kappa and gamma}, $l_{rs}$ is the length of the unique directed path from node $r$ to node $s$ in $\CP(\U,\CH)$ with $s\in\U$ to be the unique root node that is reachable from $r$, and $D_{\max}$ is defined in Eq.~\eqref{eqn:depth of T_i}.
\end{lemma}

Consider any $r\in\U$ with a self loop and suppose Eq.~\eqref{eqn:upper bound on epsilon 1} holds. One can show via Eq.~\eqref{eqn:K_r hat stabilizable} and \cite[Lemma~12]{mania2019certainty} that $\hat{K}_r$ given by Eq.~\eqref{eqn:set of DARES K hat} is also stabilizing for the pair $(\hat{A}_{rr},\hat{B}_{rr})$, i.e., $\hat{A}_{rr}+\hat{B}_{rr}\hat{K}_r$ is stable (see our arguments for \eqref{eqn:L_hat^k minus L_tilde^k} in Appendix~\ref{sec:proofs of bounds on costs} for more details). Moreover, it is well-known (e.g., \cite{bertsekas2017dynamic}) that a stabilizing solution $\hat{P}_r$ to the Ricatti equation in Eq.~\eqref{eqn:set of DARES P hat} exists if and only if $(\hat{A}_{rr},\hat{B}_{rr})$ is stabilizable and $(\hat{A}_{rr},C_{rr})$ (with $Q_{rr}=C_{rr}^TC_{rr}$) is detectable.\footnote{A solution $\hat{P}_r$ to the Ricatti equation in Eq.~(19) is stabilizing if and only if $\hat{A}_{rr}+\hat{B}_{rr}\hat{K}_r$ (with $\hat{K}_r$ given by Eq.~(18)) is stable.} The above arguments together also imply that $(\hat{A}_{rr},\hat{B}_{rr})$ is stabilizable and $(\hat{A}_{rr},C_{rr})$ (with $Q_{rr}=C_{rr}^{\top}C_{rr}$) is detectable for all $r\in\U$, under the assumption on $\varepsilon$ given by Eq.~\eqref{eqn:upper bound on epsilon 1}.

\subsection{Perturbation Bounds on Costs}\label{sec:perturb bound on cost}
Suppose $\norm{\hat{A}-A}\le\varepsilon$ and $\norm{\hat{B}-B}\le\varepsilon$, where $\varepsilon\in\R_{>0}$. In this subsection, we aim to provide an upper bound on $\hat{J}-J_{\star}$ that scales linearly with $\epsilon$, where $J_{\star}$ and $\hat{J}$ are given by Eqs.~\eqref{eqn:opt J} and \eqref{eqn:J hat}, respectively. 
\begin{lemma}
\label{lemma:upper bound on zeta_tilde}
Suppose Assumptions~\ref{ass:pairs} and \ref{ass:cost matrices} hold, and $\norm{\hat{A}-A}\le\varepsilon$ and $\norm{\hat{B}-B}\le\varepsilon$, where $\varepsilon\in\R_{>0}$ satisfies \eqref{eqn:upper bound on epsilon 2}. Then, for any $s\in\U$,
\begin{equation}
\label{eqn:upper bound on cov zeta_tilde}
\lim_{t\to\infty}\E\Big[\tilde{\zeta}_s(t)\tilde{\zeta}_s(t)^{\top}\Big]\preceq\frac{4p\sigma_w^2\tilde{\Gamma}^{4D_{\max}}\kappa^2}{1-\gamma^2}I,
\end{equation}
where $p=|\V|$, $\kappa$ and $\gamma$ are defined in \eqref{eqn:kappa and gamma}, $\tilde{\Gamma}$ is defined in~\eqref{eqn:Gamma}, and $D_{\max}$ is defined in Eq.~\eqref{eqn:depth of T_i}.
\end{lemma}

For our analysis in the sequel, we further define $\tilde{P}_r$ recursively, for all $r\in\U$, as
\begin{equation}
\label{eqn:set of DARES P tilde}
\tilde{P}_r =Q_{rr}+\hat{K}_r^{\top}R_{rr}\hat{K}_r+(A_{sr}+B_{sr}\hat{K}_r)^{\top}\tilde{P}_s(A_{sr}+B_{sr}\hat{K}_r),
\end{equation}
where $\hat{K}_r$ is given by Eq.~\eqref{eqn:set of DARES K hat}, and $s\in\U$ is the unique node such that $r\rightarrow s$. We then have the following result, which gives an upper bound on $\tilde{J}-J_{\star}$.
\begin{proposition}
\label{prop:J_tilde minus J}
Suppose Assumption~\ref{ass:pairs} and \ref{ass:cost matrices} hold, and $\norm{\hat{A}-A}\le\varepsilon$ and $\norm{\hat{B}-B}\le\varepsilon$, where $\varepsilon\in\R_{>0}$ satisfies \eqref{eqn:upper bound on epsilon 2}. It holds that
\begin{equation}
\label{eqn:exp for J_tilde}
\tilde{J}=\sigma_w^2\sum_{\substack{i\in\V\\ w_i\rightarrow s}}\tr\big(I_{\{i\},s}\tilde{P}_sI_{s,\{i\}}\big),
\end{equation}
where $\tilde{J}$ is defined in Eq.~\eqref{eqn:J tilde}. Moreover, consider the optimal cost $J_{\star}$ given by Eq.~\eqref{eqn:opt J}. For any $\varphi\in\R_{>0}$,
\begin{equation*}
\label{eqn:upper bound on J_tilde minus J}
\tilde{J}-J_{\star}\le\frac{72\kappa^4\sigma_w^2npq}{(1-\gamma^2)^2}\tilde{\Gamma}^{4D_{\max}+8}(\Gamma^3+\sigma_1(R))(1+\sigma_1(R^{-1}))(20\tilde{\Gamma}^9\sigma_1(R))^{D_{\max}}\varepsilon+\varphi,
\end{equation*}
where $\kappa$ and $\gamma$ are defined in \eqref{eqn:kappa and gamma}, $p=|\V|$ and $q=|\U|$, $D_{\max}$ is defined in Eq.~\eqref{eqn:depth of T_i}, and $\Gamma$ and $\tilde{\Gamma}$ are defined in~\eqref{eqn:Gamma}.
\end{proposition}

Next, we aim to provide an upper bound on $\hat{J}-\tilde{J}$. We first prove the following result.
\begin{lemma}
\label{lemma:upper bound on state norms}
Suppose Assumptions~\ref{ass:pairs} and \ref{ass:cost matrices} hold, and $\norm{\hat{A}-A}\le\varepsilon$ and $\norm{\hat{B}-B}\le\varepsilon$, where $\varepsilon$ satisfies \eqref{eqn:upper bound on epsilon 2}. Then, for any $s\in\U$ and for any $t\in\Z_{\ge0}$,
\begin{equation}
\label{eqn:upper bound on norm of zeta_tilde}
\E\Big[\norm{\tilde{\zeta}_s(t)}^2\Big]\le\frac{4np\sigma_w^2\tilde{\Gamma}^{4D_{\max}}\kappa^2}{1-\gamma^2},
\end{equation}
where $\tilde{\zeta}_s(t)$ is given in Eq.~\eqref{eqn:dynamics of zeta_tilde}, $p=|\V|$, $\kappa$ and $\gamma$ are defined in \eqref{eqn:kappa and gamma}, $\tilde{\Gamma}$ is defined in~\eqref{eqn:Gamma}, and $D_{\max}$ is defined in Eq.~\eqref{eqn:depth of T_i}. Moreover, for any $t\in\Z_{\ge0}$,
\begin{equation}
\label{eqn:upper bound on norm of x_tilde}
\E\Big[\norm{\tilde{x}(t)}^2\Big]\le\frac{4npq^2\sigma_w^2\tilde{\Gamma}^{4D_{\max}}\kappa^2}{1-\gamma^2},
\end{equation}
and
\begin{equation}
\label{eqn:upper bound on norm of u_tilde}
\E\Big[\norm{\tilde{u}(t)}^2\Big]\le\frac{4npq^2\sigma_w^2\tilde{\Gamma}^{4D_{\max}+2}\kappa^2}{1-\gamma^2},
\end{equation}
where $\tilde{x}(t)$ and $\tilde{u}(t)$ are given by Eqs.~\eqref{eqn:state x tilde} and \eqref{eqn:u tilde}, respectively, and $q=|\U|$.
\end{lemma}

For notational simplicity in the sequel, let us denote
\begin{equation}
\label{eqn:aux parameters}
\begin{split}
\zeta_b&=\sqrt{\frac{4np\sigma_w^2\tilde{\Gamma}^{4D_{\max}}\kappa^2}{1-\gamma^2}},\\
\bar{\varepsilon}&=\frac{(1-\gamma)^3}{768\kappa^4pq}(\tilde{\Gamma}+1)^{-2}\tilde{\Gamma}^{-9}(1+\sigma_1(R^{-1}))^{-2}(20(\tilde{\Gamma}+1)^2\tilde{\Gamma}^7\sigma_1(R))^{-D_{\max}}.
\end{split}
\end{equation}
We then have the following results.
\begin{lemma}
\label{lemma:u_hat minus u_tilde and x_hat minus x_tilde}
Suppose Assumptions~\ref{ass:pairs}-\ref{ass:cost matrices} hold, and $\norm{\hat{A}-A}\le\bar{\varepsilon}$ and $\norm{\hat{B}-B}\le\bar{\varepsilon}$, where $\bar{\varepsilon}$ is defined in \eqref{eqn:aux parameters}. Then, for all $t\in\Z_{\ge0}$,
\begin{equation}
\label{eqn:u_hat minue u_tilde}
\E\Big[\norm{\hat{u}(t)-\tilde{u}(t)}^2\Big]\le\bigg(\frac{58\kappa^2(\tilde{\Gamma}+1)^{2D_{\max}+3}p^2q^2}{(1-\gamma)^2}\zeta_b\bar{\varepsilon}\bigg)^2,
\end{equation}
and 
\begin{equation}
\label{eqn:x_hat minus x_tilde}
\E\Big[\norm{\hat{x}(t)-\tilde{x}(t)}^2\Big]\le\bigg(\frac{58\kappa^3\Gamma(\tilde{\Gamma}+1)^{2D_{\max}+3}p^2q^2}{(1-\gamma)^3}\zeta_b\bar{\varepsilon}\bigg)^2,
\end{equation}
where $\hat{u}(t)$ (resp., $\tilde{u}(t)$) is given by Eq.~\eqref{eqn:control policy} (resp., Eq.~\eqref{eqn:u tilde}), $\hat{x}(t)$ (resp., $\tilde{x}(t)$) is given by Eq.~\eqref{eqn:state x hat} (resp., Eq.~\eqref{eqn:state x tilde}), $\Gamma$ and $\tilde{\Gamma}$ are defined in~\eqref{eqn:Gamma}, $\kappa$ and $\gamma$ are defined in \eqref{eqn:kappa and gamma}, $p=|\V|$ and $q=|\U|$, $D_{\max}$ is defined in Eq.~\eqref{eqn:depth of T_i}, and $\zeta_b$ is defined in \eqref{eqn:aux parameters}.
\end{lemma}

\begin{corollary}
\label{coro:upper bound on x_hat and u_hat}
Suppose Assumptions~\ref{ass:pairs}-\ref{ass:cost matrices} hold. and $\norm{\hat{A}-A}\le\bar{\varepsilon}$ and $\norm{\hat{B}-B}\le\bar{\varepsilon}$, where $\bar{\varepsilon}$ is defined in \eqref{eqn:aux parameters}. Then, for all $t\in\Z_{\ge0}$,
\begin{equation}
\label{eqn:upper bound on x_hat(t) coro}
\E\Big[\norm{\hat{x}(t)}^2\Big]\le\bigg(\frac{58\kappa^3\Gamma(\tilde{\Gamma}+1)^{2D_{\max}+3}p^2q^2}{(1-\gamma)^3}\zeta_b\bar{\varepsilon}+q\zeta_b\bigg)^2,
\end{equation}
and 
\begin{equation}
\label{eqn:upper bound on u_hat(t) coro}
\E\Big[\norm{\hat{u}(t)}^2\Big]\le\bigg(\frac{58\kappa^2(\tilde{\Gamma}+1)^{2D_{\max}+3}p^2q^2}{(1-\gamma)^2}\zeta_b\bar{\varepsilon}+q\tilde{\Gamma}\zeta_b\bigg)^2,
\end{equation}
where $\hat{u}(t)$ is given by Eq.~\eqref{eqn:control policy}, $\hat{x}(t)$ is given by Eq.~\eqref{eqn:state x hat}, $\Gamma$ and $\tilde{\Gamma}$ are defined in~\eqref{eqn:Gamma}, $\kappa$ and $\gamma$ are defined in \eqref{eqn:kappa and gamma}, $p=|\V|$ and $q=|\U|$, $D_{\max}$ is defined in Eq.~\eqref{eqn:depth of T_i}, and $\zeta_b$ is given by \eqref{eqn:aux parameters}.
\end{corollary}
\begin{proof}
Note that
\begin{align*}
\sqrt{\E\Big[\norm{\hat{x}(t)}^2\Big]}&=\sqrt{\E\Big[\norm{\hat{x}(t)-\tilde{x}(t)+\tilde{x}(t)}^2\Big]}\\
&\le\sqrt{\E\Big[\norm{\hat{x}(t)-\tilde{x}(t)}^2\Big]}+\sqrt{\E\Big[\norm{\tilde{x}(t)}^2\Big]},
\end{align*}
where the inequality follows from Lemma~\ref{lemma:CS inequallity}. The proof of \eqref{eqn:upper bound on x_hat(t) coro} now follows directly from Lemmas~\ref{lemma:upper bound on state norms}-\ref{lemma:u_hat minus u_tilde and x_hat minus x_tilde}. Similarly, we can prove \eqref{eqn:upper bound on u_hat(t) coro}.
\end{proof}

\begin{proposition}
\label{prop:upper bound on J_hat minue J_tilde}
Suppose Assumptions~\ref{ass:pairs}-\ref{ass:cost matrices} hold, and $\norm{\hat{A}-A}\le\bar{\varepsilon}$ and $\norm{\hat{B}-B}\le\bar{\varepsilon}$, where $\bar{\varepsilon}$ is defined in Eq.~\eqref{eqn:aux parameters}. Then, for $\hat{J}$ and $\tilde{J}$ defined in Eqs.~\eqref{eqn:J hat} and \eqref{eqn:J tilde}, respectively,
\begin{equation}
\label{eqn:upper bound on J_hat minus J_tilde}
\hat{J}-\tilde{J}\le 696\frac{\kappa^6\sigma_w^2np^4q^3}{(1-\gamma)^4(1-\gamma^2)}\tilde{\Gamma}^{4D_{\max}+2}(\tilde{\Gamma}+1)^{2D_{\max}+3}(\sigma_1(Q)+\sigma_1(R))\bar{\varepsilon},
\end{equation}
where $\kappa$ and $\gamma$ are defined in \eqref{eqn:kappa and gamma}, $p=|\V|$ and $q=|\U|$, $\tilde{\Gamma}$ is defined in~\eqref{eqn:Gamma}, and $D_{\max}$ is defined in Eq.~\eqref{eqn:depth of T_i}.
\end{proposition}

Suppose $\norm{\hat{A}-A}\le\varepsilon$ and $\norm{\hat{B}-B}\le\varepsilon$ with $\varepsilon\in\R_{>0}$. We see from the results in Propositions~\ref{prop:J_tilde minus J}-\ref{prop:upper bound on J_hat minue J_tilde} that $\hat{J}-J_{\star}\le C\varepsilon$, if $\varepsilon\le C_0$, where $C$ and $C_0$ are constants that depend on the problem parameters.

\subsection{Sample Complexity Result}\label{sec:sample complexity}
We are now in place to present the sample complexity result for learning decentralized LQR with the partially nested information structure described in Section~\ref{sec:dist LQR known matrices}.
\begin{theorem}
\label{thm:sample complexity}
Suppose Assumptions~\ref{ass:pairs}-\ref{ass:cost matrices} hold, and Algorithm~\ref{algorithm:least squares} is used to obtain $\hat{A}$ and $\hat{B}$. Moreover, suppose $\norm{A}\le\vartheta$ and $\norm{B}\le\vartheta$, where $\vartheta\in\R_{>0}$, and $D_{\max}\le D$, where $D_{\max}$ is defined in Eq.~\eqref{eqn:depth of T_i} and $D$ is a universal constant. Consider any $\delta>0$. Let the input parameters to Algorithm~\ref{algorithm:least squares} satisfy $N\ge\alpha/\bar{\varepsilon}$ and $\lambda\ge\underline{\sigma}^2/40$, where
\begin{equation*}
z_b=\frac{5\kappa_0}{1-\gamma_0}\overline{\sigma}\sqrt{(\norm{B}^2m+m+n)\log\frac{4N}{\delta}},
\end{equation*}
and 
\begin{equation*}
\alpha=\frac{160}{\underline{\sigma}^2}\bigg(2n\sigma_w^2(n+m)\log\frac{N+z_b^2/\lambda}{\delta}+\lambda n\vartheta^2\bigg),
\end{equation*}
where $\kappa_0$ and $\gamma_0$ are given in Assumption~\ref{ass:stable A}, $\underline{\sigma}=\min\{\sigma_w,\sigma_u\}$, $\overline{\sigma}=\max\{\sigma_w,\sigma_u\}$, and $\bar{\varepsilon}$ is defined in \eqref{eqn:aux parameters}. Then, with probability at least $1-\delta$,
\begin{equation}
\label{eqn:J_hat minus J_star}
\hat{J}-J_{\star}\le C_1\frac{\kappa^6\sigma_w^2np^4q^3}{(1-\gamma^2)^2}\tilde{\Gamma}^{11D+5}(\tilde{\Gamma}+1)^{2D+3}(\Gamma^3+\sigma_1(R)+\sigma_1(Q))\sigma_1(R)^{D}\sqrt{\frac{\alpha}{N}},
\end{equation}
where $\hat{J}$ and $J_{\star}$ are given in Eqs.~\eqref{eqn:J hat} and \eqref{eqn:opt J}, respectively, $C_1$ is a universal constant, $\kappa$ and $\gamma$ are defined in \eqref{eqn:kappa and gamma}, and $\Gamma$ and $\tilde{\Gamma}$ are defined in~\eqref{eqn:Gamma}, $p=|\V|$, and $q=|\U|$.
\end{theorem}
\begin{proof}
Note that the results in Propositions~\ref{prop:J_tilde minus J}-\ref{prop:upper bound on J_hat minue J_tilde} hold, if $\norm{\hat{A}-A}\le\bar{\varepsilon}$ and $\norm{\hat{B}-B}\le\bar{\varepsilon}$ with $\bar{\varepsilon}$ given in \eqref{eqn:aux parameters}. Thus, letting $N\ge\frac{\alpha}{\bar{\varepsilon}}$ and $\lambda\ge\underline{\sigma}^2/40$, one can first check that $N\ge200(n+m)\log\frac{48}{\delta}$, and then obtain from Proposition~\ref{prop:upper bound on est error} that with probability at least $1-\delta$, $\hat{A}$ and $\hat{B}$ returned by Algorithm~\ref{algorithm:least squares} satisfy that $\norm{\hat{A}-A}\le\bar{\varepsilon}$ and $\norm{\hat{B}-B}\le\bar{\varepsilon}$. Now, noting that $D_{\max}\le D$, where $D$ is a universal constant, and setting $\varphi=1/\sqrt{N}$ in Proposition~\ref{prop:J_tilde minus J}, one can then show via Propositions~\ref{prop:J_tilde minus J}-\ref{prop:upper bound on J_hat minue J_tilde} that \eqref{eqn:J_hat minus J_star} holds with probability at least $1-\delta$.
\end{proof}

Thus, we have shown a $\tilde{\CO}(1/\sqrt{N})$ end-to-end sample complexity result for learning decentralized LQR with the partially nested information structure. In other words, we relate the number of data samples used for estimating the system model to the performance of the control policy proposed in Section~\ref{sec:control design}. Note that our result in Theorem~\ref{thm:sample complexity} matches with the $\CO(1/\sqrt{N})$ sample complexity result (up to logarithm factors in $N$) provided in \cite{dean2020sample} for learning centralized LQR without any information constraints. Also note that the sample complexity for learning centralized LQR has been improved to $\CO(1/N)$ in \cite{mania2019certainty}. Specifically, the authors in \cite{mania2019certainty} showed that the gap between the cost $\hat{J}$ corresponding to the control policy they proposed and the optimal cost $J_{\star}$ is upper bounded by $\CO(\varepsilon^2)$, where $\norm{\hat{A}-A}\le\varepsilon$ and $\norm{\hat{B}-B}\le\varepsilon$, if $\varepsilon$ is sufficiently small. Due to the additional challenges introduced by the information constraints on the controllers (see our discussions at the end of Section~\ref{sec:control design}), we leave investigating the possibility of improving our sample complexity result in Theorem~\ref{thm:sample complexity} for future work.

\section{Numerical Results}\label{sec:numerical results}
In this section, we illustrate the sample complexity result provided in Theorem~\ref{thm:sample complexity} with numerical experiments, where the numerical experiments are conducted based on Example~\ref{exp:running example}. Specifically, we consider the LTI system given by Eq.~\eqref{eqn:LTI in exp} with the corresponding directed graph and information graph given by Fig.~\ref{fig:directed graph} and Fig.~\ref{fig:info graph}, respectively. Under the sparsity pattern of $A$ and $B$ specified in Eq.~\eqref{eqn:LTI in exp}, we generate the nonzero entries in $A\in\R^{3\times 3}$ and $B\in\R^{3\times  3}$ independently by the Gaussian distribution $\CN(0,1)$ while satisfying Assumption~\ref{ass:stable A}. We set the covariance of the zero-mean white Gaussian noise process $w(t)$ to be $I$, and set the cost matrices to be $Q=2I$ and $R=5I$. Moreover, we set the input sequence used in the system identification algorithm (Algorithm~\ref{algorithm:least squares}) to be $u(t)\overset{\text{i.i.d.}}{\sim}\CN(0,I)$ for all $t\in\{0,\dots,N-1\}$. In order to approximate the value of $\hat{J}$ defined in Eq.~\eqref{eqn:J hat}, we simulate the system using Algorithm~\ref{algorithm:control design} for $T=2000$ and obtain $\hat{J}\approx\frac{1}{T}\sum_{t=0}^{T-1}\big(\tilde{x}(t)^{\top}Q\tilde{x}(t)+\tilde{u}(t)^{\top}R\tilde{u}(t)\big)$. Fixing the randomly generated matrices $A$ and $B$ described above, the numerical results presented in this section are obtained by averaging over $100$ independent experiments.

\begin{figure}[htbp]
\centering
\subfloat[a][Estimation Error vs. $N$]{
\includegraphics[width=0.40\linewidth]{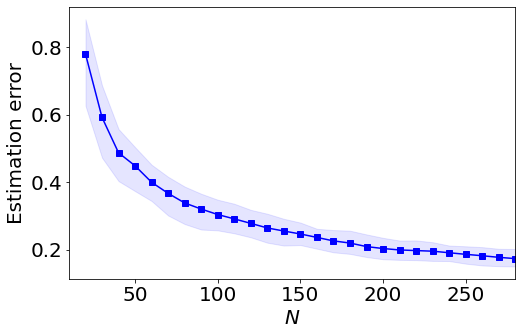}}
\subfloat[b][Cost Suboptimality vs. $N$]{
\includegraphics[width=0.38\linewidth]{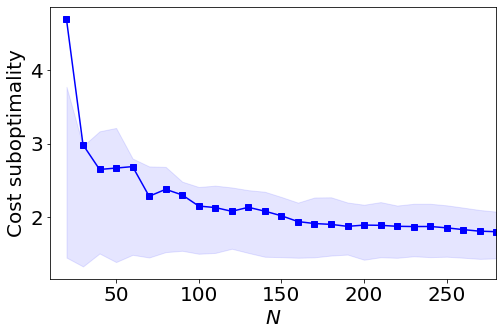}}
\caption{Both the performance of Algorithm~\ref{algorithm:least squares} and the performance of Algorithm~\ref{algorithm:control design} are plotted against the number of data samples used for estimating the system model, where shaded regions display quartiles.}
\label{fig:numerical results}
\end{figure}

In Fig.~\ref{fig:numerical results}(a), we plot the estimation error $\big\lVert\begin{bmatrix}\hat{A} & \hat{B}\end{bmatrix}-\begin{bmatrix}A & B\end{bmatrix}\big\rVert$ corresponding to Algorithm~\ref{algorithm:least squares} when we range the number of the data samples used in Algorithm~\ref{algorithm:least squares} from $N=20$ to $N=280$. Similarly, in Fig.~\ref{fig:numerical results}(b) we plot the curve corresponding to the cost suboptimality $\hat{J}-J_{\star}$, where $J_{\star}$ is obtained by the closed-form expression given in Eq.~\eqref{eqn:opt J}. According to Fig.~\ref{fig:numerical results}, we observe that the estimation error and the cost suboptimality share a similar dependency pattern on $N$. The similar dependency on $N$ aligns with the results shown in Proposition~\ref{prop:upper bound on est error} and Theorem~\ref{thm:sample complexity} that both the estimation error and the cost suboptimality scale as $\tilde{\CO}(1/\sqrt{N})$, which is a consequence of the results shown in Propositions~\ref{prop:J_tilde minus J}-\ref{prop:upper bound on J_hat minue J_tilde} that the cost suboptimality scales linearly with the estimation error. The results presented in Fig.~\ref{fig:numerical results} then also imply that our suboptimality results provided in Propositions~\ref{prop:J_tilde minus J}-\ref{prop:upper bound on J_hat minue J_tilde} can be tight for certain instances of the problem. Finally, we observe from the shaded regions in Fig.~\ref{fig:numerical results} that the cost suboptimality is more sensitive to the randomness introduced by the random input $u(t)\overset{\text{i.i.d.}}{\sim}\CN(0,I)$ for $t\in\{0,\dots,N-1\}$ and the noise $w(t)\overset{\text{i.i.d.}}{\sim}\CN(0,I)$ for $t\in\Z_{\ge0}$, when we run the $100$ experiments described above. This is potentially due to the fact that we approximated the cost suboptimality as $\frac{1}{T}\sum_{t=0}^{T-1}\big(\hat{x}(t)^{\top}Q\hat{x}(t)+\hat{u}(t)^{\top}R\hat{u}(t)\big)-J_{\star}$ with $T=2000$.

\section{Conclusion}
We considered the problem of control policy design for decentralized state-feedback linear quadratic control with a partially nested information structure, when the system model is unknown. We took a model-based learning approach consisting of two steps. First, we estimated the unknown system model from a single system trajectory of finite length, using least squares estimation. Next, we designed a control policy based on the estimated system model, which satisfies the desired information constraints. We showed that the suboptimality gap between our control policy and the optimal decentralized control policy (designed using accurate knowledge of the system model) scales linearly with the estimation error of the system model. Combining the above results, we provided an end-to-end sample complexity of learning decentralized controllers for state-feedback linear quadratic control with a partially nested information structure.

\bibliography{bibliography}

\section*{Appendix}
\appendix

\section{Proofs for Least Squares Estimation of System Matrices}\label{sec:proofs ls}
\subsection{Proof of Lemma~\ref{lemma:prob of CE}}
First, let us consider $\CE_w$. We can apply Lemma~\ref{lemma:upper bound on gaussian w} with $\delta_w=\frac{\delta}{4}$ and obtain that $\Prob(\CE_w)\ge1-\frac{\delta}{4}$. Similarly, recalling from Algorithm~\ref{algorithm:least squares} that $u(t)\overset{\text{i.i.d.}}{\sim}\CN(0,\sigma_u^2I)$ for all $t\in\{0,\dots,N-1\}$, we have from Lemma~\ref{lemma:upper bound on gaussian w} that $\Prob(\CE_u)\ge1-\frac{\delta}{4}$. Next, let us consider $\CE_{\Theta}$. Applying Lemma~\ref{lemma:est error of Theta_hat} with $\delta_{\Theta}=\frac{\delta}{4}$, we obtain $\Prob(\CE_{\Theta})\ge1-\frac{\delta}{4}$.

Finally, let us consider $\CE_z$. Consider the sequence of random vectors $\{z(t)\}_{t\ge0}$ and the filtration $\{\F_t\}_{t\ge0}$, where $z(t)\in\R^{n+m}$ is defined in \eqref{eqn:Theta_i and z_Ni}, and $\F_t=\sigma(x(0),u(0),\dots,x(t),u(t))$ for all $t\in\Z_{\ge0}$. For any $t\in\{1,\dots,N-1\}$, we note from Eq.~\eqref{eqn:overall system} that $x(t)$ is conditionally Gaussian on $x(t-1)$ and $u(t-1)$, with
\begin{equation*}
\E\big[x(t)x(t)^{\top}|\F_{t-1}\big]\succeq\E\big[w(t-1)w(t-1)^{\top}\big]=\sigma_w^2I.
\end{equation*}
Note again from Algorithm~\ref{algorithm:least squares} that $u(t)\overset{\text{i.i.d.}}{\sim}\CN(0,\sigma^2_uI)$ for all $t\in\{0,\dots,N-1\}$, and that $u(t)$ is assumed to be independent of $w(t)$ for all $t\in\{0,\dots,N-1\}$. We then see that $z(t)$ is also conditionally Gaussian on $x(t-1)$ and $u(t-1)$, with
\begin{align}\nonumber
\E\big[z(t)z(t)^{\top}|\F_{t-1}\big]&=\begin{bmatrix}\E[x(t)x(t)^{\top}|\F_{t-1}] & 0\\ 0 & \E[u(t)u(t)^{\top}]\end{bmatrix}\\\nonumber
&=\begin{bmatrix}\E[x(t)x(t)^{\top}|\F_{t-1}] & 0\\ 0 & \sigma_u^2I\end{bmatrix}\succeq\underline{\sigma}^2I.
\end{align}
Now, we can apply Lemma~\ref{lemma:lower bound on sum over z_k} with $\{z(t)\}_{t\ge0}$ and $\{\F_t\}_{t\ge0}$ described above, and let $\delta_z=\frac{\delta}{4p}$. Since $N\ge200(n+m)\log\frac{48}{\delta}$, we have from Lemma~\ref{lemma:lower bound on sum over z_k} that $\sum_{t=0}^{N-1}z(t)z(t)^{\top}\succeq\frac{(N-1)\underline{\sigma}^2}{40}I$ holds with probability at least $1-\frac{\delta}{4}$. Combining the above arguments together and applying a union bound over the events $\CE_w$, $\CE_{\psi}$, $\CE_{\Theta}$ and $\CE_z$, we complete the proof of the lemma. \hfill\qed

\subsection{Proof of Lemma~\ref{lemma:upper bound on norm of z(t)}}
First, considering any $t\in\{0,\dots,N-1\}$, we denote $\psi(t)=Bu(t)+w(t)$. We see from \eqref{eqn:events} that
\begin{align}\nonumber
\norm{\psi(k)}&\le\norm{B}\sigma_u\sqrt{5m\log\frac{4N}{\delta}}+\sigma_w\sqrt{5n\log\frac{4N}{\delta}}\\\nonumber
&\le\overline{\sigma}(\sqrt{m\norm{B}^2}+\sqrt{n})\sqrt{5\log\frac{4N}{\delta}}\\\nonumber
&\le\overline{\sigma}\sqrt{10(m\norm{B}^2+n)\log\frac{4N}{\delta}}.
\end{align}
Next, for any $t\in\{1,\dots,N-1\}$, we see from Eq.~\eqref{eqn:overall system} that 
\begin{equation*}
\label{eqn:state x(k)}
x(t)=A^tx(0)+\sum_{k=0}^{t-1}A^{t-1-k}\psi(k),
\end{equation*}
where recall that we assumed previously that $x(0)=0$. Since $A$ is stable from Assumption~\ref{ass:stable A}, we know that $\norm{A^k}\le\kappa_0\gamma_0^k$ for all $k\in\Z_{\ge0}$, where $\kappa_0\ge1$ and $\rho(A)<\gamma_0<1$. It now follows from the above arguments that
\begin{equation*}
\label{eqn:upper bound on norm of x(k)}
\norm{x(t)}\le\frac{\kappa_0}{1-\gamma_0}\overline{\sigma}\sqrt{10(m\norm{B}^2+n)\log\frac{4N}{\delta}},
\end{equation*}
for all $t\in\{0,\dots,N-1\}$. Noting that $\norm{z(t)}\le\norm{x(t)}+\norm{u(t)}$, we then have
\begin{align}\nonumber
\norm{z(t)}&\le\frac{\kappa_0}{1-\gamma_0}\overline{\sigma}\sqrt{10(m\norm{B}^2+n)\log\frac{4N}{\delta}}+\sigma_u\sqrt{5m\log\frac{4N}{\delta}}\\\nonumber
&\le\frac{5\kappa_0}{1-\gamma_0}\overline{\sigma}\sqrt{(\norm{B}^2m+m+n)\log\frac{4N}{\delta}},
\end{align}
for all $t\in\{0,\dots,N-1\}$.\hfill\qed

\subsection{Proof of Proposition~\ref{prop:upper bound on est error}}
First, we see from Eq.~\eqref{eqn:events} that on the event $\CE$ defined in Eq.~\eqref{eqn:good event}, the following holds: 
\begin{equation}
\label{eqn:upper bound on V_i}
\tr\big(\Delta(N)^{\top}V(N)\Delta(N)\big)\le 4\sigma_w^2n\log\bigg(\frac{4n}{\delta}\frac{\det(V(N))}{\det(\lambda I)}\bigg)+4\lambda n\vartheta^2,
\end{equation}
where recall that $V(N)=\lambda I +\sum_{t=0}^{N-1}z(t)z(t)^{\top}$ and $\Delta(N)=\Theta-\hat{\Theta}(N)$, and where $\Theta$ and $z(t)$ are defined in \eqref{eqn:Theta_i and z_Ni}, and $\hat{\Theta}(N)$ is given by~\eqref{eqn:least squares approach}. To obtain \eqref{eqn:upper bound on V_i}, we also use the fact that $\norm{\Theta}^2\le\norm{A}^2+\norm{B}^2\le2\vartheta^2$, which implies via $\Theta\in\R^{n\times(n+m)}$ that $\norm{\Theta}_F^2\le2n\vartheta^2$. Moreover, we have from Eq.~\eqref{eqn:events} that on the event $\CE$, the following holds:
\begin{equation*}
\label{eqn:lower bound on V_i}
V(N)\succeq\lambda I+\frac{(N-1)\underline{\sigma}^2}{40}I\succeq\frac{N\underline{\sigma}^2}{40}I,
\end{equation*}
where the second inequality follows from the choice of $\lambda$. Combining the above arguments together, one can show that 
\begin{equation*}
\label{eqn:upper bound on Delta_i}
\norm{\Delta(N)}^2\le\frac{160}{N\underline{\sigma}}\bigg(n\sigma_w^2\log\bigg(\frac{4n}{\delta}\frac{\det(V(N))}{\det(\lambda I)}\bigg)+\lambda n\vartheta^2\bigg).
\end{equation*}
Noting from Lemma~\ref{lemma:upper bound on norm of z(t)} that $\norm{z(t)}\le z_b$ for all $t\in\{0,\dots,N-1\}$, one can use similar arguments to those for \cite[Lemma~37]{cassel2020logarithmic}, and show that 
\begin{equation*}
\label{eqn:log ratio of det}
\log\frac{\det(V(N))}{\det(\lambda I)}\le(n+m)\log(N+\frac{z_b^2}{\lambda}).
\end{equation*}
Noting that $N\ge200(n+m)\log\frac{48}{\delta}$, we have $N\ge4n$. It then follows that 
\begin{align}\nonumber
\norm{\Delta(N)}^2&\le\frac{160}{N\underline{\sigma}^2}\bigg(n\sigma_w^2\log\frac{N}{\delta}+n\sigma_w^2(n+m)\log(N+\frac{z_b^2}{\lambda})+\lambda n\vartheta^2\bigg)\\\nonumber
&\le\frac{160}{N\underline{\sigma}^2}\bigg(2n\sigma_w^2(n+m)\log\frac{N+z^2_b/\lambda}{\delta}+\lambda n\vartheta^2\bigg).
\end{align}
Noting that Algorithm~\ref{algorithm:least squares} extracts $\hat{A}$ and $\hat{B}$ from $\hat{\Theta}(N)$, i.e., $\hat{\Theta}(N)=\begin{bmatrix}\hat{A} & \hat{B}\end{bmatrix}$, one can show that $\norm{\hat{A}-A}\le\norm{\hat{\Theta}(N)-\Theta}$ and $\norm{\hat{B}-B}\le\norm{\hat{\Theta}(N)-\Theta}$. Finally, since we know from Lemma~\ref{lemma:prob of CE} that $\Prob(\CE)\ge1-\delta$, we conclude that $\norm{\hat{A}-A}\le\varepsilon_0$ and $\norm{\hat{B}-B}\le\varepsilon_0$ hold with probability at least $1-\delta$.\hfill\qed

\section{Proof for Algorithm~\ref{algorithm:control design}}\label{sec:proof of control policy alg}
\subsection{Proof of Proposition~\ref{prop:alg1 feasible}}
To prove part~(a), we use an induction on $t=0,\dots,T-1$. For the base case $t=0$, we see directly from line~4 in Algorithm~\ref{algorithm:control design} and Eq.~\eqref{eqn:initial M_i} that $\M_i$ satisfies Eq.~\eqref{eqn:memory of the alg t-1} (with $t=0$) at the beginning of iteration $0$ of the for loop in lines~4-14 of the algorithm. For the induction step, consider any $t\in\{0,\dots,T-1\}$ and suppose the memory $\M_i$ satisfies Eq.~\eqref{eqn:memory of the alg t-1} at the beginning of iteration $t$ of the for loop in lines~4-14 of the algorithm. 

To proceed, let us consider any $s\in\CL(\T_i)$ with $j\in\V$ and $s_j(0)=s$ in the for loop in lines~5-9 of Algorithm~\ref{algorithm:control design}, where $\CL(\T_i)$ is defined in Eq.~\eqref{eqn:leaf nodes in T_i}. We will show that $\hat{w}_j(t-D_{ij}-1)$ in line~7, and thus $\hat{\zeta}_s(t-D_{ij})$ in line~8, can be determined using Eq.~\eqref{eqn:dynamics of zeta hat} and the current memory $\M_i$ of Algorithm~\ref{algorithm:control design}. As suggested by the first two cases in Eq.~\eqref{eqn:est w_i(t)}, we can focus on the case when $t-D_{ij}>0$ (otherwise, $\hat{\zeta}_s(t-D_{ij})$ can be directly determined). We then note from the third case in Eq.~\eqref{eqn:est w_i(t)} that in order to determine $\hat{w}_j(t-D_{ij}-1)$, we need to know $x_j(t-D_{ij})$, and $x_{j_1}(t-D_{ij}-1)$ and $\hat{u}_{j_1}(t-D_{ij}-1)$ for all $j_1\in\CN_j$, where $\CN_j$ is given in Assumption~\ref{ass:info structure}. Also note that $D_{ij_1}\le D_{ij}+1$ for all $j_1\in\CN_j$, which implies that $t-D_{\max}-1\le t-D_{ij}-1\le t-D_{ij_1}$. Thus, we have that $x_j(t-D_{ij})\in\tilde{\I}_i(t)$, and $x_{j_1}(t-D_{ij}-1)\in\tilde{\I}_i(t)$ for all $j_1\in\CN_j$, where $\tilde{\I}_i(t)$ is defined in Eq.~\eqref{eqn:info set used}. Now, considering any $j_1\in\CN_j$, we note from Eq.~\eqref{eqn:control policy} that $\hat{u}_{j_1}(t-D_{ij}-1)=\sum_{r\ni j_1}I_{\{j_1\},r}\hat{K}_r\hat{\zeta}_r(t-D_{ij}-1)$. Thus, in order to determine $\hat{u}_{j_1}(t-D_{ij}-1)$, it suffices to determine $\hat{\zeta}_r(t-D_{ij}-1)$ for all $r\in\U$ such that $j_1\in r$ (i.e., for all $r\ni j_1$). Note that $j_1\rightsquigarrow i$, i.e., there exists a directed path from node $j_1$ to node $i$ in $\G(\V,\A)$. Moreover, noting the definition of $\CP(\U,\CH)$ given by \eqref{eqn:def of info graph} with its properties discussed in Lemma~\ref{lemma:properties of info graph} and Remark~\ref{lemma:properties of info graph}, and noting the way we defined the set $\T_i$, one can show that for any $r\ni j_1$, it holds that $r\in\T_i$.
Next, considering any $r\ni j_1$, we recall from Eq.~\eqref{eqn:set of leaf nodes of s} that $\CL_r$ denotes the set of leaf nodes that can reach $r$ in $\CP(\U,\CH)$, i.e., $\CL_r=\{v\in\CL:v\rightsquigarrow r\}=\{v\in\CL(\T_i):v\rightsquigarrow r\}$, where the second equality again follows from the properties of $\CP(\U,\CH)$ and the definition of $\CL(\T_i)$ in Eq.~\eqref{eqn:leaf nodes in T_i}. 
Now, we split our arguments into two cases: $r$ is a root node in $\T_i$ (i.e., $r$ has a self loop), and $r$ does not have a self loop. 

First, suppose $r$ has a self loop. For the case when $r$ is a leaf node in $\T_i$ (i.e., $r$ is an isolated node in $\CP(\U,\CH)$ and $r\in\CL(\T_i)$), we see that $\hat{\zeta}_r(k)\in\M_i$ with $\M_i$ given by Eq.~\eqref{eqn:memory of the alg t-1}, for all $k\in\{k-2D_{\max}-1,\dots,k-D_{ij_r}-1\}$, where $j_r\in\V$ and $s_{j_r}(0)=r$. Noting from the definition of $\T_i$ that $i,j_r\in r$, we have from the construction of $\CP(\U,\CH)$ in Eq.~\eqref{eqn:def of info graph} that $j_r\rightarrow i$ in $\G(\V,\A)$ with $D_{ij_r}=0$. It follows that $\hat{\zeta}_r(t-D_{ij}-1)\in\M_i$ with $\M_i$ given by Eq.~\eqref{eqn:memory of the alg t-1}. Thus, we focus on the case when $r$ is not a leaf node in $\T_i$, i.e., $r\in\CR(\T_i)$. We now see from Eq.~\eqref{eqn:dynamics of zeta hat} that given $\hat{\zeta}_r(t-D_{ij}-2)$, and $\hat{\zeta}_{r^{\prime}}(t-D_{ij}-2)$ for all $r^{\prime}$ such that $r^{\prime}\rightarrow r$ (with $r\neq r^{\prime}$) in $\CP(\U,\CH)$, the state $\hat{\zeta}_r(t-D_{ij}-1)$ can be determined. Let us consider any $r^{\prime}$ such that $r^{\prime}\rightarrow r$ (with $r^{\prime}\neq r$) in $\CP(\U,\CH)$, and denote  $\CL_{r^{\prime}}=\{v^{\prime}\in\CL(\T_i):v^{\prime}\rightsquigarrow r^{\prime}\}$, where we note that $\CL_{r^{\prime}}\subseteq\CL_r$. For any $k\in\Z$, one can recursively apply Eq.~\eqref{eqn:dynamics of zeta hat} to show that given $\hat{\zeta}_{v^{\prime}}(k-l_{v^{\prime}r^{\prime}})$ for all $v^{\prime}\in\CL_{r^{\prime}}$, the state $\hat{\zeta}_{r^{\prime}}(k)$ can be determined, where $l_{v^{\prime}r^{\prime}}$ is the length of the (unique) directed path from $v^{\prime}$ to $r^{\prime}$ in $\CP(\U,\CH)$. Further considering any $v^{\prime}\in\CL_{r^{\prime}}$, and noting that $v^{\prime}\in\CL(\T_i)$, we have that $\hat{\zeta}_{v^{\prime}}(k)\in\M_i$ for all $k\in\{t-2D_{\max}-1,\dots,t-D_{ij_{v^{\prime}}}-1\}$ with $j_{v^{\prime}}\in\V$ and $s_{j_{v^{\prime}}}(0)=v^{\prime}$, where $\M_i$ is given by Eq.~\eqref{eqn:memory of the alg t-1}. Recalling again the definition of $\CP(\U,\CH)$ in \eqref{eqn:def of info graph}, and noting that $v^{\prime}\in\CL_{r^{\prime}}$, $j_{v^{\prime}}\in v^{\prime}$, $j_1\in r$, and $r^{\prime}\rightarrow r$ in $\CP(\U,\CH)$, one can show that 
\begin{equation}
\label{eqn:distance relation 2}
 D_{j_1j_{v^{\prime}}}\le l_{v^{\prime}r^{\prime}}+1\le D_{\max}. 
\end{equation}
We further split our arguments into $D_{ij_{v^{\prime}}}\le D_{ij}$ and $D_{ij_{v^{\prime}}}\ge D_{ij}+1$. First, supposing $D_{ij_{v^{\prime}}}\le D_{ij}$, we have
\begin{equation}
\label{eqn:distance relation 3}
\begin{split}
&t-2D_{\max}-1\le t-D_{\max}-l_{v^{\prime}r^{\prime}}-1\\
&t-D_{ij_{v^{\prime}}}-1\ge t-D_{ij}-l_{v^{\prime}r^{\prime}}-2.
\end{split}
\end{equation}
Second, suppose $D_{ij_{v^{\prime}}}\ge D_{ij}+1$. Recall from Remark~\ref{remark:order of elements in CL} that we let the for loop in lines~5-9 of Algorithm~\ref{algorithm:control design} iterate over the elements in $\CL(\T_i)$ according to a certain order of the elements in $\CL(\T_i)$. We then see from the inequality $D_{ij_{v^{\prime}}}\ge D_{ij}+1$ that $s_{j_{v^{\prime}}}(0)\in\CL(\T_i)$ (with $j_{v^{\prime}}\in\V$ and $s_{j_{v^{\prime}}}(0)=v^{\prime}$) has already been considered by the for loop in lines~5-9 in Algorithm~\ref{algorithm:control design}, i.e., the states $\hat{\zeta}_{v^{\prime}}(k)$ for all $k\in\{t-2D_{\max}-1,\dots,t-D_{ij_{v^{\prime}}}\}$ are in the current memory of Algorithm~\ref{algorithm:control design}, denoted as $\M_i^{\prime}$, when we consider the $s\in\CL(\T_i)$ with $j\in\V$ and $s_j(0)=s$ in the for loop in lines~5-9 of the algorithm. Moreover, we have from the above arguments that $j_{v^{\prime}}\rightsquigarrow j_1\rightsquigarrow i$ in $\G(\V,\A)$, i.e., there is a directed path from node $j_{v^{\prime}}$ to node $i$ that goes through node $j_1$ in $\G(\V,\A)$. It then follows that 
\begin{equation}
\label{eqn:distance relation 4}
D_{ij_{v^{\prime}}}\le  D_{ij_1}+D_{j_1j_{v^{\prime}}}\le D_{ij}+D_{jj_1}+D_{j_1j_{v^{\prime}}},
\end{equation} 
where $D_{jj_1}\in\{0,1\}$. Combining \eqref{eqn:distance relation 2} and \eqref{eqn:distance relation 4}, we obtain
\begin{equation}
\label{eqn:distance relation 3_1}
\begin{split}
&t-2D_{\max}-1\le t- D_{\max}-l_{v^{\prime}r^{\prime}}-1\\
& t-D_{ij_{v^{\prime}}}\ge t-D_{ij}-l_{v^{\prime}r^{\prime}}-2.
\end{split}
\end{equation}
It then follows from \eqref{eqn:distance relation 3} and \eqref{eqn:distance relation 3_1} and our arguments above that the states $\hat{\zeta}_{v^{\prime}}(k)$ for all $k\in\{t-D_{\max}-l_{v^{\prime}r^{\prime}}-1,\dots,t-D_{ij}-l_{v^{\prime}r^{\prime}}-2\}$ are in the current memory $\M_i^{\prime}$ described above, for all $v^{\prime}\in\CL_{r^{\prime}}$. Combining the above arguments together, we have that  $\hat{\zeta}_{r^{\prime}}(k)$ can be determined from Eq.~\eqref{eqn:dynamics of zeta hat} and the current memory $\M_i^{\prime}$, for all $k\in\{t-D_{\max}-1,\dots,t-D_{ij}-2\}$ and for all $r^{\prime}$ such that $r^{\prime}\rightarrow r$ (with $r\neq r^{\prime}$) in $\CP(\U,\CH)$. Moreover, recalling that $r\in\CR(\T_i)$ as we argued above, we see from Eq.~\eqref{eqn:memory of the alg t-1} that $\hat{\zeta}_r(t-D_{\max}-1)\in\M_i^{\prime}$. One can now apply Eq.~\eqref{eqn:dynamics of zeta hat} multiple times to show that $\hat{\zeta}_r(k)$ can be determined from the current memory $\M_i^{\prime}$ described above, for all $k\in\{t-D_{\max},\dots,t-D_{ij}-1\}$.

Next, suppose $r$ does not have a self loop. Similarly to our arguments above, we first consider the case when $r$ is a leaf node in $\T_i$, i.e., $r\in\CL(\T_i)$. We see that $\hat{\zeta}_r(t-D_{ij_r}-1)\in\M_i$, where $j_r\in\V$ with $s_{j_r}(0)=r$, and $\M_i$ is defined in Eq.~\eqref{eqn:memory of the alg t-1}. Since $j_1,j_r\in r$, we have from the construction of $\CP(\U,\CH)$ in \eqref{eqn:def of info graph} that $j_r\rightarrow j_1$ in $\G(\V,\A)$ with $D_{j_1j_r}=0$. Noting that $j_1\rightsquigarrow i$ in $\CP(\V,\A)$ as we argued above, we then have the following:
\begin{align}
D_{ij_r}\le D_{ij}+D_{jj_1}+D_{j_1j_r}=D_{ij}+D_{jj_1},\label{eqn:distance realtion 5}
\end{align}
where $D_{jj_1}\in\{0,1\}$. Now, supposing $D_{ij_r}\le D_{ij}$, we see directly see from Eq.~\eqref{eqn:memory of the alg t-1} that $\hat{\zeta}_r(t-D_{ij}-1)\in\M_i$ holds. Supposing $D_{ij_r}\ge D_{ij}+1$, we see from \eqref{eqn:distance realtion 5} that $D_{ij_r}=D_{ij}+1$. Using similar arguments to those above for the case when $r$ has a self loop (particularly, the order of the elements in $\CL(\T_i)$ over which the for loop in lines~5-9 of Algorithm~\ref{algorithm:control design} iterates), one can show that the states $\hat{\zeta}_r(k)$ for all $k\in\{t-2D_{\max}-1,\dots,t-D_{ij_r}\}$ have been added to the current memory of Algorithm~\ref{algorithm:control design}, denoted as $\M_i^{\prime\prime}$, when we consider the $s\in\CL(\T_i)$ with $j\in\V$ and $s_j(0)=s$ in the for loop in lines~5-9 of the algorithm. It follows that $\hat{\zeta}_r(t-D_{ij}-1)\in\M^{\prime\prime}_i$. Then, we consider the case when $r$ is not a leaf node in $\T_i$. We see from Eq.~\eqref{eqn:dynamics of zeta hat} that given $\hat{\zeta}_{r^{\prime}}(t-D_{ij}-2)$ for all $r^{\prime}$ such that $r^{\prime}\rightarrow r$ (with $r\neq r^{\prime}$) in $\CP(\U,\CH)$, the state $\hat{\zeta}_r(t-D_{ij}-1)$ can be determined. The remaining arguments then follow directly from those above for the case when $r$ has a self loop.

In summary, we have shown that $\hat{\zeta}_r(t-D_{ij}-1)$ can be determined from Eq.~\eqref{eqn:dynamics of zeta hat} and the current memory of Algorithm~\ref{algorithm:control design}, for all $r\ni j_1$, for all $j_1\in\CN_j$, and for all $s\in\CL(\T_i)$ with $j\in\V$ and $s_j(0)=s$. It then follows from our arguments above that  $\hat{w}_j(t-D_{ij}-1)$ in line~7 of Algorithm~\ref{algorithm:control design}, and thus $\hat{\zeta}_s(t-D_{ij})$ in line~8 of Algorithm~\ref{algorithm:control design}, can be determined using Eq.~\eqref{eqn:dynamics of zeta hat} and the current memory of Algorithm~\ref{algorithm:control design}, for all $s\in\CL(\T_i)$ with $j\in\V$ and $s_j(0)=s$. In other words, we have shown that $\hat{\zeta}_s(t-D_{ij})$ can be added to the memory of Algorithm~\ref{algorithm:control design} in line~9, for all $s\in\CL(\T_i)$ with $j\in\V$ and $s_j(0)=s$. 

Now, let us consider any $s\in\CR(\T_i)$ with $\CR(\T_i)$ defined in Eq.~\eqref{eqn:root nodes in T_i}. We will show that $\hat{\zeta}_r(t-D_{\max}-1)$ can be determined using Eq.~\eqref{eqn:dynamics of zeta hat} and the states in $\M_i$ given by Eq.~\eqref{eqn:memory of the alg t-1}, for all $r$ such that $r\rightarrow s$ in $\CP(\U,\CH)$. Note from our definition of $\CR(\T_i)$ in Eq.~\eqref{eqn:root nodes in T_i} that $s$ is not a leaf node. Following similar arguments to those above, let us consider any $r$ such that $r\rightarrow s$ (with $r\neq s$) in $\CP(\U,\CH)$, and denote $\CL_r=\{v\in\CL(\T_i):v\rightsquigarrow r\}$. Further considering any $v\in\CL_r$, and noting that $v\in\CL(\T_i)$, we have that $\hat{\zeta}_v(k)\in\M_i$ for all $k\in\{t-2D_{\max}-1,\dots,t-D_{ij_v}-1\}$, where $\M_i$ is given by Eq.~\eqref{eqn:memory of the alg t-1}, and $j_v\in\V$ with $s_{j_v}(0)=v$. Similarly to Eq.~\eqref{eqn:distance relation 2}, we have that $l_{vr}\le D_{\max}$, which implies that $t-2D_{\max}-1\le t-D_{\max}-l_{vr}-1$. Therefore, we see that $\hat{\zeta}_v(t-D_{\max}-l_{vr}-1)\in\M_i$ with $\M_i$ given by Eq.~\eqref{eqn:memory of the alg t-1}, for all $v\in\CL_r$. Using similar arguments to those above, one can now recursively apply Eq.~\eqref{eqn:dynamics of zeta hat} to show that  $\hat{\zeta}_r(t-D_{\max}-1)$ can be determined from $\M_i$ given by Eq.~\eqref{eqn:memory of the alg t-1}, for all $r$ such that $r\rightarrow s$ (with $r\neq s$). Since $\hat{\zeta}_s(t-D_{\max}-1)\in\M_i$, we see from Eq.~\eqref{eqn:dynamics of zeta hat} that $\hat{\zeta}_s(t-D_{\max})$ can be determined from $\M_i$ given by Eq.~\eqref{eqn:memory of the alg t-1}. Thus, we have shown that $\hat{\zeta}_s(t-D_{\max})$ can be added to the memory of Algorithm~\ref{algorithm:control design} in line~12, for all $s\in\CR(\T_i)$.

Combining all the above arguments together and noting line~14 in Algorithm~\ref{algorithm:control design}, we see that at the beginning of iteration $(t+1)$ of the for loop in lines~4-14 of Algorithm~\ref{algorithm:control design}, the memory $\M_i$ of the algorithm satisfies  
\begin{equation}
\label{eqn:memory of the alg t}
\M_i=\big\{\hat{\zeta}_s(k):k\in\{t-2D_{\max},\dots,t-D_{ij}\},s\in\CL(\T_i),j\in\V,s_j(0)=s\big\}\cup\{\hat{\zeta}_s(t-D_{\max}):s\in\CR(\T_i)\}.
\end{equation}
This completes the induction step for the proof of Eq.~\eqref{eqn:memory of the alg t-1}, and thus completes the proof of part~(a).

We then prove part~(b). Consider any $t\in\{0,1,\dots,T-1\}$. In order to prove part~(b), it suffices for us to show that $\hat{\zeta}_r(t)$ for all $r\ni i$ can be determined using Eq.~\eqref{eqn:dynamics of zeta hat} and the memory $\M_i$ after line~12 (and before line~14) in iteration $t$ of the for loop in lines~4-14 of Algorithm~\ref{algorithm:control design}, which is given by 
\begin{multline}
\label{eqn:memory of the alg t'}
\M_i=\big\{\hat{\zeta}_s(k):k\in\{t-2D_{\max}-1,\dots,t-D_{ij}\},s\in\CL(\T_i),j\in\V,s_j(0)=s\big\}\\\cup\{\hat{\zeta}_s(k):k\in\{t-D_{\max}-1,t-D_{\max}\},s\in\CR(\T_i)\}.
\end{multline}
Considering any $r\ni i$, one can show via the definition of $\CP(\U,\CH)$ in \eqref{eqn:def of info graph} and the definition of $\T_i$ that $r\in\T_i$. Again, we split our arguments into two cases: $r$ has a self loop, and $r$ does not have a self loop. 

First, suppose $r$ has a self loop. For the case when $r\in\CL(\T_i)$ (i.e., $r$ is an isolated node in $\CP(\U,\CH)$), we see that $\hat{\zeta}_r(t-D_{ij_r})\in\M_i$ with $\M_i$ given by Eq.~\eqref{eqn:memory of the alg t'}, where $j_r\in\V$ with $s_{j_r}(0)=r$. Noting that $i,j_r\in r$, we see from the definitions of $\CP(\U,\CH)$ in \eqref{eqn:def of info graph} that $j_r\rightarrow i$ in $\G(\A,\V)$ with $D_{ij_r}=0$. It follows that $\hat{\zeta}_r(t)\in\M_i$ with $\M_i$ given by Eq.~\eqref{eqn:memory of the alg t'}. Thus, we focus on the case when $r$ is not a leaf node, i.e., $r\in\CR(\T_i)$. We see from Eq.~\eqref{eqn:dynamics of zeta hat} that given $\hat{\zeta}_r(t-1)$, and $\hat{\zeta}_{r^{\prime}}(t-1)$ for all $r^{\prime}\rightarrow r$ (with $r^{\prime}\neq r$) in $\CP(\U,\CH)$, the state $\hat{\zeta}_r(t)$ can be determined. Again, let us consider any $r^{\prime}$ such that $r^{\prime}\rightarrow r$ in $\CP(\U,\CH)$, and denote $\CL_{r^{\prime}}=\{v^{\prime}\in\CL(\T_i):v^{\prime}\rightsquigarrow r^{\prime}\}$. Further considering any $v^{\prime}\in\CL_{r^{\prime}}$, and noting that $v^{\prime}\in\CL(\T_i)$, we have that $\hat{\zeta}_{v^{\prime}}(k)\in\M_i$ with $\M_i$ given by Eq.~\eqref{eqn:memory of the alg t'}, for all $k\in\{t-2D_{\max}-1,\dots,t-D_{ij_{v^{\prime}}}\}$, where $j_{v^{\prime}}\in\V$ and $s_{j_{v^{\prime}}}(0)=v^{\prime}$. Similarly to \eqref{eqn:distance relation 2}, we have that $D_{ij_{v^{\prime}}}\le l_{v^{\prime}r^{\prime}}+1\le D_{\max}$, which implies that 
\begin{equation}
\label{eqn:distance relation 5}
\begin{split}
&t-2D_{\max}-1\le t-D_{\max}-l_{v^{\prime}r^{\prime}}-1\\
&t-D_{ij_{v^{\prime}}}\ge t-l_{v^{\prime}r^{\prime}}-1.
\end{split}
\end{equation}
It then follows from \eqref{eqn:distance relation 5} that $\hat{\zeta}_{v^{\prime}}(k)\in\M_i$ with $\M_i$ given by Eq.~\eqref{eqn:memory of the alg t'}, for all $k\in\{t-D_{\max}-l_{v^{\prime}r^{\prime}}-1,\dots,t-l_{v^{\prime}r^{\prime}}-1\}$, and for all $v^{\prime}\in\CL_r$.
Using similar arguments to those before, one can recursively use Eq.~\eqref{eqn:dynamics of zeta hat} to show that $\hat{\zeta}_{r^{\prime}}(k)$ can be determined from $\M_i$ given by Eq.~\eqref{eqn:memory of the alg t'}, for all $k\in\{t-D_{\max}-1,\dots,t-1\}$. Moreover, recalling that $r\in\CR(\T_i)$ as we argued above, we see that $\hat{\zeta}_r(t-D_{\max}-1)\in\M_i$ with $\M_i$ given by Eq.~\eqref{eqn:memory of the alg t'}. One can then apply Eq.~\eqref{eqn:dynamics of zeta hat} multiple times and show that $\hat{\zeta}_r(t)$ can be determined from $\M_i$ given by Eq.~\eqref{eqn:memory of the alg t'}. Next, suppose $r$ does not have a self loop. 
Using similar arguments to those above for the case when $r$ has a self loop, one can show that $\hat{\zeta}_r(t)$ can be determined using Eq.~\eqref{eqn:dynamics of zeta hat} and the current memory $\M_i$ given in Eq.~\eqref{eqn:memory of the alg t'}.

Combining the above arguments together, we conclude that for any $t\in\{0,1,\dots,T-1\}$, the control input $\hat{u}_i(t)$ in line~13 can be determined using Eq.~\eqref{eqn:dynamics of zeta hat} and the memory $\M_i$ given by Eq.~\eqref{eqn:memory of the alg t'}. This completes the proof of part~(b).\hfill\qed

\section{Proofs for Perturbation Bounds on Solutions to Ricatti Equations}\label{sec:proofs of DARE bounds}
\subsection{Proof of Lemma~\ref{lemma:upper bounds on P_r hat and K_r hat self loop}}
Consider any $r\in\U$ that has a self loop. To show that \eqref{eqn:P_r hat and P_r self loop} holds under the assumption on $\varepsilon$ given in \eqref{eqn:upper bound on epsilon 1}, we use \cite[Proposition~2]{mania2019certainty}. Specifically, since $\norm{\hat{A}-A}\le\varepsilon$ and $\norm{\hat{B}-B}\le\varepsilon$ hold, one can show that $\norm{\hat{A}_{rr}-A_{rr}}\le\varepsilon$ and $\norm{\hat{B}_{rr}-B_{rr}}\le\varepsilon$ hold, for all $r\in\U$. Similarly, noting that $\norm{A_{rr}}\le\norm{A}$ and $\norm{B_{rr}}\le\norm{B}$ hold, for all $r\in\U$, one can show that $\norm{A_{rr}+B_{rr}K_r}\le\tilde{\Gamma}^2$. Moreover, note that $1\le\sigma_n(R)\le\sigma_{n_r}(R_{rr})\le\sigma_1(R_{rr})\le\sigma_1(R)$ and $1\le\sigma_m(Q)\le\sigma_{m_r}(Q_{rr})\le\sigma_1(Q_{rr})\le\sigma_1(Q)$ from Assumption~\ref{ass:cost matrices}, where $n_r\triangleq\sum_{i\in r}n_i$ and $m_r\triangleq\sum_{i\in r}m_i$. Recalling the definitions of $\kappa$, $\gamma$ and $\Gamma$, the proof of \eqref{eqn:P_r hat and P_r self loop} under the assumption on $\varepsilon$ given in \eqref{eqn:upper bound on epsilon 1} now follows from \cite[Proposition~2]{mania2019certainty}. 

Next, let denote
\begin{equation*}
\label{eqn:f(epsilon)}
f(\varepsilon)=6\frac{\kappa^2}{1-\gamma^2}\tilde{\Gamma}^5(1+\sigma_1(R^{-1}))\varepsilon,
\end{equation*}
and note that $\norm{P_r}\le\Gamma$. Moreover, we see from Eq.~\eqref{eqn:set of DARES K hat} that 
\begin{equation*}
\label{eqn:K_v hat}
\hat{K}_r= -(R_{rr}+\hat{B}_{rr}^{\top} \hat{P}_{r} \hat{B}_{rr}^{\top})^{-1} \hat{B}_{rr}^{\top} \hat{P}_{r} \hat{A}_{rr},
\end{equation*}
where $\norm{\hat{A}_{rr}-A_{rr}}\le\varepsilon$ and $\norm{\hat{B}_{rr}-B_{rr}}\le\varepsilon$. We have
\begin{equation*}
\norm{\hat{B}^{\top}_{rr}\hat{P}_{r}\hat{B}_{rr} - B^{\top}_{rr}P_{r}\hat{B}_{rr}}
\le\norm{\hat{B}^{\top}_{rr}\hat{P}_{rr}\hat{B}_{rr} - B^{\top}_{rr}\hat{P}_{r}\hat{B}_{rr}}+ \norm{B^{\top}_{rr}\hat{P}_{r}\hat{B}_{rr} - B^{\top}_{rr}P_{r}\hat{B}_{rr}} 
+ \norm{B^{\top}_{rr}P_{r}\hat{B}_{rr} - B^{\top}_{rr}P_{r}B_{rr}},
\end{equation*}
which implies that 
\begin{equation*}
\label{eqn:upper bound split 1}
\norm{\hat{B}^{\top}_{rr}\hat{P}_{r}\hat{B}_{rr} - B^{\top}_{rr}P_{r}B_{rr}}
\le\norm{\hat{P}_{r}}\norm{\hat{B}_{rr}}\varepsilon + \norm{B_{rr}}\norm{\hat{B}_{rr}}f(\varepsilon) + \norm{B_{rr}}\norm{P_r}\varepsilon.
\end{equation*}
Noting that $\varepsilon\le f(\varepsilon)\le\frac{1}{6}$ and recalling the definition of $\Gamma$ given in~\eqref{eqn:Gamma}, we have 
\begin{equation*}
\label{eqn:upper bound split 2}
\norm{\hat{B}^{\top}_{rr}\hat{P}_{r}\hat{B}_{rr} - B^{\top}_{rr}P_{r}B_{rr}}\le3\tilde{\Gamma}^2f(\varepsilon).
\end{equation*}
Similarly, one can show that 
\begin{equation*}
\label{eqn:upper bound split 3}
\norm{\hat{B}^{\top}_{rr}\hat{P}_{r}\hat{A}_{rr} - B^{\top}_{rr}P_{r}A_{rr}}\le3\tilde{\Gamma}^2f(\varepsilon).
\end{equation*}
Now, following similar arguments to those in the proof of \cite[Lemma~2]{mania2019certainty}, one can show that 
\begin{equation*}
\label{eqn:K_v hat and K_v}
\norm{\hat{K}_r-K_r}\le3\tilde{\Gamma}^3f(\varepsilon)\le1,
\end{equation*}
where the second inequality follows from the assumption on $\varepsilon$ given in \eqref{eqn:upper bound on epsilon 1}. \hfill\qed

\subsection{Proof of Lemma~\ref{lemma:upper bounds on P_r and K_r general}}
First, let us consider any $r\in\U$ such that $l_{rs}=1$, i.e., $r\rightarrow s$. Since $s$ is the unique root node that is reachable from $r$, we see from Lemma~\ref{lemma:properties of info graph} and Remark~\ref{remark:tree info graph} that $s$ has a self loop. Noting that $\sigma_1(R)\ge1$ from Assumption~\ref{ass:cost matrices}, and that $\tilde{\Gamma}\ge1$, we see that any $\varepsilon$ satisfying \eqref{eqn:upper bound on epsilon 2} also satisfies \eqref{eqn:upper bound on epsilon 1}. Thus, we have from \eqref{eqn:P_r hat and P_r self loop} in Lemma~\ref{lemma:upper bounds on P_r hat and K_r hat self loop} that 
\begin{equation*}
\label{eqn:P_rv hat and P_rv}
\norm{\hat{P}_s-P_s}\le f(\varepsilon)\le\frac{1}{6},
\end{equation*}
where 
\begin{equation*}
\label{eqn:f(epsilon)}
f(\varepsilon)=6\frac{\kappa^2}{1-\gamma^2}\tilde{\Gamma}^5(1+\sigma_1(R^{-1}))\varepsilon,
\end{equation*}
and we note that $\norm{P_s}\le\Gamma$. Moreover, we see from Eq.~\eqref{eqn:set of DARES K hat} that 
\begin{equation*}
\label{eqn:K_v hat}
\hat{K}_r= -(R_{rr}+\hat{B}_{sr}^{\top} \hat{P}_{s} \hat{B}_{sr}^{\top})^{-1} \hat{B}_{sr}^{\top} \hat{P}_{s} \hat{A}_{sr},
\end{equation*}
where $\norm{\hat{A}_{sr}-A_{rr}}\le\varepsilon$ and $\norm{\hat{B}_{sr}-B_{sr}}\le\varepsilon$ (since $\norm{\hat{A}-A}\le\varepsilon$ and $\norm{\hat{B}-B}\le\varepsilon$). Now, using similar arguments to those in the proof of Lemma~\ref{lemma:upper bounds on P_r hat and K_r hat self loop}, one can also show that 
\begin{equation*}
\label{eqn:upper bound split 2}
\norm{\hat{B}^{\top}_{sr}\hat{P}_{s}\hat{B}_{sr} - B^{\top}_{sr}P_{s}B_{sr}}\le3\tilde{\Gamma}^2f(\varepsilon),
\end{equation*}
and
\begin{equation*}
\label{eqn:upper bound split 3}
\norm{\hat{B}^{\top}_{sr}\hat{P}_{s}\hat{A}_{sr} - B^{\top}_{sr}P_{s}A_{sr}}\le3\tilde{\Gamma}^2f(\varepsilon).
\end{equation*}
Using similar arguments to those in the proof of \cite[Lemma~2]{mania2019certainty}, one can now show that 
\begin{equation}
\label{eqn:K_v hat and K_v}
\norm{\hat{K}_r-K_r}\le3\tilde{\Gamma}^3f(\varepsilon)\le1,
\end{equation}
where the second inequality follows from the assumption on $\varepsilon$ given in \eqref{eqn:upper bound on epsilon 2}, and we note that $\norm{K_r}\le\Gamma$. Hence, we have shown that \eqref{eqn:K_r hat and K_r general} holds for $l_{rs}=1$

To prove Eq.~\eqref{eqn:P_r hat and P_r general} for $l_{rs}=1$, we first recall the expressions for $P_r$ and $\hat{P}_r$ given in Eqs.~\eqref{eqn:set of DARES P} and \eqref{eqn:set of DARES P hat}, respectively. Using similar arguments to those above, we have
\begin{align}\nonumber
&\norm{\hat{A}_{sr}+\hat{B}_{sr}\hat{K}_r-A_{sr}-B_{sr}K_r}\\\nonumber
\le&\norm{\hat{A}_{sr}-A_{sr}}+\norm{\hat{B}_{sr}\hat{K}_r-B_{sr}\hat{K}_r}+\norm{B_{sr}\hat{K}_r-B_{sr}K_r}\\
\le&\varepsilon+(\Gamma+1)\varepsilon+\Gamma3\widetilde{\Gamma}^3 f(\varepsilon)\le4\tilde{\Gamma}^4f(\varepsilon),\label{eqn:upper bound split 4}
\end{align}
where the first inequality in \eqref{eqn:upper bound split 4} follows from \eqref{eqn:K_v hat and K_v} and the definition of $\Gamma$ given in~\eqref{eqn:Gamma}. To obtain the second inequality in \eqref{eqn:upper bound split 4}, we first note from Eq.~\eqref{eqn:set of DARES P} that $P_r\succeq Q_{rr}\succeq\sigma_m(Q)I\succeq I$, where the last inequality follows from Assumption~\ref{ass:cost matrices}. We then have from~\eqref{eqn:Gamma} that $\Gamma\ge\norm{P_r}\ge1$, which further implies that $\tilde{\Gamma}+1\le\tilde{\Gamma}^4$. It follows that the second inequality in \eqref{eqn:upper bound split 4} holds. To proceed, denoting $\hat{L}_{sr}=\hat{A}_{sr}+\hat{B}_{sr}\hat{K}_r$ and $L_{sr}=A_{sr}+B_{sr}K_r$, we have from Eqs.~\eqref{eqn:set of DARES P} and \eqref{eqn:set of DARES P hat} that 
\begin{equation}
\label{eqn:P_v hat minus P_v}
\norm{\hat{P}_r-P_r}\le\norm{\hat{K}_r^{\top}R_{rr}\hat{K}_r - K_r^{\top}R_{rr}K_r}+\norm{\hat{L}_{sr}^{\top}\hat{P}_{s}\hat{L}_{sr} - L_{sr}^{\top}P_{s}L_{sr}},
\end{equation}
where we note that $\norm{L_{sr}}\le\tilde{\Gamma}^2$. From the above arguments, we have the following:
\begin{align}\nonumber
\norm{\hat{K}_r^{\top}R_{rr}\hat{K}_r-K^{\top}_rR_{rr}K_r}&\le\norm{\hat{K}_r^{\top}R_{rr}\hat{K}_r-K_r^{\top}R_{rr}\hat{K}_r}+\norm{K_r^{\top}R_{rr}\hat{K}_r-K_r^{\top}R_{rr}K_r}\\\nonumber
&\le3\tilde{\Gamma}^3f(\varepsilon)\sigma_1(R)(\Gamma+3\tilde{\Gamma}^3f(\varepsilon))+\Gamma\sigma_1(R)3\tilde{\Gamma}^3f(\varepsilon)\\\nonumber
&\le6\tilde{\Gamma}^4\sigma_1(R)f(\varepsilon)+9\tilde{\Gamma}^6\sigma_1(R)f(\varepsilon)^2,
\end{align}
and 
\begin{align}\nonumber
&\norm{\hat{L}_{sr}^{\top}\hat{P}_{s}\hat{L}_{sr} - L_{sr}^{\top}P_{s}L_{sr}}\\\nonumber
\le&\norm{\hat{L}_{sr}^{\top}\hat{P}_{s}\hat{L}_{sr}-L_{sr}^{\top}\hat{P}_{s}\hat{L}_{sr}}+\norm{L_{sr}^{\top}\hat{P}_{s}\hat{L}_{sr}-L_{sr}^{\top}P_{s}\hat{L}_{sr}}+\norm{L_{sr}^{\top}P_{s}\hat{L}_{sr}-L_{sr}^{\top}P_{s}L_{sr}}\\\nonumber
\le&4\tilde{\Gamma}^4f(\varepsilon)(\Gamma+1)(\tilde{\Gamma}^2+4\tilde{\Gamma}^4f(\varepsilon))+\tilde{\Gamma}^2f(\varepsilon)(\tilde{\Gamma}^2+4\tilde{\Gamma}^4f(\varepsilon))+\tilde{\Gamma}^2\Gamma 4\tilde{\Gamma}^4f(\varepsilon)\\\nonumber
\le&4\tilde{\Gamma}^7f(\varepsilon)+16\tilde{\Gamma}^9f(\varepsilon)^2+\tilde{\Gamma}^4f(\varepsilon)+4\tilde{\Gamma}^6f(\varepsilon)^2+4\tilde{\Gamma}^7f(\varepsilon).
\end{align}
Noting that $f(\varepsilon)\le\frac{1}{6}$, one can now obtain from \eqref{eqn:P_v hat minus P_v} that
\begin{equation}
\label{eqn:P_v hat minus P_v 1}
\norm{\hat{P}_r-P_r}\le20\tilde{\Gamma}^9\sigma_1(R)f(\varepsilon)\le\frac{1}{6},
\end{equation}
where the second inequality again follows from the assumption on  $\varepsilon$ given in \eqref{eqn:upper bound on epsilon 2}.

Next, let us consider any $r\in\U$ such that $l_{sr}=2$ (where $s\in\U$ is the unique root node that is reachable from $r$), and denote the unique directed path from $r$ to $s$ in $\CP(\U,\CH)$ as $r\rightarrow r_1\rightarrow s$. We see that $r_1\in\U$ satisfies \eqref{eqn:K_v hat and K_v} and \eqref{eqn:P_v hat minus P_v 1}. Repeating the above arguments for obtaining \eqref{eqn:K_v hat and K_v} and \eqref{eqn:P_v hat minus P_v 1} one more time, one can show that 
\begin{equation*}
\label{eqn:K_v hat and K_v l=2}
\norm{\hat{K}_r-K_r}\le3\tilde{\Gamma}^320\tilde{\Gamma}^9\sigma_1(R)f(\varepsilon)\le1,
\end{equation*}
and
\begin{equation*}
\label{eqn:P_v hat minus P_v l=2}
\norm{\hat{P}_r-P_r}\le(20\tilde{\Gamma}^9\sigma_1(R))^2f(\varepsilon)\le\frac{1}{6},
\end{equation*}
where we use again the assumption on $\varepsilon$ given by \eqref{eqn:upper bound on epsilon 2}. Further repeating the above arguments, and noting from Eq.~\eqref{eqn:depth of T_i} and the definition of $\CP(\U,\CH)$ in \eqref{eqn:def of info graph} that $l_{sr}\le D_{\max}$ for all $s,r\in\U$, one can show that \eqref{eqn:K_r hat and K_r general}-\eqref{eqn:P_r hat and P_r general} hold, for all $r\in\U$ without a self loop, under the assumption on $\varepsilon$ given in \eqref{eqn:upper bound on epsilon 2}.\hfill\qed

\section{Proofs for Perturbation Bounds on Costs}\label{sec:proofs of bounds on costs}
\subsection{Proof of Lemma~\ref{lemma:upper bound on zeta_tilde}}
First, let us consider any $s\in\U$ that has a self loop. Noting the construction of the information graph $\CP=(\U,\CH)$ given in \eqref{eqn:def of info graph}, one can show that Eq.~\eqref{eqn:dynamics of zeta_tilde} can be rewritten as
\begin{equation}
\label{eqn:dynamics of zeta_tilde 1 self loop}
\tilde{\zeta}_s(t+1)=(A_{ss}+B_{ss}\hat{K}_s)\tilde{\zeta}_s(t)+\sum_{v\in\CL_s}H(v,s)\sum_{w_j\rightarrow v}I_{v,\{j\}}w_j(t-l_{vs}),
\end{equation}
where $\CL_s=\{v\in\CL:v\rightsquigarrow s\}$ is the set of leaf nodes in $\CP(\U,\CH)$ that can reach $s$, $l_{vs}$ is the length of the (unique) directed path from node $v$ to node $s$ in $\CP(\U,\CH)$ with $l_{vs}=0$ if $v=s$, and
\begin{equation*}
\label{eqn:H(v,s)}
H(v,s) \triangleq (A_{sr_1}+B_{sr_1}\hat{K}_{r_1})\cdots(A_{r_{l_{vs}-1}v}+B_{r_{l_{vs}-1}v}\hat{K}_v),
\end{equation*}
with $H(v,s)=I$ if $v=s$, where $\hat{K}_r$ is given by Eq.~\eqref{eqn:set of DARES K} for all $r\in\U$, and $v,r_{l_{vs}-1},\dots,r_1,s$ are the nodes along the directed path from $v$ to $s$ in $\CP(\U,\CH)$. Recalling from Eq.~\eqref{eqn:dynamics of zeta_tilde} that $\tilde{\zeta}_s(0)=\sum_{w_i\rightarrow s}I_{s,\{i\}}x_i(0)=0$, in Eq.~\eqref{eqn:dynamics of zeta_tilde 1 self loop} we set $w_j(t-l_{vs})=0$ if $t-l_{vs}<0$. Now, under the assumption on $\varepsilon$ given in Eq.~\eqref{eqn:upper bound on epsilon 2}, we see from \eqref{eqn:K_r hat and K_r general} in Lemma~\ref{lemma:upper bounds on P_r and K_r general} that $\norm{\hat{K}_r}\le\tilde{\Gamma}$, which implies that $\norm{A_{sr}+B_{sr}\hat{K}_r}\le\tilde{\Gamma}^2$, for all $r\in\U$ with $r\neq s$. Noting that $l_{vs}\le D_{\max}$ from the construction of $\CP(\U,\CH)$, we have $\norm{H(v,s)}\le\tilde{\Gamma}^{2D_{\max}}$, for all $v\in\CL_s$. Considering any $t\in\Z_{\ge0}$ and denoting 
\begin{equation}
\label{eqn:eta_s}
\eta_s(t)=\sum_{v\in\CL_s}H(v,s)\sum_{w_j\rightarrow v}I_{v,\{j\}}w_j(t-l_{vs}),
\end{equation}
we have
\begin{equation*}
\E\Big[\eta_s(t)\eta_s(t)^{\top}\Big]=\E\Big[\sum_{v\in\CL_s}\sum_{w_j\rightarrow v}H(v,s)I_{v,\{j\}}w_j(t-l_{vs})w_j(t-l_{vs})^{\top}I_{\{j\},v}H(v,s)^{\top}\Big],
\end{equation*}
where we use the fact from $w(t)\sim\CN(0,\sigma_w^2I)$ that $w_{j_1}(t)$ and $w_{j_2}(t)$ are independent for all $j_1,j_2\in\V$ with $j_1\neq j_2$, and the fact that for any $v\in\U$ with $s_j(0)=v$, $w_j$ is the only noise term such that $w_j\rightarrow v$ (see Footnote~2). Moreover, we see that $\eta_s(t_1)$ and $\eta_s(t_2)$ are independent for all $t_1,t_2\in\Z_{\ge0}$ with $t_1\neq t_2$, and that $\eta_s(t)$ is independent of $\tilde{\zeta}_s(t)$ for all $t\in\Z_{\ge0}$. Now, considering any $k\in\Z_{\ge0}$ such that $k-l_{vs}\ge0$ for all $v\in\CL_s$, and noting that $w(t)\sim\CN(0,\sigma_w^2I)$ for all $t\in\Z_{\ge0}$, we have
\begin{equation}
\label{eqn:cov of eta_s}
\E\Big[\eta_s(k)\eta_s(k)^{\top}\Big]=\sigma_w^2\sum_{v\in\CL_s}\sum_{w_j\rightarrow v}H(v,s)I_{v,\{j\}}I_{\{j\},v}H(v,s)^{\top}.
\end{equation}
Let us denote the right-hand side of Eq.~\eqref{eqn:cov of eta_s} as $\bar{W}_s$, and denote
\begin{equation*}
\label{eqn:W_bar s and L_tilde ss}
\tilde{L}_{ss}=A_{ss}+B_{ss}\hat{K}_s.
\end{equation*}
Fixing any $\tau\in\Z_{\ge1}$ such that $\tau-l_{vs}\ge0$ for all $v\in\CL_s$, and considering any $t\ge\tau$, one can then unroll Eq.~\eqref{eqn:dynamics of zeta_tilde 1 self loop} and show that
\begin{align}
\E\Big[\tilde{\zeta}_s(t)\tilde{\zeta}_s(t)^{\top}\Big]= \tilde{L}_{ss}^{t-\tau}\E\Big[\tilde{\zeta}_s(\tau)\tilde{\zeta}_s(\tau)^{\top}\Big](\tilde{L}_{ss}^{\top})^{t-\tau}+\sum_{k=0}^{t-\tau-1}L_{ss}^k\bar{W}_s(\tilde{L}_{ss}^{\top})^k.\label{eqn:cov of zeta_s tilde}
\end{align}
Under the assumption on $\varepsilon$ given in \eqref{eqn:upper bound on epsilon 2}, one can obtain from Lemma~\ref{lemma:upper bounds on P_r hat and K_r hat self loop} that $\norm{\tilde{L}_{ss}^k}\le\kappa(\frac{\gamma+1}{2})^k$ for all $k\ge0$, where $0<\frac{\gamma+1}{2}<1$, which implies that $\tilde{L}_{ss}$ is stable. It follows that
\begin{align*}
\lim_{t\to\infty}\E\Big[\tilde{\zeta}_s(t)\tilde{\zeta}_s(t)^{\top}\Big]&= \lim_{t\to\infty}\sum_{k=0}^{t-\tau-1}L_{ss}^k\bar{W}_s(\tilde{L}_{ss}^{\top})^k\\
&\preceq\norm{\bar{W}_s}\lim_{t\to\infty}\sum_{k=0}^{k-0-1}\norm{\tilde{L}_{ss}^k}\norm{(\tilde{L}_{ss}^{\top})^k}I\\
&\preceq\frac{4\norm{\bar{W}_s}\kappa^2}{1-\gamma^2}I.
\end{align*}
Noting that $|\CL_s|\le p$ from the definition of $\CP(\U,\CH)$ given in \eqref{eqn:def of info graph}, and that for any $v\in\U$ with $s_j(0)=v$, $w_j$ is the only noise term such that $w_j\rightarrow v$, as we argued above, we have from Eq.~\eqref{eqn:cov of eta_s} that
\begin{align*}
\norm{\bar{W}_s}\le\sigma_w^2 p\max_{v\in\CL_s}\norm{H(v,s)}^2\le\sigma_w^2p\tilde{\Gamma}^{4D_{\max}},
\end{align*}
where the second inequality follows from the fact that $\norm{H(v,s)}\le\tilde{\Gamma}^{2D_{\max}}$ as we argued above. It then follows that \eqref{eqn:upper bound on cov zeta_tilde} holds.

Next, let us consider any $s\in\U$ that does not have a self loop. Similarly to Eq.~\eqref{eqn:dynamics of zeta_tilde 1 self loop}, one can rewrite Eq.~\eqref{eqn:dynamics of zeta_tilde} as
\begin{equation*}
\label{eqn:dynamics of zeta_tilde 1 no self loop}
\tilde{\zeta}_s(t+1) = \sum_{v\in\CL_s}H(v,s)\sum_{w_j\rightarrow v}I_{v,\{j\}}w_j(t-l_{vs}).
\end{equation*}
Using similar arguments to those above, one can show that \eqref{eqn:upper bound on cov zeta_tilde} also holds.\hfill\qed

\subsection{Proof of Proposition~\ref{prop:J_tilde minus J}}
First, since $\varepsilon$ satisfies \eqref{eqn:upper bound on epsilon 2} (and thus \eqref{eqn:upper bound on epsilon 1}), we have from \eqref{eqn:K_r hat stabilizable} in Lemma~\ref{lemma:upper bounds on P_r and K_r general} that $A_{ss}+B_{ss}\hat{K}_s$ is stable for any $s\in\U$ that has a self loop. Now, using similar arguments to those for the proofs of Theorem~2 and Corollary~4 in \cite{lamperski2015optimal}, and leveraging Lemma~\ref{lemma:zeta_s1 and zeta_s2 indept} and Eqs.~\eqref{eqn:u tilde}-\eqref{eqn:dynamics of zeta_tilde}, \eqref{eqn:set of DARES K hat} and \eqref{eqn:set of DARES P tilde}, one can show that Eq.~\eqref{eqn:exp for J_tilde} holds. To proceed, for any $t\in\Z_{\ge0}$ and for any $T\ge t$ and, we set $T^{\prime}=T+\lceil\frac{\varphi}{J_{\star}}T\rceil$, and define
\begin{equation}
\label{eqn:J(x_tilde)}
J_T(\tilde{x}(t))=\E\Big[\sum_{k=t}^{T^{\prime}-1}(x(k)^{\top}Qx(k)+u^{\star}(k)^{\top}Ru^{\star}(k))\Big],
\end{equation}
where $u^{\star}(k)=\sum_{s\in\U}I_{\V,s}K_s\zeta_s(k)$ is the optimal control policy given by Eq.~\eqref{eqn:exp for u^star}, for all $i\in\V$ and for all $k\ge t$, and where $K_r$ and $\zeta_r(k)$ are given by Eqs.~\eqref{eqn:set of DARES K} and \eqref{eqn:dynamics of zeta}, respectively, for all $r\in\U$. Moreover, on the right-hand side of Eq.~\eqref{eqn:J(x_tilde)}, we set $x(t)=\tilde{x}(t)$, where $\tilde{x}(t)$ given by Eq.~\eqref{eqn:state x tilde} is the state after applying the control policy $\tilde{u}(k)$ in Eq.~\eqref{eqn:u tilde} for $k\in\{0,\dots,t-1\}$. Noting that $\tilde{x}(0)=x(0)$ as we discussed at the beginning of Section~\ref{sec:sub guarantees}, and that $T+\frac{\varphi}{J_{\star}}T\le T^{\prime}\le T+\frac{\varphi}{J_{\star}}T+1$, we see that
\begin{align*}
\lim_{T\to\infty}\frac{1}{T}J_T(\tilde{x}(0))&=\lim_{T\to\infty}\big(\frac{T^{\prime}}{T}\frac{1}{T^{\prime}}J_T(\tilde{x}(0))\big)\\
&=\lim_{T\to\infty}\frac{T^{\prime}}{T}\big(\lim_{T\to\infty}\frac{1}{T^{\prime}}J_T(\tilde{x}(0))\big)\\
&\le(1+\frac{\varphi}{J_{\star}})J_{\star}=J_{\star}+\varphi,
\end{align*}
where we use the fact that $\lim_{T\to\infty}\frac{1}{T^{\prime}}J_T(\tilde{x}(0))=J_{\star}$.

Recalling the definition of $\tilde{J}$ in Eq.~\eqref{eqn:J tilde}, we denote $\tilde{c}(t)=\tilde{x}(t)^{\top}Q\tilde{x}(t)+\tilde{u}(t)^{\top}R\tilde{u}(t)$. We then have the following:
\begin{align}\nonumber
\tilde{J}&=\lim_{T\to\infty}\frac{1}{T}\E\Big[\sum_{t=0}^{T-1}\tilde{c}(t)\Big]\\\nonumber
&=\lim_{T\to\infty}\frac{1}{T}\bigg(\E\Big[\sum_{t=0}^{T-1}\big(\tilde{c}(t)-J_T(\tilde{x}(t))\big)+\sum_{t=0}^{T-1}J_T(\tilde{x}(t))\Big]\bigg)\\\nonumber
&=\lim_{T\to\infty}\frac{1}{T}\bigg(\E\Big[\sum_{t=0}^{T-2}\big(\tilde{c}(t)-J_T(\tilde{x}(t))\big)+\sum_{t=0}^{T-2}J_T(\tilde{x}(t+1))+J_T(\tilde{x}(0))\Big]\bigg)\\\nonumber
&=\lim_{T\to\infty}\frac{1}{T}\bigg(\E\Big[\sum_{t=0}^{T-2}\big(\tilde{c}(t)+J_T(\tilde{x}(t+1))-J_T(\tilde{x}(t))\big)+J_T(\tilde{x}(0))\Big]\bigg).
\end{align}
Using similar arguments to those for the proof of \cite[Theorem~2]{lamperski2015optimal}, one can show that 
\begin{multline}
\label{eqn:J(x_t+1) minus J(x_t)}
J_T(\tilde{x}(t+1))-J_T(\tilde{x}(t))=\E\Big[\sum_{r\in\U}\tilde{\zeta}_r(t+1)^{\top}P_r(t+1)\tilde{\zeta}_r(t+1)\Big]\\-\E\Big[\sum_{r\in\U}\tilde{\zeta}_r(t)^{\top}P_r(t)\tilde{\zeta}_r(t)\Big]
-\sigma_w^2\sum_{\substack{i\in\V\\ w_i\rightarrow r}}\tr\big(I_{\{i\},r}P_r(t+1)I_{r,\{i\}}\big),
\end{multline}
where $\tilde{\zeta}_r(k)=\sum_{w_i\to r}I_{r,\{i\}}\tilde{x}_i(k)$ for $k\in\{t,t+1\}$. To obtain $P_r(t)$ for all $t\in\{0,\dots,T\}$ and for all $r\in\U$ in Eq.~\eqref{eqn:J(x_t+1) minus J(x_t)}, we use the following recursion:
\begin{equation}
\label{eqn:recursive DARES P}
P_r(k)=Q_{rr}+K_r^{\top}R_{rr}K_r+(A_{sr}+B_{sr}K_r)^{\top}P_s(k+1)(A_{sr}+B_{sr}K_r),
\end{equation}
initialized with $P_r(T^{\prime})=0$ for all $r\in\U$, where $T^{\prime}=T+\lceil\frac{\varphi}{J_{\star}}T\rceil$, and for each $r\in\U$, we let $s\in\U$ be the unique node such that $r\rightarrow s$, and $K_r$ is given by Eq.~\eqref{eqn:set of DARES K} for all $r\in\U$. Combining the above arguments together, we obtain the following:
\begin{align}\nonumber
\tilde{J}-J_{\star}&=\lim_{T\to\infty}\frac{1}{T}\E\Big[\sum_{t=0}^{T-1}\Big(\tilde{c}(t)+\sum_{r\in\U}\tilde{\zeta}_r(t+1)^{\top}P_r(t+1)\tilde{\zeta}_r(t+1)-\sum_{r\in\U}\tilde{\zeta}_r(t)^{\top}P_r(t)\tilde{\zeta}_r(t)\\\nonumber
&\qquad\qquad\qquad\qquad\qquad\qquad\quad-\sigma_w^2\sum_{\substack{i\in\V\\ w_i\rightarrow r}}\tr\big(I_{\{i\},r}P_r(t+1)I_{r,\{i\}}\big)\Big)\Big]+\lim_{T\to\infty}\frac{1}{T}J_T(\tilde{x}(0))-J_{\star}\\\nonumber
&\le\lim_{T\to\infty}\frac{1}{T}\E\Big[\sum_{t=0}^{T-1}\sum_{r\in\U}\Big(\tilde{\zeta}_r(t)^{\top}(Q_{rr}+\hat{K}_r^{\top}R_{rr}\hat{K}_r)\tilde{\zeta}_r(t)+\tilde{\zeta}_r(t+1)^{\top}P_r(t+1)\tilde{\zeta}_r(t+1)\\
&\qquad\qquad\qquad\qquad\qquad\qquad-\tilde{\zeta}_r(t)^{\top}P_r(t)\tilde{\zeta}_r(t)-\sigma_w^2\sum_{\substack{i\in\V\\ w_i\rightarrow r}}\tr\big(I_{\{i\},r}P_r(t+1)I_{r,\{i\}}\big)\Big)\Big]+\varphi\label{eqn:J_tilde minus J 2}\\\nonumber
&=\lim_{T\to\infty}\frac{1}{T}\E\Big[\sum_{t=0}^{T-1}\sum_{r\in\U}\Big(\tilde{\zeta}_r(t)^{\top}\big(Q_{rr}+\hat{K}_r^{\top}R_{rr}\hat{K}_r-P_r(t)\big)\tilde{\zeta}_r(t)\\
&\qquad\qquad\qquad\qquad\qquad+\sum_{v\to r}\tilde{\zeta}_v^{\top}(t)(A_{rv}+B_{rv}\hat{K}_v)^{\top}P_r(t+1)(A_{rv}+B_{rv}\hat{K}_v)\tilde{\zeta}_v(t)\Big)\Big]+\varphi\label{eqn:J_tilde minus J 3}\\\nonumber
&=\lim_{T\to\infty}\frac{1}{T}\E\Big[\sum_{t=0}^{T-1}\sum_{r\in\U}\Big(\tilde{\zeta}_r(t)^{\top}\big(Q_{rr}+\hat{K}_r^{\top}R_{rr}\hat{K}_r-P_r(t)\big)\tilde{\zeta}_r(t)\\
&\qquad\qquad\qquad\qquad\qquad\quad+\tilde{\zeta}_r^{\top}(t)\big((A_{sr}+B_{sr}\hat{K}_r)^{\top}P_s(t+1)(A_{sr}+B_{sr}\hat{K}_r)\big)\tilde{\zeta}_r(t)\Big)\Big]+\varphi,\label{eqn:J_tilde minus J 4}
\end{align}
where for each $r\in\U$ in Eq.~\eqref{eqn:J_tilde minus J 4}, we let $s\in\U$ be the unique node such that $r\rightarrow s$. To obtain Eq.~\eqref{eqn:J_tilde minus J 2}, we note from Lemma~\ref{lemma:zeta_s1 and zeta_s2 indept} that $\tilde{x}(t)=\sum_{r\in\U}I_{\V,r}\tilde{\zeta}_r(t)$ for all $t\in\Z_{\ge0}$, where $\E[\tilde{\zeta}_r(t)]=0$ for all $r\in\U$, and $\tilde{\zeta}_{r_1}(t)$ and $\tilde{\zeta}_{r_2}(t)$ are independent for all $r_1,r_2\in\U$ with $r_1\neq r_2$. Moreover, we note from Eq.~\eqref{eqn:u tilde} that $\tilde{u}(t)=\sum_{r\in\U}I_{\V,s}\hat{K}_r\tilde{\zeta}_r(t)$ for all $t\in\Z_{\ge0}$, where $\hat{K}_r$ is given by Eq.~\eqref{eqn:set of DARES K hat}. Combining the above arguments together, and recalling that $\tilde{c}_t=\tilde{x}(t)^{\top}Q\tilde{x}(t)+\tilde{u}(t)^{\top}R\tilde{u}(t)$, we obtain Eq.~\eqref{eqn:J_tilde minus J 2}. To obtain Eq.~\eqref{eqn:J_tilde minus J 3}, we first apply Eq.~\eqref{eqn:dynamics of zeta_tilde} and notice $w(t)\sim\CN(0,\sigma_w^2I)$. Next, we use the facts that the information graph $\CP(\U,\CH)$ defined in \eqref{eqn:def of info graph} is a tree (see Lemma~\ref{lemma:properties of info graph} and Remark~\ref{remark:tree info graph}), and that $\tilde{\zeta}_{r_1}(t)$ and $\tilde{\zeta}_{r_2}(t)$ are independent for all $r_1,r_2\in\U$ with $r_1\neq r_2$ and for all $t\in\Z_{\ge0}$, as we argued above. To obtain Eq.~\eqref{eqn:J_tilde minus J 4}, we leverage again the tree structure of $\CP(\U,\CH)$. 

Now, leveraging the recursion in Eq.~\eqref{eqn:recursive DARES P}, and using similar arguments to those for the proof of \cite[Lemma~12]{fazel2018global}, one can show via \eqref{eqn:J_tilde minus J 4} that 
\begin{align*}
\tilde{J}-J_{\star}&\le\lim_{T\to\infty}\frac{1}{T}\E\Big[\sum_{t=0}^{T-1}\sum_{r\in\U}\Big(\tilde{\zeta}_r(t)^{\top}(\hat{K}_r-K_r)^{\top}\big(R_{rr}+B_{sr}^{\top}P_s(t+1)B_{sr}\big)(\hat{K}_r-K_r)\tilde{\zeta}_r(t)\\
&\qquad\qquad+2\tilde{\zeta}_r(t)^{\top}(\hat{K}_r-K_r)^{\top}\big((R_{rr}+B_{sr}^{\top}P_s(t+1)B_{sr})K_r+B_{sr}^{\top}P_s(t+1)A_{sr}\big)\tilde{\zeta}_r(t)\Big)\Big]+\varphi.
\end{align*}
Recall from Lemma~\ref{lemma:opt solution} that $A_{ss}+B_{ss}K_s$ is stable for any $s\in\U$ that has a self loop. Using similar arguments to those for the proof of \cite[Corollary~4]{lamperski2015optimal}, one can show via Eq,~\eqref{eqn:recursive DARES P} that $P_r(t)\to P_r$ as $T\to\infty$, for all $t\in\{0,\dots,T\}$ and for all $r\in\U$, where $P_r$ is given by Eq.~\eqref{eqn:set of DARES P}. It then follows that
\begin{align}\nonumber
\tilde{J}-J_{\star}&\le\lim_{T\to\infty}\frac{1}{T}\E\Big[\sum_{t=0}^{T-1}\sum_{r\in\U}\Big(\tilde{\zeta}_r(t)^{\top}(\hat{K}_r-K_r)^{\top}\big(R_{rr}+B_{sr}^{\top}P_sB_{sr}\big)(\hat{K}_r-K_r)\tilde{\zeta}_r(t)\\\nonumber
&\qquad\qquad\quad+2\tilde{\zeta}_r(t)^{\top}(\hat{K}_r-K_r)^{\top}\big((R_{rr}+B_{sr}^{\top}P_sB_{sr})K_r+B_{sr}^{\top}P_sA_{sr}\big)\tilde{\zeta}_r(t)\Big)\Big]+\varphi\\\nonumber
&=\lim_{T\to\infty}\frac{1}{T}\E\Big[\sum_{t=0}^{T-1}\sum_{r\in\U}\Big(\tilde{\zeta}_r(t)^{\top}(\hat{K}_r-K_r)^{\top}\big(R_{rr}+B_{sr}^{\top}P_sB_{sr}\big)(\hat{K}_r-K_r)\tilde{\zeta}_r(t)\Big)+\varphi\\
&=\tr\Big(\sum_{r\in\U}(\hat{K}_r-K_r)^{\top}\big(R_{rr}+B_{sr}^{\top}P_sB_{sr}\big)(\hat{K}_r-K_r)\lim_{t\to\infty}\E\Big[\tilde{\zeta}_r(t)\tilde{\zeta}_r(t)^{\top}\Big]\Big)+\varphi,\label{eqn:J_tilde minus J 5}
\end{align}
where the equality follows from Eq.~\eqref{eqn:set of DARES K}, and the second equality follows from the fact that the limit $\lim_{t\to\infty}\E\big[\tilde{\zeta}_r(t)\tilde{\zeta}_r(t)^{\top}\big]$ exists, for all $r\in\U$, as we argued in the proof of Lemma~\ref{lemma:upper bound on zeta_tilde}.

Finally, substituting \eqref{eqn:upper bound on cov zeta_tilde} in Lemma~\ref{lemma:upper bound on zeta_tilde} into the right-hand side of ~\eqref{eqn:J_tilde minus J 5}, we obtain
\begin{align}\nonumber
\tilde{J}-J_{\star}&\le\frac{4p\sigma_w^2\tilde{\Gamma}^{4D_{\max}}\kappa^2}{1-\gamma^2}\tr\Big(\sum_{r\in\U}(K_r-\hat{K}_r)^{\top}(R_{rr}+B_{sr}^{\top}P_sB_{sr})(K_r-\hat{K}_r)\Big)+\varphi\\\nonumber
&\le\frac{4p\sigma_w^2\tilde{\Gamma}^{4D_{\max}}\kappa^2}{1-\gamma^2}(\Gamma^3+\sigma_1(R))\tr\Big(\sum_{r\in\U}(K_r-\hat{K}_r)^{\top}(K_r-\hat{K}_r)\Big)+\varphi\\\nonumber
&\le\frac{4p\sigma_w^2\tilde{\Gamma}^{4D_{\max}}\kappa^2}{1-\gamma^2}(\Gamma^3+\sigma_1(R))\sum_{r\in\U}\min\{m_r,n_r\}\norm{K_r-\hat{K}_r}^2+\varphi\\
&\le\frac{72\kappa^4\sigma_w^2pqn}{(1-\gamma^2)^2}\tilde{\Gamma}^{4D_{\max}+8}(\Gamma^3+\sigma_1(R))(1+\sigma_1(R^{-1}))(20\tilde{\Gamma}^9\sigma_1(R))^{D_{\max}}\varepsilon+\varphi,\label{eqn:upper bound on J_tilde minus J 1}
\end{align}
where the third inequality follows from the fact that $K_r,\hat{K}_r\in\R^{n_r\times m_r}$ for all $r\in\U$, with $n_r\triangleq\sum_{i\in r}n_i$ and $m_r\triangleq\sum_{i\in r}m_i$. To obtain \eqref{eqn:upper bound on J_tilde minus J 1}, we first note that $\varepsilon$ is assumed to satisfy \eqref{eqn:upper bound on epsilon 2} (and thus \eqref{eqn:upper bound on epsilon 1}). Recalling $|\U|=q$, and $n_i\ge m_i$ for all $i\in\V$ as we assumed previously, we then obtain \eqref{eqn:upper bound on J_tilde minus J 1} from Lemmas~\ref{lemma:upper bounds on P_r hat and K_r hat self loop}-\ref{lemma:upper bounds on P_r and K_r general}, where we also use the fact that $\norm{\hat{K}_r-K_r}\le1$ for all $r\in\U$.\hfill\qed

\subsection{Proof of Lemma~\ref{lemma:upper bound on state norms}}
First, considering any $s\in\U$ and any $t\in\Z_{\ge0}$, we have 
\begin{align}\nonumber
\E\Big[\norm{\tilde{\zeta}_s(t)}^2\Big]&=\E\Big[\tilde{\zeta}_s(t)^{\top}\tilde{\zeta}_s(t)\Big]=\E\Big[\tr\Big(\tilde{\zeta}_s(t)\tilde{\zeta}_s(t)^{\top}\Big)\Big]=\tr\Big(\E\Big[\tilde{\zeta}_s(t)\tilde{\zeta}_s(t)^{\top}\Big]\Big)\\\nonumber
&\le n\sigma_1\Big(\E\Big[\tilde{\zeta}_s(t)\tilde{\zeta}_s(t)^{\top}\Big]\Big).
\end{align}
Following similar arguments to those in the proof of Lemma~\ref{lemma:upper bound on zeta_tilde} (particularly Eq.~\eqref{eqn:cov of zeta_s tilde}), one can show that 
\begin{align*}
\E\Big[\tilde{\zeta}_s(t)\tilde{\zeta}_s(t)^{\top}\Big]\preceq \tilde{L}_{ss}^{t}\E\Big[\tilde{\zeta}_s(0)\tilde{\zeta}_s(0)^{\top}\Big](\tilde{L}_{ss}^{\top})^{t}+\sigma_w^2p\tilde{\Gamma}^{4D_{\max}}\sum_{k=0}^{t-1}L_{ss}^k(\tilde{L}_{ss}^{\top})^k,
\end{align*}
where $\tilde{L}_{ss}=A_{ss}+B_{ss}\hat{K}_s$, and $\tilde{\zeta}_s(0)=\sum_{w_i\rightarrow s}I_{s,\{i\}}x_i(0)=0$. Since $\varepsilon$ satisfies \eqref{eqn:upper bound on epsilon 2} and thus \eqref{eqn:upper bound on epsilon 1}, we see from Lemma~\ref{lemma:upper bounds on P_r hat and K_r hat self loop} that $\norm{\tilde{L}_{ss}^k}\le\kappa(\frac{\gamma+1}{2})^k$ for all $k\in\Z_{\ge0}$. It now follows that
\begin{align*}
\E\Big[\tilde{\zeta}_s(t)\tilde{\zeta}_s(t)^{\top}\Big]&\preceq\sigma_w^2p\kappa^2\tilde{\Gamma}^{4D_{\max}}\sum_{k=0}^t\Big(\frac{\gamma+1}{2}\Big)^{2k}I\\
&\preceq\frac{4\sigma_w^2p\kappa^2\tilde{\Gamma}^{4D_{\max}}}{1-\gamma^2}I.
\end{align*}
Combining the above arguments together, we obtain \eqref{eqn:upper bound on norm of zeta_tilde}.

Next, recalling from Lemma~\ref{lemma:zeta_s1 and zeta_s2 indept} that $\tilde{x}(t)=\sum_{s\in\U}I_{\V,s}\tilde{\zeta}_s(t)$ for all $t\in\Z_{\ge0}$, we then have 
\begin{align}\nonumber
\E\Big[\norm{\tilde{x}(t)}^2\Big]&=\E\Big[\big(\sum_{s\in\U}I_{\V,s}\tilde{\zeta}_s(t)\big)^{\top}\big(\sum_{s\in\U}I_{\V,s}\tilde{\zeta}_s(t)\big)\Big]\\\nonumber
&\le\Big(\sum_{s\in\U}\sqrt{\E\Big[\tilde{\zeta}_s(t)^{\top}I_{s,\V}I_{\V,s}\tilde{\zeta}_s(t)\Big]}\Big)^2\\\nonumber
&\le\Big(\sum_{s\in\U}\sqrt{\E\Big[\norm{\tilde{\zeta}_s(t)}^2\Big]}\Big)^2\\\nonumber
&\le\frac{4npq^2\sigma_w^2\tilde{\Gamma}^{4D_{\max}}\kappa^2}{1-\gamma^2},
\end{align}
where the first inequality follows from Lemma~\ref{lemma:CS inequallity}, and the last inequality follows from the fact that $|\U|=q$. This completes the proof of \eqref{eqn:upper bound on norm of x_tilde}.

Finally, we note from Eq.~\eqref{eqn:u tilde} that $\tilde{u}(t)=\sum_{s\in\U}I_{\V,s}\hat{K}_s(t)\tilde{\zeta}_s(t)$. Moreover, since $\varepsilon$ satisfies \eqref{eqn:upper bound on epsilon 2} (and thus \eqref{eqn:upper bound on epsilon 1}), we know from Lemmas~\ref{lemma:upper bounds on P_r hat and K_r hat self loop}-\ref{lemma:upper bounds on P_r and K_r general} that $\norm{\hat{K}_s-K_s}\le1$ for all $s\in\U$, which implies via the definition of $\Gamma$ given in~\eqref{eqn:Gamma} that $\norm{\hat{K}_s}\le\tilde{\Gamma}$ for all $s\in\U$. Now, using similar arguments to those above, we can show that \eqref{eqn:upper bound on norm of u_tilde} holds.\hfill\qed

\subsection{Proof of Lemma~\ref{lemma:u_hat minus u_tilde and x_hat minus x_tilde}}
For notational simplicity in this proof, we denote 
\begin{equation*}
\label{eqn:delta_h and beta}
\delta_h=p(\tilde{\Gamma}+1)^{2D_{\max}-1},\ \beta=\tilde{\Gamma}^{2D_{\max}},
\end{equation*}
\begin{equation*}
\label{eqn:Lambda_1}
\Lambda_1=q\tilde{\Gamma}\bigg(\frac{2\kappa p(\beta+1)}{1-\gamma}+\frac{32\kappa^2p(\tilde{\Gamma}+1)}{(1-\gamma)^2}\bigg)\bigg(1+\frac{\kappa\Gamma}{1-\gamma}\bigg),
\end{equation*}
and
\begin{equation*}
\label{eqn:Lambda_2}
\Lambda_2=\frac{2pq\tilde{\Gamma}\kappa}{1-\gamma}\big((\beta+1)q(\tilde{\Gamma}+1)+\delta_h\big)\zeta_b+\frac{16\kappa^2pq\tilde{\Gamma}(\tilde{\Gamma}+1)}{(1-\gamma)^2}\big(2q(\tilde{\Gamma}+1)+2\big)\zeta_b.
\end{equation*}
We first prove \eqref{eqn:u_hat minue u_tilde}. Based on the above notations, we can show that
\begin{align*}
\Lambda_2&=\Big(\frac{2\kappa pq^2}{1-\gamma}\tilde{\Gamma}(\tilde{\Gamma}+1)(\beta+1)+\frac{32\kappa^2 pq^2}{(1-\gamma)^2}\tilde{\Gamma}(\tilde{\Gamma}+1)^2\Big)\zeta_b+\Big(\frac{2\kappa pq}{1-\gamma}\tilde{\Gamma}\delta_h+\frac{32\kappa^2 pq}{(1-\gamma)^2}\tilde{\Gamma}(\tilde{\Gamma}+1)\Big)\zeta_b\\
&\le\frac{34\kappa^2 pq^2}{(1-\gamma)^2}(\tilde{\Gamma}+1)^2\tilde{\Gamma}^{2D_{\max}+1}\zeta_b+\frac{18\kappa^2 p^2q}{(1-\gamma)^2}(\tilde{\Gamma}+1)^{2D_{\max}+3}\zeta_b\\
&\le\frac{52\kappa^2 p^2q^2}{(1-\gamma)^2}(\tilde{\Gamma}+1)^{2D_{\max}+3}\zeta_b,
\end{align*}
where the first inequality follows from the facts that $\frac{\kappa}{1-\gamma}>1$, and that $\beta+1\le\tilde{\Gamma}^{2D_{\max}+1}$ since $\tilde{\Gamma}=\Gamma+1\ge2$ (see our arguments in the proof of Lemma~\ref{lemma:upper bounds on P_r and K_r general}). We then have
\begin{equation}
\label{eqn:1.1Lambda_2 lower bound}
1.1\Lambda_2\le \frac{58\kappa^2(\tilde{\Gamma}+1)^{2D_{\max}+3}p^2q^2}{(1-\gamma)^2}\zeta_b.\end{equation}
Thus, in order to show that \eqref{eqn:u_hat minue u_tilde} holds for all $t\in\Z_{\ge0}$, it suffices to show that $\E\big[\norm{u(t)-\tilde{u}(t)}^2\big]\le (1.1\Lambda_2\bar{\varepsilon})^2$ holds for all $t\in\Z_{\ge0}$. To this end, we prove via an induction on $t=0,1,\dots$. For any $t\in\Z_{\ge0}$, we recall from Eqs.~\eqref{eqn:control policy} and \eqref{eqn:u tilde} that $\hat{u}_i(t)=\sum_{r\ni i}I_{\{i\},r}\hat{K}_r\hat{\zeta}_r(t)$ and $\tilde{u}_i(t)=\sum_{r\ni i}I_{\{i\},r}\hat{K}_r\tilde{\zeta}_r(t)$, respectively, for all $i\in\V$, where $\hat{\zeta}_r(t)$ and $\tilde{\zeta}_r(t)$ are given by Eqs.~\eqref{eqn:dynamics of zeta hat} and \eqref{eqn:dynamics of zeta_tilde}, respectively, and $\hat{K}_r$ is given by Eq.~\eqref{eqn:set of DARES K hat}, for all $r\in\U$. As we argued before, in Eqs.~\eqref{eqn:dynamics of zeta hat} and \eqref{eqn:dynamics of zeta_tilde} we have $\hat{\zeta}_r(0)=\tilde{\zeta}_r(0)=\sum_{w_i\rightarrow r}I_{r,\{i\}}x_i(0)$ for all $r\in\U$. Hence, we have $\hat{u}(0)=\tilde{u}(0)$, which implies that \eqref{eqn:u_hat minue u_tilde} holds for $t=0$, completing the proof of the base step of the induction. 

For the induction step, suppose $\E\big[\norm{\hat{u}(k)-\tilde{u}(k)}^2\big]\le(1.1\Lambda_2\bar{\varepsilon})^2$ holds for all $k\in\{0,\dots,t\}$. 
Now, considering any $k\in\{0,\dots,t\}$, we can unroll the expressions of $\hat{x}(k)$ and $\tilde{x}(k)$ given by Eqs.~\eqref{eqn:state x hat} and \eqref{eqn:state x tilde}, respectively, and obtain
\begin{equation*}
\label{eqn:exp for x_hat(k)}
\hat{x}(k)=A^{k}\hat{x}(0)+\sum_{k^{\prime}=0}^{k-1}A^{k-k^{\prime}-1}(B\hat{u}(k^{\prime})+w(k^{\prime})),
\end{equation*}
and 
\begin{equation*}
\label{eqn:exp for x_tilde(k)}
\tilde{x}(k)=A^{k}\tilde{x}(0)+\sum_{k^{\prime}=0}^{k-1}A^{k-k^{\prime}-1}(B\tilde{u}(k^{\prime})+w(k^{\prime})),
\end{equation*}
where we note that $\hat{x}(0)=\tilde{x}(0)=x(0)$. It then follows  that 
\begin{align}\nonumber
\sqrt{\E\Big[\norm{\hat{x}(k)-\tilde{x}(k)}^2\Big]}&\le\sum_{k^{\prime}=0}^{k-1}\sqrt{\E\Big[\norm{A^{k-k^{\prime}-1}B(\hat{u}(k^{\prime})-\tilde{u}(k^{\prime}))}^2\Big]}\\\nonumber
&\le\sum_{k^{\prime}=0}^{k-1}\sqrt{\E\Big[\norm{A^{k-k^{\prime}-1}}^2\norm{B}^2\norm{\hat{u}(k^{\prime})-\tilde{u}(k^{\prime})}^2\Big]}\\\nonumber
&\le\norm{B}\sum_{k^{\prime}=0}^{k-1}\norm{A^{k-k^{\prime}-1}}\sqrt{\E\Big[\norm{\hat{u}(k^{\prime})-\tilde{u}(k^{\prime})}^2\Big]}\\
&\le\Gamma1.1\Lambda_2\bar{\varepsilon}\sum_{k^{\prime}=0}^{k-1}\norm{A^{k-k^{\prime}-1}}\le1.1\Gamma\Lambda_2\bar{\varepsilon}\frac{\kappa}{1-\gamma},\label{eqn:x_hat(k) minus x_tilde(k) 1}
\end{align}
where the first inequality follows from Lemma~\ref{lemma:CS inequallity}. To obtain the first inequality in \eqref{eqn:x_hat(k) minus x_tilde(k) 1}, we use the induction hypothesis. To obtain the second inequality in \eqref{eqn:x_hat(k) minus x_tilde(k) 1}, we use the fact that $\norm{A^{k^{\prime}}}\le\kappa\gamma^{k^{\prime}}$ (with $0<\gamma<1$), for all $k^{\prime}\in\Z_{\ge0}$, from Assumption~\ref{ass:stable A}. Recalling from our arguments in Section~\ref{sec:control design} (particularly, Eq.~\eqref{eqn:est w_i(t)}), one can show that 
\begin{equation*}
\label{eqn:exp for w_hat(k)}
\hat{w}(k) = \hat{x}(k+1)-\hat{A}\hat{x}(k)-\hat{B}\hat{u}(k),
\end{equation*}
where $\hat{w}(k)=\begin{bmatrix}\hat{w}_1(k)^{\top} & \cdots &\hat{w}_p(k)^{\top}\end{bmatrix}^{\top}$ is an estimate of $w(k)$ in Eq.~\eqref{eqn:overall system}. From Eq.~\eqref{eqn:state x hat}, we see that 
\begin{equation*}
\label{eqn:exp for w(k)}
w(k)=\hat{x}(k+1)-A\hat{x}(k)-B\hat{u}(k).  
\end{equation*}
It follows that
\begin{align}\nonumber
\norm{\hat{w}(k)-w(k)}&\le\norm{\hat{A}-A}\norm{\hat{x}(k)}+\norm{\hat{B}-B}\norm{\hat{u}(k)}\\\nonumber
&\le(\norm{\hat{x}(k)}+\norm{\hat{u}(k)})\bar{\varepsilon}.
\end{align}
Recall from Lemma~\ref{lemma:upper bound on state norms} that $\E\big[\norm{\tilde{x}(k)}^2\big]\le q^2\zeta_b^2$ and $\E\big[\norm{\tilde{u}(k)}^2\big]\le q^2\tilde{\Gamma}^2\zeta_b^2$, for all $k\in\Z_{\ge0}$. We then obtain
\begin{align}\nonumber
\E\Big[\norm{\hat{x}(k)}^2\Big]&=\E\Big[\norm{\hat{x}(k)-\tilde{x}[k]+\tilde{x}[k]}^2\Big]\\\nonumber
&\le\Big(\sqrt{\E\Big[\norm{\hat{x}(k)-\tilde{x}(k)}^2\Big]}+\sqrt{\E\Big[\norm{\tilde{x}[k]}^2\Big]}\Big)^2\\\nonumber
&\le\Big(\frac{1.1\Gamma\Lambda_2\kappa}{1-\gamma}\bar{\varepsilon}+q\zeta_b\Big)^2,
\end{align}
where the first inequality follows again from Lemma~\ref{lemma:CS inequallity}, and the second inequality uses \eqref{eqn:x_hat(k) minus x_tilde(k) 1}. Similarly, we have
\begin{align}\nonumber
\E\Big[\norm{\hat{u}(k)}^2\Big]&=\E\Big[\norm{\hat{u}(k)-\tilde{u}[k]+\tilde{u}[k]}^2\Big]\\\nonumber
&\le\Big(\sqrt{\E\Big[\norm{\hat{u}(k)-\tilde{u}(k)}^2\Big]}+\sqrt{\E\Big[\norm{\tilde{u}[k]}^2\Big]}\Big)^2\\\nonumber
&\le(1.1\Lambda_2\bar{\varepsilon}+q\tilde{\Gamma}\zeta_b)^2.
\end{align}
It then follows that 
\begin{align}\nonumber
\E\Big[\norm{\hat{w}(k)-w(k)}^2\Big]&=\E\Big[\norm{(\hat{A}-A)\hat{x}(k)+(\hat{B}-B)\hat{u}(k)}^2\Big]\\\nonumber
&\le\Big(\sqrt{\E\Big[\norm{(\hat{A}-A)\hat{x}(k)}^2\Big]}+\sqrt{\E\Big[\norm{(\hat{B}-B)\hat{u}(k)}^2\Big]}\Big)^2\\\nonumber
&\le\Big(\sqrt{\E\Big[\norm{\hat{x}(k)}^2\Big]}+\sqrt{\E\Big[\norm{\hat{u}(k)}^2\Big]}\Big)^2\bar{\varepsilon}^2\\\nonumber
&\le\big(q(\tilde{\Gamma}+1)\zeta_b+\frac{1.1\Gamma\Lambda_2\kappa}{1-\gamma}\bar{\varepsilon}+1.1\Lambda_2\big)^2\bar{\varepsilon}^2,
\end{align}
where the first inequality again follows from Lemma~\ref{lemma:CS inequallity}, and the second inequality follows from $\norm{\hat{A}-A}\le\bar{\varepsilon}$ and $\norm{\hat{B}-B}\le\bar{\varepsilon}$. Denoting
\begin{equation}
\label{eqn:delta_w}
\delta_w=q(\tilde{\Gamma}+1)\zeta_b+\frac{1.1\Gamma\Lambda_2\kappa}{1-\gamma}\bar{\varepsilon}+1.1\Lambda_2\bar{\varepsilon},
\end{equation}
we have
\begin{equation}
\label{eqn:w_hat(k) minus w(k) 1}
\E\Big[\norm{\hat{w}(k)-w(k)}^2\Big]\le\delta_w^2\bar{\varepsilon}^2\quad \forall k\in\{0,\dots,t\}.
\end{equation}
Moreover, note that
\begin{align}\nonumber
\E\Big[\norm{w(k)}^2\Big]&=\E\Big[\tr(w(k)w(k)^{\top})\Big]=\tr\Big(\E\Big[w(k)w(k)^{\top}\Big]\Big)\\
&=n\sigma^2_w\le\zeta_b^2\quad\forall k\in\Z_{\ge0}.\label{eqn:norm of w(k)}
\end{align}

To proceed, let us consider any $s\in\U$ that has a self loop. Recalling the arguments in the proof of Lemmas~\ref{lemma:upper bound on zeta_tilde}, we can rewrite Eq.~\eqref{eqn:dynamics of zeta_tilde} as 
\begin{equation}
\label{eqn:zeta_tilde s}
\tilde{\zeta}_s(t+1)=(A_{ss}+B_{ss}\hat{K}_s)\tilde{\zeta}_s(t)+\eta_s(t),
\end{equation}
with
\begin{equation}
\label{eqn:eta_s(t)}
\eta_s(t)=\sum_{v\in\CL_s}H(v,s)\sum_{w_j\rightarrow v}I_{v,\{j\}}w_j(t-l_{vs}),
\end{equation}
where $\CL_s=\{v\in\CL:v\rightsquigarrow s\}$ is the set of leaf nodes in $\CP(\U,\CH)$ that can reach $s$, $l_{vs}$ is the length of the (unique) directed path from node $v$ to node $s$ in $\CP(\U,\CH)$ with $l_{vs}=0$ if $v=s$, and
\begin{equation*}
\label{eqn:H(v,s) 1}
H(v,s) = (A_{sr_1}+B_{sr_1}\hat{K}_{r_1})\cdots(A_{r_{l_{vs}-1}v}+B_{r_{l_{vs}-1}v}\hat{K}_v),
\end{equation*}
with $H(v,s)=I$ if $v=s$, where $v,r_{l_{vs}-1},\dots,r_1,s$ are the nodes along the directed path from $v$ to $s$ in $\CP(\U,\CH)$. We also recall from the arguments in the proof of Lemma~\ref{lemma:upper bound on zeta_tilde} that $\norm{H(v,s)}\le\beta$ for all $v\in\CL_s$. We then see from \eqref{eqn:cov of eta_s} in the proof of Lemma~\ref{lemma:upper bound on zeta_tilde} and the definition of $\zeta_b$ in \eqref{eqn:aux parameters} that 
\begin{align}\nonumber
\E\Big[\norm{\eta_s(k)}^2\Big]&=\E\Big[\tr(\eta_s(k)\eta_s(k)^{\top})\Big]=\tr\Big(\E\Big[\eta_s(k)\eta_s(k)^{\top}\Big]\Big)\\
&\le\sigma_w^2np\beta^2\le\zeta_b^2 \quad\forall k\in\Z_{\ge0}.\label{eqn:upper bound on eta_s 1}
\end{align}
Similarly, one can rewrite Eq.~\eqref{eqn:dynamics of zeta hat} as 
\begin{equation}
\label{eqn:zeta_hat s}
\hat{\zeta}_s(t+1)=(\hat{A}_{ss}+\hat{B}_{ss}\hat{K}_s)\hat{\zeta}_s(t)+\hat{\eta}_s(t),
\end{equation}
where 
\begin{equation}
\label{eqn:eta_s(t) hat}
\hat{\eta}_s(t)=\sum_{v\in\CL_s}\hat{H}(v,s)\sum_{w_j\rightarrow v}I_{v,\{j\}}\hat{w}_j(t-l_{vs}),
\end{equation}
where
\begin{equation*}
\label{eqn:H_hat(v,s) 1}
\hat{H}(v,s) = (\hat{A}_{sr_1}+\hat{B}_{sr_1}\hat{K}_{r_1})\cdots(\hat{A}_{r_{l_{vs}-1}v}+\hat{B}_{r_{l_{vs}-1}v}\hat{K}_v),
\end{equation*}
with $\hat{H}(v,s)=I$ if $v=s$. Note that for any $v\in\CL_s$ and for any $w_j\rightarrow v$ in Eqs.~\eqref{eqn:eta_s(t)} and \eqref{eqn:eta_s(t) hat}, we set $\hat{w}_j(k-l_{vs})=w_j(k-l_{vs})=0$ if $k<l_{vs}$. One can check that $\bar{\varepsilon}$ satisfies \eqref{eqn:upper bound on epsilon 1} and \eqref{eqn:upper bound on epsilon 2}. We then have from Lemmas~\ref{lemma:upper bounds on P_r hat and K_r hat self loop}-\ref{lemma:upper bounds on P_r and K_r general} that $\norm{\hat{K}_r}\le\Gamma+1=\tilde{\Gamma}$ for all $r\in\U$. It follows that $\norm{\hat{A}_{sr}+\hat{B}_{sr}\hat{K}_r}\le(\tilde{\Gamma}+1)^2$ for all $r\in\U$ with $r\neq s$. Noting from the construction of $\CP(\U,\CH)$ in \eqref{eqn:def of info graph} that $l_{vs}\le D_{\max}\le p$ for all $v\in\CL_s$, one can now show that
\begin{equation}
\label{eqn:H_hat minus H}
\norm{H(v,s)-\hat{H}(v,s)}\le \delta_h\bar{\varepsilon}.
\end{equation}
which also implies that 
\begin{equation}
\label{eqn:upper bound on norm of H_hat}
\norm{\hat{H}(v,s)}\le\delta_h\bar{\varepsilon}+\beta,
\end{equation}
for all $v\in\CL_s$. For any $k\in\{0,\dots,t\}$, we then have from the above arguments that 
\begin{align}\nonumber
&\sqrt{\E\Big[\norm{\eta_s(k)-\hat{\eta}_s(k)}^2\Big]}\\\nonumber
=&\bigg(\E\bigg[\Big\|\sum_{v\in\CL_s}\sum_{w_j\rightarrow v}\big(H(v,s)I_{v,\{j\}}w_j(k-l_{vs})-\hat{H}(v,s)I_{v,\{j\}}\hat{w}_j(k-l_{vs})\big)\Big\|^2\bigg]\bigg)^{\frac{1}{2}}\\\nonumber
\le&\sum_{v\in\CL_s}\sum_{w_j\rightarrow v}\Big(\sqrt{\E\Big[\big\|\big(H(v,s)-\hat{H}(v,s)\big)w_j(k-l_{vs})\big\|^2\Big]}+\sqrt{\E\Big[\big\|\hat{H}(v,s)\big(w_j(k-l_{vs})-\hat{w}_j(k-l_{vs})\big)\big\|^2\Big]}\Big)\\
\le& p\big(\delta_h\zeta_b\bar{\varepsilon}+(\delta_h\bar{\varepsilon}+\beta)\delta_w\bar{\varepsilon}\big),\label{eqn:eta_s minus eta_s hat}
\end{align}
where the first inequality follows from Lemma~\ref{lemma:CS inequallity}. To obtain \eqref{eqn:eta_s minus eta_s hat}, we first note \eqref{eqn:w_hat(k) minus w(k) 1}-\eqref{eqn:norm of w(k)} and \eqref{eqn:H_hat minus H}-\eqref{eqn:upper bound on norm of H_hat}. We then use the fact that $|\CL_s|\le p$ from the definition of the information graph $\CP(\U,\CH)$ given by \eqref{eqn:def of info graph}, and the fact that for any $v\in\U$ with $s_j(0)=v$, $w_j$ is the only noise term such that $w_j\rightarrow v$ (see Footnote~2). From \eqref{eqn:upper bound on eta_s 1} and \eqref{eqn:eta_s minus eta_s hat}, we also obtain
\begin{align}\nonumber
\sqrt{\E\Big[\norm{\hat{\eta}_s(k)}^2\Big]}&=\sqrt{\E\Big[\norm{\hat{\eta}_s(k)-\eta_s(k)+\eta_s(k)}^2\Big]}\\\nonumber
&\le\sqrt{\E\Big[\norm{\hat{\eta}_s(k)-\eta_s(k)}^2\Big]}+\sqrt{\E\Big[\norm{\eta_s(k)}^2\Big]}\\
&\le p\big(\delta_h\zeta_b\bar{\varepsilon}+(\delta_h\bar{\varepsilon}+\beta)\delta_w\bar{\varepsilon}\big)+\zeta_b,\label{eqn:upper bound on eta_s hat}
\end{align}
where the first inequality follows from Lemma~\ref{lemma:CS inequallity}.

Now, let us denote $\tilde{L}_{ss}=A_{ss}+B_{ss}\hat{K}_s$ and $\hat{L}_{ss}=\hat{A}_{ss}+\hat{B}_{ss}\hat{K}_s$. Recalling that $\hat{\zeta}_s(0)=\tilde{\zeta}_s(0)=\sum_{w_i\rightarrow s}I_{s,\{i\}}x_i(0)$, where $x(0)=0$ as we assumed before, one can unroll Eqs.~\eqref{eqn:zeta_tilde s} and \eqref{eqn:zeta_hat s}, and show that 
\begin{equation}
\label{eqn:zeta_hat s minus zeta_tilde s}
\hat{\zeta}_s(t+1)-\tilde{\zeta}_s(t+1)=\sum_{k=0}^t\big(\hat{L}_{ss}^{t-k}\hat{\eta}_s(k)-\tilde{L}_{ss}^{t-k}\tilde{\eta}_s(k)\big).
\end{equation}
Since $\norm{\hat{A}-A}\le\bar{\varepsilon}$ and $\norm{\hat{B}-B}\le\bar{\varepsilon}$, where $\bar{\varepsilon}$ satisfies \eqref{eqn:upper bound on epsilon 1}, as we argued above, we have from Lemma~\ref{lemma:upper bounds on P_r hat and K_r hat self loop} that 
\begin{equation}
\label{eqn:upper bound on L_ss^k}
\norm{\tilde{L}_{ss}^k}\le\kappa(\frac{\gamma+1}{2})^k\quad\forall k\in\Z_{\ge0},
\end{equation}
where $\kappa\in\R_{\ge1}$ and $\gamma\in\R$, with $0<\gamma<1$. Moreover, since $\norm{\hat{L}_{ss}-\tilde{L}_{ss}}=\norm{\hat{A}_{ss}-A_{ss}+\hat{K}_s(\hat{B}_{ss}-B_{ss})}\le(\tilde{\Gamma}+1)\bar{\varepsilon}$, we have from Lemma~\ref{lemma:matrix power perturbation bound} that 
\begin{align}\nonumber
\norm{\hat{L}_{ss}^k-\tilde{L}_{ss}^k}&\le k\kappa^2\big(\kappa(\tilde{\Gamma}+1)\bar{\varepsilon}+\frac{\gamma+1}{2}\big)^{k-1}(\tilde{\Gamma}+1)\bar{\varepsilon}\\
&\le k\kappa^2(\frac{\gamma+3}{4})^{k-1}(\tilde{\Gamma}+1)\bar{\varepsilon}\quad\forall k\in\Z_{\ge0},\label{eqn:L_hat^k minus L_tilde^k}
\end{align}
where \eqref{eqn:L_hat^k minus L_tilde^k} follows from the choice of $\bar{\varepsilon}$ in \eqref{eqn:aux parameters}. Now, considering any term in the summation on the right-hand side of \eqref{eqn:zeta_hat s minus zeta_tilde s}, we have 
\begin{align}\nonumber
&\sqrt{\E\Big[\norm{\hat{L}_{ss}^{t-k}\hat{\eta}_s(k)-\tilde{L}_{ss}^{t-k}\tilde{\eta}_s(k)}^2\Big]}\\\nonumber
\le&\sqrt{\E\Big[\norm{(\hat{L}_{ss}^{t-k}-\tilde{L}_{ss}^{t-k})\hat{\eta}_s(k)}^2\Big]}+\sqrt{\E\Big[\norm{\tilde{L}_{ss}^{t-k}(\hat{\eta}_s(k)-\eta_s(k))}^2\Big]}\\\nonumber
\le&(t-k)\kappa^2(\frac{\gamma+3}{4})^{t-k-1}(\tilde{\Gamma}+1)\bar{\varepsilon}\big(p(\delta_w\delta_h\bar{\varepsilon}+\delta_h\zeta_b+\delta_w\beta)\bar{\varepsilon}+\zeta_b\big)+\kappa(\frac{\gamma+1}{2})^{t-k}p(\delta_w\delta_h\bar{\varepsilon}+\delta_h\zeta_b+\delta_w\beta)\bar{\varepsilon}\\
\le&(t-k)\kappa^2(\frac{\gamma+3}{4})^{t-k-1}(\tilde{\Gamma}+1)p\big(2\delta_w+2\zeta_b\big)\bar{\varepsilon}+\kappa(\frac{\gamma+1}{2})^{t-k}p\big((\beta+1)\delta_w+\delta_h\zeta_b\big)\bar{\varepsilon},\label{eqn:single term in summation 2}
\end{align}
where the first inequality follows from Lemma~\ref{lemma:CS inequallity}, and the second inequality uses the upper bounds provided in~\eqref{eqn:eta_s minus eta_s hat}-\eqref{eqn:upper bound on eta_s hat} and \eqref{eqn:upper bound on L_ss^k}-\eqref{eqn:L_hat^k minus L_tilde^k}. Moreover, one can show that $\bar{\varepsilon}$ defined in \eqref{eqn:aux parameters} satisfies that $\delta_h\bar{\varepsilon}\le1$ and $\beta\bar{\varepsilon}\le1$, which, via algebraic manipulations, yield \eqref{eqn:single term in summation 2}. We then see from \eqref{eqn:zeta_hat s minus zeta_tilde s} that
\begin{align}\nonumber
&\sqrt{\E\Big[\norm{\hat{\zeta}_s(t+1)-\tilde{\zeta}_s(t+1)}^2\Big]}\\\nonumber
\le&\sum_{k=0}^t\sqrt{\E\Big[\norm{\hat{L}_{ss}^{t-k}\hat{\eta}_s(k)-\tilde{L}_{ss}^{t-k}\tilde{\eta}_s(k)}^2\Big]}\\\nonumber
\le& \bigg(\kappa p\big((\beta+1)\delta_w+\delta_h\zeta_b\big)\bar{\varepsilon}\sum_{k=0}^t(\frac{\gamma+1}{2})^{t-k}\bigg)+\kappa^2(\tilde{\Gamma}+1)p\big(2\delta_w+2\zeta_b\big)\bar{\varepsilon}\sum_{k=0}^t(t-k)(\frac{\gamma+3}{4})^{t-k-1}\\
\le&\frac{2\kappa p}{1-\gamma}\big((\beta+1)\delta_w+\delta_h\zeta_b\big)\bar{\varepsilon}+\frac{16\kappa^2(\tilde{\Gamma}+1)p}{(1-\gamma)^2}\big(2\delta_w+2\zeta_b\big)\bar{\varepsilon},\label{eqn:zeta_hat s minus zeta_tilde s 1}
\end{align}
where the first inequality follows from Lemma~\ref{lemma:CS inequallity}, and \eqref{eqn:zeta_hat s minus zeta_tilde s 1} follows from standard formulas for series. Now, substituting Eq.~\eqref{eqn:delta_w} into the right-hand side of \eqref{eqn:zeta_hat s minus zeta_tilde s 1}, one can show that
\begin{align}
\sqrt{\E\Big[\norm{\hat{\zeta}_s(t+1)-\tilde{\zeta}_s(t+1)}^2\Big]}\le\frac{1}{q\tilde{\Gamma}}\big(\Lambda_1(1.1\Lambda_2\bar{\varepsilon})+\Lambda_2\big)\bar{\varepsilon}, \label{eqn:zeta_hat s minus zeta_tilde s 2}
\end{align}
where we note that $\Lambda_1>0$ and $\Lambda_2>0$ by their definitions.

Next, considering any $s\in\U$ that does not have a self loop, we have from the arguments in the proof of Lemma~\ref{lemma:upper bound on zeta_tilde} that Eq.~\eqref{eqn:dynamics of zeta_tilde} can be rewritten as $\tilde{\zeta}_s(t+1)=\eta_s(t)$, where $\eta_s(t)$ is defined in Eq.~\eqref{eqn:eta_s(t)}. Similarly, Eq.~\eqref{eqn:dynamics of zeta hat} can be rewritten as $\hat{\zeta}_s(t+1)=\hat{\eta}_s(t)$, where $\hat{\eta}_s(t)$ is defined in Eq.~\eqref{eqn:eta_s(t) hat}. Using similar arguments to those above, one can then show that \eqref{eqn:zeta_hat s minus zeta_tilde s 2} also holds.

Further recalling Eqs.~\eqref{eqn:control policy} and \eqref{eqn:u tilde}, we know that $\hat{u}(t+1)=\sum_{s\in\U}I_{\V,s}\hat{K}_s\hat{\zeta}(t+1)$ and $\tilde{u}(t+1)=\sum_{s\in\U}I_{\V,s}\hat{K}_s\tilde{\zeta}(t+1)$, which imply that 
\begin{align}\nonumber
\sqrt{\E\Big[\norm{\hat{u}(t+1)-\tilde{u}(t+1)}^2\Big]}&=\sqrt{\E\Big[\big\|\sum_{s\in\U}I_{\V,s}\hat{K}_s(\hat{\zeta}_s(t+1)-\tilde{\zeta}_s(t+1))\big\|^2\Big]}\\\nonumber
&\le\sum_{s\in\U}\sqrt{\E\Big[\norm{I_{\V,s}\hat{K}_s(\hat{\zeta}_s(t+1)-\tilde{\zeta}_s(t+1))}^2\Big]}\\\nonumber
&\le\tilde{\Gamma}\sum_{s\in\U}\sqrt{\E\Big[\norm{\hat{\zeta}_s(t+1)-\tilde{\zeta}_s(t+1)}^2\Big]}\\\nonumber
&\le\big(\Lambda_1(1.1\Lambda_2\bar{\varepsilon})+\Lambda_2\big)\bar{\varepsilon},
\end{align}
where the first inequality follows from Lemma~\ref{lemma:CS inequallity}, the second inequality follows from $\norm{\hat{K}_s}\le\tilde{\Gamma}$ for all $s\in\U$, as we argued above, and the last inequality follows from \eqref{eqn:zeta_hat s minus zeta_tilde s 2} and the fact that $|\U|=q$. Moreover, we can show that
\begin{align*}
\Lambda_1&\le q\tilde{\Gamma}\frac{34\kappa^2 p}{(1-\gamma)^2}(\tilde{\Gamma}+1)(\beta+1)\frac{\tilde{\Gamma}\kappa}{1-\gamma}\\
&\le\frac{34\kappa^2pq}{(1-\gamma)^3}(\tilde{\Gamma}+1)\tilde{\Gamma}^{2D_{\max}+3},
\end{align*}
where the first inequality follows from the fact that $\frac{\kappa}{1-\gamma}>1$, and the second inequality follows from the fact that $\beta+1\le\tilde{\Gamma}^{2D_{\max}+1}$ as we argued above. One can then show that $\bar{\varepsilon}$ given in \eqref{eqn:aux parameters} satisfies that $0<\bar{\varepsilon}\le\frac{1}{11\Lambda_1}$. Thus, we obtain the following:
\begin{align}\nonumber
&1.1\Lambda_1\bar{\varepsilon}+1\le 1.1\\\nonumber
\Leftrightarrow\ &1.1\Lambda_1\Lambda_2\bar{\varepsilon}+\Lambda_2\le1.1\Lambda_2\\\nonumber
\Leftrightarrow\ &\Lambda_1(1.1\Lambda_2\bar{\varepsilon})\bar{\varepsilon}+\Lambda_2\bar{\varepsilon}\le1.1\Lambda_2\bar{\varepsilon}.
\end{align}
It follows that
\begin{equation*}
\sqrt{\E\Big[\norm{\hat{u}(t+1)-\tilde{u}(t+1)}^2\Big]}\le 1.1\Lambda_2\bar{\varepsilon},
\end{equation*}
which completes the induction step, i.e., we have shown that $\sqrt{\E\big[\norm{\hat{u}(k)-\tilde{u}(k)}^2\big]}\le1.1\Lambda_2\bar{\varepsilon}$ holds for all $k\in\{0,\dots,t+1\}$.

Next, using similar arguments to those for \eqref{eqn:x_hat(k) minus x_tilde(k) 1}, we have $\sqrt{\E\big[\norm{\hat{x}(t)-\tilde{x}(t)}^2\big]}\le\frac{1.1\Gamma\Lambda_2\kappa}{1-\gamma}\bar{\varepsilon}$ for all $t\in\Z_{\ge0}$. It then follows from \eqref{eqn:1.1Lambda_2 lower bound} that \eqref{eqn:x_hat minus x_tilde} holds for all $t\in\Z_{\ge0}$.\hfill\qed

\subsection{Proof of Proposition~\ref{prop:upper bound on J_hat minue J_tilde}}
For notational simplicity in this proof, we denote 
\begin{equation}
\label{eqn:Lambda}
\Lambda=\frac{58\kappa^2(\tilde{\Gamma}+1)^{2D_{\max}+3}p^2q^2}{(1-\gamma)^2}.
\end{equation}
For all $t\in\Z_{\ge0}$, we then see from Lemma~\ref{lemma:u_hat minus u_tilde and x_hat minus x_tilde} that 
\begin{equation*}
\label{eqn:u_hat(k) minus u_tilde(k) simple}
\E\Big[\norm{\hat{u}(t)-\tilde{u}(t)}^2\Big]\le(\Lambda\zeta_b\bar{\varepsilon})^2,
\end{equation*}
and
\begin{equation*}
\label{eqn:x_hat(k) minus x_tilde(k) simple}
\E\Big[\norm{\hat{x}(t)-\tilde{x}(t)}^2\Big]\le\Big(\frac{\kappa\Gamma}{1-\gamma}\Lambda\zeta_b\bar{\varepsilon}\Big)^2,
\end{equation*}
where $\hat{u}(k)$ (resp., $\tilde{u}(k)$) is given by Eq.~\eqref{eqn:control policy} (resp., Eq.~\eqref{eqn:u tilde}), $\hat{x}(k)$ (resp., $\tilde{x}(k)$) is given by Eq.~\eqref{eqn:state x hat} (resp., Eq.~\eqref{eqn:state x tilde}), and $\zeta_b$ is defined in Eq.~\eqref{eqn:aux parameters}. Similarly, we see from Corollary~\ref{coro:upper bound on x_hat and u_hat} that
\begin{equation*}
\label{eqn:upper bound on x_hat(k) general}
\E\Big[\norm{\hat{x}(t)}^2\Big]\le\Big(\frac{\kappa\Gamma}{1-\gamma}\Lambda\zeta_b\bar{\varepsilon}+q\zeta_b\Big)^2,
\end{equation*}
and 
\begin{equation*}
\label{eqn:upper bound on u_hat(k) general}
\E\Big[\norm{\hat{u}(t)}^2\Big]\le(\Lambda\zeta_b\bar{\varepsilon}+q\tilde{\Gamma}\zeta_b)^2,
\end{equation*}
for all $t\in\Z_{\ge0}$. 

To proceed, we have the following:
\begin{align}\nonumber
\hat{J}-\tilde{J}&=\limsup_{T\to\infty}\E\Big[\frac{1}{T}\sum_{t=0}^{T-1}\big(\hat{x}(t)^{\top}Q\hat{x}(t)+\hat{u}(t)^{\top}R\hat{u}(t)\big)\Big]-\lim_{T\to\infty}\E\Big[\frac{1}{T}\sum_{t=0}^{T-1}\big(\tilde{x}(t)^{\top}Q\tilde{x}(t)+\tilde{u}(t)^{\top}R\tilde{u}(t)\big)\Big]\\
&=\limsup_{T\to\infty}\E\Big[\frac{1}{T}\sum_{t=0}^{T-1}\big(\hat{x}(t)^{\top}Q\hat{x}(t)-\tilde{x}(t)^{\top}Q\tilde{x}(t)+\hat{u}(t)^{\top}R\hat{u}(t)-\tilde{u}(t)^{\top}R\tilde{u}(t)\big)\Big].\label{eqn:J_hat minus J_tilde}
\end{align}
Now, considering any term in the summation on the right-hand side of Eq.~\eqref{eqn:J_hat minus J_tilde}, and dropping the dependency on $t$ for notational simplicity, we have the following:
\begin{align}\nonumber
&\E\Big[\hat{x}^{\top}Q\hat{x}-\tilde{x}^{\top}Q\tilde{x}\Big]\\\nonumber
=&\E\Big[\hat{x}^{\top}Q(\hat{x}-\tilde{x})+(\hat{x}-\tilde{x})^{\top}Q\tilde{x}\Big]\\\nonumber
\le&\E\Big[\norm{Q\hat{x}}\norm{\hat{x}-\tilde{x}}\Big]+\E\Big[\norm{\hat{x}-\tilde{x}}\norm{Q\tilde{x}}\Big]\\\nonumber
\le&\sqrt{\E\Big[\norm{Q\hat{x}}^2\Big]\E\Big[\norm{\hat{x}-\tilde{x}}^2\Big]}+\sqrt{\E\Big[\norm{\hat{x}-\tilde{x}}^2\Big]\E\Big[\norm{Q\tilde{x}}^2\Big]}\\\nonumber
\le&\sigma_1(Q)\Big(\frac{\kappa\Gamma\Lambda\zeta_b}{1-\gamma}\bar{\varepsilon}+q\zeta_b\Big)\frac{\kappa\Gamma\Lambda\zeta_b}{1-\gamma}\bar{\varepsilon}+\sigma_1(Q)\frac{\kappa\Gamma\Lambda\zeta_b}{1-\gamma}\bar{\varepsilon}q\zeta_b\\
=&\sigma_1(Q)\Big(\frac{\kappa\Gamma\Lambda\zeta_b}{1-\gamma}\bar{\varepsilon}+2q\zeta_b\Big)\frac{\kappa\Gamma\Lambda\zeta_b}{1-\gamma}\bar{\varepsilon},\label{eqn:single term in J_hat minus J_tilde x}
\end{align}
where the first two inequalities follow from the Cauchy-Schwartz inequality, and the third inequality follows from the upper bounds on $\E\big[\norm{\hat{x}}^2\big]$, $\E\big[\norm{\hat{x}-\tilde{x}}^2\big]$, and $\E\big[\norm{\tilde{x}}^2\big]$ given above and in Lemma~\ref{lemma:upper bound on state norms}. Similarly, we have
\begin{align}\nonumber
&\E\Big[\hat{u}^{\top}R\hat{u}-\tilde{u}^{\top}R\tilde{u}\Big]\\\nonumber
\le&\sqrt{\E\Big[\norm{R\hat{u}}^2\Big]\E\Big[\norm{\hat{u}-\tilde{u}}^2\Big]}+\sqrt{\E\Big[\norm{\hat{u}-\tilde{u}}^2\Big]\E\Big[\norm{R\tilde{u}}^2\Big]}\\\nonumber
\le&\sigma_1(R)(\Lambda\zeta_b\bar{\varepsilon}+q\tilde{\Gamma}\zeta_b)\Lambda\zeta_b\bar{\varepsilon}+\sigma_1(R)\Lambda\zeta_b\bar{\varepsilon}q\tilde{\Gamma}\zeta_b\\
=&\sigma_1(R)(\Lambda\zeta_b\bar{\varepsilon}+2q\tilde{\Gamma}\zeta_b)\Lambda\zeta_b\bar{\varepsilon},\label{eqn:single term in J_hat minus J_tilde u}
\end{align}
where the second inequality follows from the upper bounds on $\E\big[\norm{\hat{u}}^2\big]$, $\E\big[\norm{\hat{u}-\tilde{u}}^2\big]$, and $\E\big[\norm{\tilde{u}}^2\big]$ given above and in Lemma~\ref{lemma:upper bound on state norms}. Combining \eqref{eqn:single term in J_hat minus J_tilde x} and \eqref{eqn:single term in J_hat minus J_tilde u} together, we obtain from Eq.~\eqref{eqn:J_hat minus J_tilde} that
\begin{align}\nonumber
\hat{J}-\tilde{J}&\le\sigma_1(Q)\Big(\frac{\kappa\Gamma\Lambda\zeta_b}{1-\gamma}\bar{\varepsilon}+2q\zeta_b\Big)\frac{\kappa\Gamma\Lambda\zeta_b}{1-\gamma}\bar{\varepsilon}+\sigma_1(R)(\Lambda\zeta_b\bar{\varepsilon}+2q\tilde{\Gamma}\zeta_b)\Lambda\zeta_b\bar{\varepsilon}\\\nonumber
&\le\Big(\frac{\kappa\tilde{\Gamma}\zeta_b}{1-\gamma}\Big)^2(\Lambda^2\bar{\varepsilon}+2q\Lambda)(\sigma_1(Q)+\sigma_1(R))\bar{\varepsilon}\\
&\le\Big(\frac{\kappa\tilde{\Gamma}\zeta_b}{1-\gamma}\Big)^23\Lambda pq(\sigma_1(Q)+\sigma_1(R))\bar{\varepsilon},\label{eqn:J_hat minus J_tilde 1}
\end{align}
where the second inequality follows from the fact that $\frac{\kappa\tilde{\Gamma}}{1-\gamma}\ge1$. To obtain \eqref{eqn:J_hat minus J_tilde 1}, one can show that $\Lambda^2\bar{\varepsilon}\le\Lambda pq$. Finally substituting the expressions for $\zeta_b$ and $\Lambda$ given in \eqref{eqn:aux parameters} and \eqref{eqn:Lambda}, respectively, we obtain from \eqref{eqn:J_hat minus J_tilde 1} that \eqref{eqn:upper bound on J_hat minus J_tilde} holds.\hfill\qed

\section{Auxiliary Lemmas}\label{sec:technical lemmas}
\begin{lemma}\label{lemma:upper bound on gaussian w}
\cite[Lemma~34]{cassel2020logarithmic} Let $w(t)\in\R^n$ be a Gaussian random vector with distribution $\CN(0,\sigma_w^2 I)$, for all $t\in\{0,\dots,N-1\}$, where $\sigma_w\in\R_{\ge0}$. Then for any $N\ge2$ and for any $\delta_w>0$, the following holds with probability at least $1-\delta_w$:
\begin{equation*}
\max_{0\le t\le N-1}\norm{w(t)}\le\sigma_w\sqrt{5n\log\frac{N}{\delta_w}}.
\end{equation*}
\end{lemma}

\begin{lemma}\label{lemma:lower bound on sum over z_k}
\cite[Lemma~36]{cassel2020logarithmic}
Let $\{z(t)\}_{t\ge0}$ be a sequence of random vectors that is adapted to a filtration $\{\F_t\}_{t\ge0}$, where $z(t)\in\R^d$ for all $t\in\Z_{\ge0}$. Suppose $z(t)$ is conditionally Gaussian on $\F_{t-1}$ with $\E[z(t)z(t)^{\top}|\F_{t-1}]\ge\sigma_z^2I$, for all $t\in\Z_{\ge1}$, where $\sigma_z\in\R_{>0}$. Then, for any $\delta_z\in\R_{>0}$ and for any $t\ge200d\log\frac{12}{\delta_z}$, the following holds with probability at least $1-\delta_z$:
\begin{equation*}
\sum_{k=0}^{t-1}z(k)z(k)^{\top}\succeq\frac{(t-1)\sigma_z^2}{40}I.
\end{equation*}
\end{lemma}

\begin{lemma}
\label{lemma:CS inequallity}
Let $X_1,\dots,X_t$ be a sequence of random vectors, where $t\in\Z_{\ge1}$. Then,
\begin{equation*}
\E\Big[\big(\sum_{k=1}^t X_k\big)^{\top}\big(\sum_{k=1}^t X_k\big)\Big]\le\Big(\sum_{k=1}^t\sqrt{\E\Big[\norm{X_k}^2\Big]}\Big)^2.
\end{equation*}
\end{lemma}
\begin{proof}
We have the following:
\begin{align}\nonumber
\E\Big[\big(\sum_{k=1}^t X_k\big)^{\top}\big(\sum_{k=1}^t X_k\big)\Big]&=\sum_{k=1}^t\E\Big[\norm{X_k}^2\Big]+2\sum_{k_1=1}^t\sum_{\substack{k_2=1\\k_2\neq k_1}}^t\E\Big[X_{k_1}^{\top}X_{k_2}\Big]\\\nonumber
&\le\sum_{k=1}^t\E\Big[\norm{X_k}^2\Big]+2\sum_{k_1=1}^t\sum_{\substack{k_2=1\\k_2\neq k_1}}^t\E\Big[\norm{X_{k_1}}\norm{X_{k_2}}\Big]\\\nonumber
&\le\sum_{k=1}^t\E\Big[\norm{X_k}^2\Big]+2\sum_{k_1=1}^t\sum_{\substack{k_2=1\\k_2\neq k_1}}^t\sqrt{\E\Big[\norm{X_{k_1}}^2\Big]}\sqrt{\E\Big[\norm{X_k}^2\Big]}\\\nonumber
&=\Big(\sum_{k=1}^t\sqrt{\E\Big[\norm{X_k}^2\Big]}\Big)^2,
\end{align}
where the first and second inequalities follow from the Cauchy-Schwarz inequality.
\end{proof}

\begin{lemma}\label{lemma:matrix power perturbation bound}
\cite[Lemma~5]{mania2019certainty}
Consider any matrix $M\in\R^{n\times n}$ and any matrix $\Delta\in\R^{n\times n}$. Let $\kappa_M\in\R_{\ge1}$ and $\gamma_M\in\R_{>0}$ be such that $\gamma_M>\rho(M)$, and $\norm{M^k}\le\kappa_M\gamma_M^k$ for all $k\in\Z_{\ge0}$. Then, for all $k\in\Z_{\ge0}$,
\begin{equation*}
\norm{(M+\Delta)^k-M^k}\le k\kappa_M^2(\kappa_M\norm{\Delta}+\gamma_M)^{k-1}\norm{\Delta}.
\end{equation*}
\end{lemma}

\end{document}